\documentclass{amsart}

\usepackage{amsmath}
\usepackage{amsthm}
\usepackage{amsfonts}
\usepackage{amssymb}
\usepackage{graphics}

\newcommand\R{{\mathbb{R}}}
\newcommand\C{{\mathbb{C}}}
\newcommand\Z{{\mathbf{Z}}}

\newcommand\Q{{\mathbf{Q}}}

\newcommand\N{{\mathbf{N}}}
\newcommand\n{{\mathbf{n}}}

\newcommand\E{{\mathbf{E}}}
\newcommand\F{{\mathbf{F}}}

\renewcommand\Re{{\operatorname{Re}}}

\newcommand\st{{\operatorname{st}}}
\newcommand\eps{{\varepsilon}}

\newcommand\sgn{\operatorname{sgn}}
\newcommand\Frob{\operatorname{Frob}}
\newcommand\ultra{{}^*}

\theoremstyle{plain}
  \newtheorem{theorem}{Theorem}

  \newtheorem{proposition}[theorem]{Proposition}
  
  \newtheorem{lemma}[theorem]{Lemma}
  \newtheorem{corollary}[theorem]{Corollary}

\theoremstyle{definition}
  \newtheorem{definition}[theorem]{Definition}
  \newtheorem{example}[theorem]{Example}
  \newtheorem{remark}[theorem]{Remark}

\begin{document}

\title[Expanding polynomials and a regularity lemma]{Expanding polynomials over finite fields of large characteristic, and a regularity lemma for definable sets}

\author{Terence Tao}
\address{UCLA Department of Mathematics, Los Angeles, CA 90095-1555.}
\email{tao@math.ucla.edu}


\begin{abstract}  Let $P: \F \times \F \to \F$ be a polynomial of bounded degree over a finite field $\F$ of large characteristic.  In this paper we establish the following dichotomy: either $P$ is a \emph{moderate asymmetric expander} in the sense that $|P(A,B)| \gg |\F|$ whenever $A, B \subset \F$ are such that $|A| |B| \geq C |\F|^{2-1/8}$ for a sufficiently large $C$, or else $P$ takes the form $P(x,y) = Q(F(x)+G(y))$ or $P(x,y) = Q(F(x) G(y))$ for some polynomials $Q,F,G$.  This is a reasonably satisfactory classification of polynomials of two variables that moderately expand (either symmetrically or asymmetrically).  We obtain a similar classification for weak expansion (in which one has $|P(A,A)| \gg |A|^{1/2} |\F|^{1/2}$ whenever $|A| \geq C |\F|^{1-1/16}$), and a partially satisfactory classification for almost strong asymmetric expansion (in which $|P(A,B)| = (1-O(|\F|^{-c})) |\F|$ when $|A|, |B| \geq |\F|^{1-c}$ for some small absolute constant $c>0$).

The main new tool used to establish these results is an \emph{algebraic regularity lemma} that describes the structure of dense graphs generated by definable subsets over finite fields of large characteristic.  This lemma strengthens the Sz\'emeredi regularity lemma in the algebraic case, in that while the latter lemma decomposes a graph into a bounded number of components, most of which are $\eps$-regular for some small but fixed $\epsilon$, the latter lemma ensures that all of the components are $O(|\F|^{-1/4})$-regular.  This lemma, which may be of independent interest, relies on some basic facts about the \'etale fundamental group of an algebraic variety.
\end{abstract}

\maketitle

\section{Introduction}

\subsection{Expanding polynomials}

Let $\F$ be a finite field, let $k \geq 1$ be an integer, and let $P: \F^k \to \F$ be a polynomial of $k$ variables defined over $\F$.  We will be interested in the regime when the order $|\F|$ of $\F$ is large\footnote{In fact, the new results of this paper will be restricted to the regime in which the \emph{characteristic} of $\F$ is large, and not just the order.}, but $k$ and the degree of $P$ remains bounded; one could formalise this by working\footnote{This is analogous to how the concept of an expander graph does not, strictly speaking, apply in any non-trivial sense to a single (standard) graph, but should instead be applied to a \emph{sequence} of such graphs, or to a single \emph{nonstandard} graph.} with a \emph{sequence} $\F_\n$ of fields whose order is going to infinity, and a sequence $P_\n: \F_\n^k \to \F_\n$ of polynomials of uniformly bounded degree on each of these fields, where $k$ is independent of $\n$.  But in the discussion that follows we will suppress the dependence on the sequence parameter $\n$ to simplify the exposition.  (Later on, we will use the formalism of nonstandard analysis to make this suppression of $\n$ more precise.)  Given $k$ subsets $A_1,\ldots,A_k$ of $\F$, we may form the set
$$ P(A_1,\ldots,A_k) := \{ P(a_1,\ldots,a_k) \mid a_1 \in A_1,\ldots,a_k \in A_k\}.$$
One of the main objectives of this paper is to study the expansion properties of $P$, which informally refers to the phenomenon that for ``typical'' polynomials $P$ and ``non-trivial'' $A_1,\ldots,A_k$, the set $P(A_1,\ldots,A_k)$ tends to be significantly larger than any of the $A_1,\ldots,A_k$.  We will focus in particular on the following five concepts in increasing order of strength, essentially following the notation from \cite{hls}:

\begin{enumerate}
\item We say that $P$ is a \emph{weak expander} if there are absolute constants $c,C>0$ such that $|P(A,\ldots,A)| \geq C^{-1} |A|^{1-c} |\F|^c$ whenever $A \subset \F$ and $|A| \geq C |\F|^{1-c}$.
\item We say that $P$ is a \emph{moderate expander} if there are absolute constants $c,C>0$ such that $|P(A,\ldots,A)| \geq C^{-1} |\F|$ whenever $A \subset \F$ and $|A| \geq C |\F|^{1-c}$.
\item We say that $P$ is a \emph{almost strong expander} if there are absolute constants $c,C>0$ such that $|P(A,\ldots,A)| \geq |\F| - C |\F|^{1-c}$ whenever $A \subset \F$ and $|A|,\ldots,|A| \geq C |\F|^{1-c}$.
\item We say that $P$ is a \emph{strong expander} if there are absolute constants $c,C>0$ such that $|P(A,\ldots,A)| \geq |\F| - C$ whenever $A \subset \F$ and $|A|,\ldots,|A| \geq C |\F|^{1-c}$.
\item We say that $P$ is a \emph{very strong expander} if there are absolute constants $c,C>0$ such that $P(A,\ldots,A) = \F$ whenever $A \subset \F$ and $|A|,\ldots,|A| \geq C |\F|^{1-c}$.
\end{enumerate}

As noted previously, these notions are trivial in the setting of a fixed field $\F$ and polynomial $P$, but acquire non-trivial meaning when these objects are allowed to depend on some parameter $\n$.  It is certainly also of interest to understand expansion when the sets $A_1,\ldots,A_k$ are small (as opposed to having cardinality at least $C |\F|^{1-c}$), but we have nothing new to say about this case and will not discuss it further here, and refer the interested reader to \cite{hls} and \cite{bt} for a survey of the situation.

In this paper, we will also consider the \emph{asymmetric} case when the sets involved are distinct:

\begin{enumerate}
\item We say that $P$ is a \emph{weak asymmetric expander} if there are absolute constants $c,C>0$ such that $|P(A_1,\ldots,A_k)| \geq C^{-1} \min(|A_1|,\ldots,|A_k|)^{1-c} |\F|^c$ whenever $|A_1|,\ldots,|A_k| \geq C |\F|^{1-c}$.
\item We say that $P$ is a \emph{moderate asymmetric expander} if there are absolute constants $c,C>0$ such that $|P(A_1,\ldots,A_k)| \geq C^{-1} |\F|$ whenever $|A_1|,\ldots,|A_k| \geq C |\F|^{1-c}$.
\item We say that $P$ is a \emph{almost strong asymmetric expander} if there are absolute constants $c,C>0$ such that $|P(A_1,\ldots,A_k)| \geq |\F| - C |\F|^{1-c}$ whenever $|A_1|,\ldots,|A_k| \geq C |\F|^{1-c}$.
\item We say that $P$ is a \emph{strong asymmetric expander} if there are absolute constants $c,C>0$ such that $|P(A_1,\ldots,A_k)| \geq |\F| - C$ whenever $|A_1|, \ldots, |A_k| \geq C |\F|^{1-c}$.
\item We say that $P$ is a \emph{very strong asymmetric expander} if there are absolute constants $c,C>0$ such that $P(A_1,\ldots,A_k) = \F$ whenever $|A_1|, \ldots, |A_k| \geq C |\F|^{1-c}$.
\end{enumerate}

Clearly, any of the asymmetric expansion properties implies the symmetric counterpart; for instance, moderate asymmetric expansion implies moderate expansion.  

When $k=1$, $P$ cannot be an expander in any of the above senses, thanks to the trivial inequality $|P(A)| \leq |A|$.  For $k \geq 2$, there are some obvious examples of non-expanders.  For instance, when $k=2$, the polynomial $P(x,y) :=x+y$ is not an expander in any of the above senses (in the limit when $|\F|$ goes to infinity), as can be seen by setting $A_1=A_2$ equal to an arithmetic progression.  In a similar vein, $P(x,y) :=xy$ is not an expander as one can set $A_1=A_2$ equal to a geometric progression.  More generally, if $P$ takes the additive form
\begin{equation}\label{p12}
 P(x_1,x_2) = Q(F_1(x_1)+F_2(x_2))
\end{equation}
or the multiplicative form
\begin{equation}\label{q12}
 P(x_1,x_2) = Q(F_1(x_1) F_2(x_2))
\end{equation}
for some polyomials $Q,F_1,F_2: \F \to \F$ of bounded degree, then $P$ will not be an expander in any of the asymmetric senses, as can be seen by taking $A_i = F^{-1}_i(E_i)$ for $i=1,2$, where $E_1, E_2$ are randomly chosen arithmetic (resp. geometric) progressions of fixed length $L$ of equal spacing (resp. ratio).  By setting instead $A := F^{-1}_1(E_1) \cap F^{-1}_2(E_2)$, we see from the first moment method that we can find length $L$ progressions $E_1,E_2$ of equal spacing with $|A| \geq L^2/|\F|$ (resp. $|A| \geq L^2/(|\F|-1))$ in the additive (resp. multiplicative) case; taking $L$ close to $|\F|$, we conclude that such polynomials cannot be moderate expanders (although this argument is not strong enough rule out weak expansion, unless $F_1=F_2$).  A construction in \cite{gyar} also shows that no polynomial of two variables can be a strong expander.

On the other hand, by using estimates related to the sum-product phenomenon, there are several results in the literature establishing various sorts of expansion for certain classes of polynomials.  We will only give a sample of the known results here (focusing exclusively on the regime of large subsets of $\F$), and refer the reader to \cite{hls} and \cite{bt} for a more comphensive survey of results.  Solymosi \cite{soli} used graph-theoretic methods to establish weak expansion for polynomials of the form $P(x_1,x_2) = f(x_1)+x_2$ when $f$ was a nonlinear polynomial of bounded degree; his results also show weak asymmetric expansion for polynomials of the form $P(x_1,x_2,x_3) = f(x_1)+x_2+x_3$.  These results were generalised in \cite{hls} (by a Fourier-analytic method), establishing for instance weak expansion for $P(x_1,x_2) = f(x_1)+g(x_2)$ for non-constant polynomials $f,g$ whose degrees are distinct and less than the characteristic of $\F$, with a similar result for polynomials of the form $P(x_1,x_2) = f(x_1) g(x_2)$.  In \cite{shk}, Shkredov established very strong expansion for the polynomial $P(x_1,x_2,x_3) = x_1^2+x_1x_2+x_3$, and moderate expansion for $P(x_1,x_2)	= x_1(x_1+x_2)$, while in \cite{garaev} weak expansion for $P(x_1,x_2) := x_1(x_2+1)$ in fields of prime order was established.  As a consequence of their results on the finite field distinct distances problem, Iosevich and Rudnev \cite{ir} established the strong expansion of polynomials of the form $P(x_1,\ldots,x_d,y_1,\ldots,y_d) := \sum_{i=1}^d (x_i-y_i)^2$ for any $d \geq 2$, and in a similar spirit Vu \cite{vu} established\footnote{In fact, the result in \cite{vu} yields a stronger lower bound on expressions such as $|A+A| + |f(A,A)|$.}  the moderate expansion of polyomials of the form $P(x_1,x_2,y_1,y_2) = f(x_1-y_1,x_2-y_2)$ for any symmetric polynomial $f$ of bounded degree which was \emph{non-degenerate} in the sense that $f$ is not of the form $f(x_1,x_2) = Q(ax_1 + bx_2)$ for some polynomial $Q$ and constants $a,b$.   In \cite{heg}, moderate expansion for polynomials of the form $P(x_1,x_2)= f(x_1) + x_1^k g(x_2)$ was established when $f(x_1)$ is affinely independent of $x_1^k$, improving upon earlier work of Bourgain \cite{borg}. In \cite{bt}, it was shown that any polynomial $P(x_1,x_2)$ that is not of the form $F_1(x_1) + F_2(x_2)$ or $F_1(x_1)F_2(x_2)$, is monic in each of the two variables $x_1,x_2$, and is \emph{non-composite} in that it is not of the form $P = Q \circ R$ for some polynomials functions $Q: \overline{\F} \to \overline{\F}$, $R: \overline{\F}^2 \to \overline{\F}$ over the algebraic completion of $\F$ with $Q$ non-linear, then $P$ is a weak asymmetric expander.

We can now present our first set of new results regarding expansion, in the context of polynomials of two variables in a field of large characteristic.  We first give the formulation that pertains to moderate asymmetric expansion.

\begin{theorem}[Moderate asymmetric expansion]\label{expand-thm-modas} For any degree $d$, there exists a constant $C$ such that the following statement holds.  Let $\F$ be a finite field of characteristic at least $C$, and let $P: \F \times \F \to \F$ be a polynomial of degree at most $d$.  Then at least one of the following statements hold:
\begin{itemize}
\item[(i)] (Additive structure)  One has
\begin{equation}\label{padd}
 P(x_1,x_2) = Q(F_1(x_1)+F_2(x_2))
\end{equation}
(as a polynomial identity in the indeterminates $x_1,x_2$) for some polynomials $Q, F_1, F_2: \F \to \F$.  
\item[(ii)] (Multiplicative structure)  One has
\begin{equation}\label{pmult}
P(x_1,x_2) = Q(F_1(x_1) F_2(x_2))
\end{equation}
for some polynomials $Q, F_1, F_2: \F \to \F$.
\item[(iii)] (Moderate asymmetric expansion)  One has
$$ |P(A_1,A_2)| \geq C^{-1} |\F|$$
whenever $A_1, A_2$ are subsets of $\F$ with $|A_1| |A_2| \geq C |\F|^{2-1/8}$.
\end{itemize}
\end{theorem}

The degree of $Q,F_1,F_2$ is not specified in (i), (ii), but it is easy to see that one can restrict to the case when $Q, F_1, F_2$ have degree at most $d$ with no loss of generality, since (except in degenerate cases when one of $Q,F_1,F_2$ is constant) it is not possible to have either (i) or (ii) hold if $Q$, $F_1$, or $F_2$ has degree greater than $d$.   The exponent $1/8$ appearing in the above theorem is an artefact arising from the number of times we were forced to apply the Cauchy-Schwarz inequality in our arguments, and we do not believe it to be optimal.  When option (iii) occurs, we are also able to obtain additional bounds on the set $\{ (a_1,a_2,a_3) \in A_1\times A_2 \times A_3: P(a_1,a_2) = a_3 \}$ for various sets $A_1,A_2,A_3$; see Remark \ref{three} below.

Thus, we see that in the large characteristic case, the only polynomials in two variables that are not moderate asymmetric expanders are the polynomials given by the examples \eqref{p12}, \eqref{q12}, which as discussed previously are not weak asymptotic expanders or moderate expanders.  In particular, this shows that there is no distinction between moderate asymmetric expansion, moderate expansion, and weak asymmetric expansion, at least for polynomials of two variables in the large characteristic case.  This result partially addresses a conjecture of Bukh and Tsimerman \cite[\S 9]{bt} and of Vu \cite[Problem 4]{vu}, at least in the case of large characteristic and reasonably dense sets $A$, and it seems likely that the methods here could be used to make further progress on these conjectures.  We remark that some analogous results, in which the finite field $\F$ was replaced by the real field $\R$, the complex field $\C$, or the rationals $\Q$ were obtained in \cite{elekes-ronyai}, \cite{elekes-szabo}, and \cite{solymosi} respectively, using very different methods (related to the Szemer\'edi-Trotter theorem) from those used here.

Theorem \ref{expand-thm} will be established at the end of Section \ref{algo}.  By combining the above results with the Fourier-analytic arguments in \cite{hls}, we can obtain a similar criterion for weak expansion:

\begin{theorem}[Weak expansion]\label{expand-thm-weak} For any degree $d$, there exists a constant $C$ such that the following statement holds.  Let $\F$ be a finite field of characteristic at least $C$, and let $P: \F \times \F \to \F$ be a polynomial of degree at most $d$.  Then at least one of the following statements hold:
\begin{itemize}
\item[(i)] (Additive structure)  One has
$$ P(x_1,x_2) = Q(aF(x_1)+bF(x_2))$$
for some polynomials $Q, F: \F \to \F$, and some elements $a,b \in \F$.
\item[(ii)] (Multiplicative structure)  One has
$$ P(x_1,x_2) = Q(F(x_1)^a F(x_2)^b)$$
for some polynomials $Q, F: \F \to \F$, and some natural numbers $a,b$ with $a,b \leq C$.
\item[(iii)] (Weak expansion)  One has
$$ |P(A,A)| \geq C^{-1} |\F|^{1/2} |A|^{1/2}$$
whenever $A \subset \F$ with $|A| \geq C |\F|^{1-1/16}$.
\end{itemize}
\end{theorem}

Again, this is a reasonably good classification of the polynomials which weakly expand, except that in the case (i) of additive structure, some further information on the ratio $b/a$ of the two elements should be obtained (this ratio should be ``low complexity'' in some sense). We will not pursue this issue here; it boils down to the expansion properties of $A+\alpha A$ for various values of $\alpha$, an issue studied in \cite{laba}, \cite{bukh}, \cite{cill} in the case when $A$ lies in the integers $\Z$ rather than in $\F$ (see also \cite{plagne} for some initial results in the finite field setting).  Theorem \ref{expand-thm-weak} will be established in Section \ref{weak-sec}.

We now turn to the analogue of the above results for almost strong asymmetric expansion, where our results are unfortunately somewhat less satisfactory.  The situation here is necessarily more complicated, as can be seen by the following simple observation: if a polynomial $P$ is an almost strong asymmetric expander, then its square $P^2$ is automatically\footnote{This observation technically answers the question posed at the end of \cite[\S 1]{hls} as to whether there are moderate expanders which are not strong expanders; for instance one can take the square $(x_1^2+x_1x_2+x_3)^2$ of the strong expander of Shkredov \cite{shk}.  However, this is something of a ``cheat'', and the interesting question remains of whether there exists a \emph{non-composite} polynomial $P$ which is a moderate expander but not a strong expander.  In view of Theorem \ref{expand-thm-ass} below, this basically reduces (in the large characteristic case, at least) to the task of locating a non-composite polynomial that obeys an algebraic constraint \eqref{alg-con} without having additive or multiplicative structure.} a moderate asymmetric expander, but not an almost strong asymmetric expander, because $P^2$ is clearly restricted to the quadratic residues.  More generally, if $P$ obeys a polynomial identity of the form
\begin{equation}\label{pgh}
 P( f(x_1), g(x_2) ) = h( Q(x_1,x_2) )
\end{equation}
for some polynomials $f,g,h: \F \to \F$ and $Q: \F \times \F \to \F$ with $h$ nonlinear and $f,g$ non-constant, then it is likely that $P$ will not be an almost strong asymmetric expander, because $P(f(\F),g(\F)) \subset h(\F)$ and $h(\F)$ is not, in general, equal\footnote{One can however construct some examples of nonlinear polynomials $h$ that are bijective on a certain finite field, for instance $x \mapsto x^3$ is bijective on a field $\F_p$ of prime order $p$ whenever $p-1$ is coprime to $3$.} to all of $\F$ and is instead usually just a dense subset of $\F$.  An example of such a polynomial identity\footnote{Admittedly, this is also an example of a polynomial $P$ that obeys the multiplicative structure \eqref{pmult}.  We were unable to either exhibit or rule out an example of a polynomial $P$ obeying \eqref{pgh} but not \eqref{pmult}, but it seems of interest to resolve this question.} is
\begin{equation}\label{px1}
 P( x_1^n, x_2^n ) = P(x_1, x_2)^n
\end{equation}
when $P$ is a monomial of the form $P(x_1,x_2) = x_1^a x_2^b$ and $a,b,n$ are arbitrary natural numbers, which shows that $P$ maps $n^{\operatorname{th}}$ powers to $n^{\operatorname{th}}$ powers.

Our next main result is an attempt to assert that this is the only additional obstruction to almost strong expansion, beyond the obstructions already identified in Theorem \ref{expand-thm}.  Unfortunately, due to limitations in our arguments,  we will be forced to generalise the above example, in which the polynomials $f,g,h,Q$ are defined over the algebraic closure $\overline{\F}$ of $\F$, and for which the domains of $f,g,Q$ are constructed using affine algebraic curves (possibly of positive genus) instead of the affine line.  More precisely, we have:

\begin{theorem}[Almost strong asymmetric expansion]\label{expand-thm-ass} For any degree $d$, there exists a constant $C$ such that the following statement holds.  Let $\F$ be a finite field of characteristic at least $C$, and let $P: \F \times \F \to \F$ be a polynomial of degree at most $d$.  Then at least one of the following statements hold:
\begin{itemize}
\item[(i)] (Additive structure)  One has
$$ P(x_1,x_2) = Q(F_1(x_1)+F_2(x_2))$$
for some polynomials $Q, F_1, F_2$.
\item[(ii)] (Multiplicative structure)  One has
$$ P(x_1,x_2) = Q(F_1(x_1) F_2(x_2))$$
for some polynomials $Q, F_1, F_2$.
\item[(iii)] (Algebraic constraint) There exist irreducible affine curves $V \subset \overline{\F}^m, W \subset \overline{\F}^n$ of complexity\footnote{We will review algebraic geometry notation such as ``irreducible'', ``affine'', ``curve'', and ``complexity'' in the next section.} at most $C$ and definable over an extension of $\F$ of degree at most $C$, as well as polynomial maps $f: \overline{\F}^n \to \overline{\F}$, $g: \overline{\F}^m \to \overline{\F}$, $Q: \overline{\F}^n \times \overline{\F}^m \to \overline{\F}$, $h: \overline{\F} \to \overline{\F}$ of degree at most $C$, and whose coefficients lie in an extension of $\F$ of degree at most $C$, such that the restrictions of $f,g$ to $V,W$ respectively are non-constant, and $h$ has degree at least two, and one has the identity
\begin{equation}\label{alg-con}
 P( f(x_1), g(x_2) ) = h( Q(x_1,x_2) )
\end{equation}
for all $x_1 \in V$ and $x_2 \in W$.
\item[(iv)] (Almost strong asymmetric expansion)  One has
$$ |\F \backslash P(A_1,A_2)| \leq C |\F| \left(\frac{|A_1| |A_2|}{|\F|^{2-1/8}}\right)^{-1/2}$$
whenever $A_1, A_2$ are non-empty subsets of $\F$.
\end{itemize}
\end{theorem}

We establish this result in Section \ref{second-sec}.
While this theorem in principle gives a purely algebraic description of the polynomials $P$ that are not almost strong asymmetric expanders, it is not fully satisfactory, due to the excessively complicated form of (iii).  It would be of interest to simplify this constraint\footnote{For instance, it is conceivable that one could eliminate the cases when the curves $V,W$ have positive genus, so that they could be replaced by the affine line $\overline{\F}$, thus reducing the constraint \eqref{alg-con} back to the simpler constraint \eqref{pgh}.  Also, it might be possible (perhaps by utilising some Galois theory) to reduce to the case where the curves $V,W$ and maps $f,g,Q,h$ are defined over $\F$ rather than over $\overline{\F}$.  Finally, it may be possible to get some more effective bounds on the degrees of $V,W,f,g,Q,h$ in terms of the degree of $P$, although the example \eqref{px1} indicates that one may have to make some ``minimality'' assumptions on these objects before an effective degree bound can be obtained.  We were unable to achieve any of these goals, but believe that they are all worthwhile to pursue.  Some variant of Ritt's theory of decomposition into prime polynomials (see e.g. \cite{zieve}) may be relevant for this purpose.}, in order to make this theorem more useful in applications (and in particular, to allow one to exhibit explicit examples of strong expander polymomials in two variables).

It is likely that one can iterate the above results to obtain some classification of various types of expanding polynomials in three or more variables, but we will not pursue this question here.

\subsection{An algebraic regularity lemma}\label{algreg-sec}

The main new tool that we introduce to establish the above results is an \emph{algebraic regularity lemma} which improves upon the Szemer\'edi regularity lemma \cite{szemeredi-reg} in the case of dense graphs that are \emph{definable} in the language of fields over a field of large characteristic; this lemma seems to be of independent interest.  To describe this new lemma, let us first give a formulation of the usual regularity lemma:

\begin{lemma}[Szemer\'edi regularity lemma]\label{szreg}\cite{szemeredi-reg}  If $\eps > 0$, then there exists $C = C_\eps> 0$ such that the following statements hold: whenever $V,W$ are non-empty finite sets and $E \subset V \times W$, then there exists partitions $V = V_1 \cup \ldots \cup V_a$, $W = W_1 \cup \ldots \cup W_b$ into non-empty sets, and a set $I \subset \{1,\ldots,a\} \times \{1,\ldots,b\}$ of exceptional pairs, with the following properties:
\begin{itemize}
\item[(i)] (Low complexity) $a,b \leq C$.
\item[(ii)] (Few exceptions) $\sum_{(i,j) \in I} |V_i| |W_j| \leq \eps |V| |W|$.
\item[(iii)] ($\eps$-regularity)  For all $(i,j) \in \{1,\ldots,a\} \times \{1,\ldots,b\} \backslash I$, and all $A \subset V_i, B \subset W_j$, one has
$$ \left| |E \cap (A \times B)| - d_{ij} |A| |B|\right| \leq \eps |V_i| |W_j|$$
where $d_{ij} := \frac{|E \cap (V_i \times W_j)|}{|V_i| |W_j|}$.
\end{itemize}
\end{lemma}

The dependence of $C_\eps$ on $\eps$ is notoriously poor (tower exponential in nature); see \cite{gowers-sz}.

Now we restrict attention to definable sets.  If $\F$ is a field, and $n \geq 0$ is a natural number, a \emph{definable subset} of $\F^n$ is defined to be any set of the form
\begin{equation}\label{fan}
 \{ (x_1,\ldots,x_n) \in \F^n \mid p(x_1,\ldots,x_n) \hbox{ is true} \}
\end{equation}
where $p()$ is any formula involving $n$ variables $x_1,\ldots,x_n$ and a finite number of additional constants $c_1,\ldots,c_m \in \F$ and bound variables $y_1,\ldots,y_l$, as well as the ring operations\footnote{One could also, if one wished, also include the inversion operation $()^{-1}$ (after handling somehow the fact that $0^{-1}$ is undefined), but of course any formula involving this operation can be replaced with an equivalent (albeit slightly longer) formula involving the operations $+,-,\times$ and some additional variables and existential quantifiers.} $+, \times$, parentheses $(,)$, the equality sign $=$, the logical connectives $\neg, \wedge, \vee, \implies$, and the quantifiers $\forall, \exists$ (where the quantification is understood to be over the field $\F$). Thus, for instance, the $\F$-points
$$ V(\F) = \{ (x_1,\ldots,x_n) \in \F^n \mid P_1(x_1,\ldots,x_n) = \ldots = P_m(x_1,\ldots,x_n) = 0 \}$$
of an algebraic variety $V$ defined over $\F$, where $P_1,\ldots,P_m: \F^n \to \F$ are polynomials with coefficients in $\F$, form a definable set.  A bit more generally, the $\F$-points of any \emph{constructible set} in $\F^n$ (i.e. a boolean combination of a finite number of algebraic varieties in $\F^n$) is a definable set.  As is well known, in the case when $\F$ is algebraically closed, the constructible sets are the only definable sets, thanks to the presence of quantifier elimination (or Hilbert's nullstellensatz) in this setting; but for non-algebraically closed fields, other definable sets also exist.  For instance, the set
\begin{equation}\label{qf}
 Q := \{ x \in \F \mid \exists y \in \F: x = y^2 \}
\end{equation}
of quadratic residues in $\F$ is definable, but is usually not constructible.

Now we specialise to the case where $\F$ is a finite field.  Strictly speaking, the theory of definable sets on such fields is trivial, since every subset of $\F^n$ is finite and thus automatically definable. However, one can recover a more interesting theory by limiting the \emph{complexity} of the definable sets being considered.  Let us say that a subset $E$ of $\F^n$ is a \emph{definable set of complexity at most $M$} if the ambient dimension $n$ is at most $M$, and $E$ can be expressed in the form \eqref{fan} for some formula $\phi$ of length at most $M$, where we consider all variables, constants, operations, parentheses, equality symbols, logical operations, and quantifiers to have unit length.  One is then interested in the regime where $M$ stays bounded, but the cardinality or characteristic of $\F$ goes to infinity.

We can now give the algebraic regularity lemma.

\begin{lemma}[Algebraic regularity lemma]\label{algreg}  If $M > 0$, then there exists $C = C_M> 0$ such that the following statements hold: whenever $\F$ is a finite field of characteristic at least $C$, $V,W$ are non-empty definable sets over $\F$ of complexity at most $M$, and $E \subset V \times W$ is another definable set over $\F$ of complexity at most $M$, then there exists partitions $V = V_1 \cup \ldots \cup V_a$, $W = W_1 \cup \ldots \cup W_b$, with the following properties:
\begin{itemize}
\item[(i)] (Largeness) For all $i \in \{1,\ldots,a\}$ and $j \in \{1,\ldots,b\}$, one has $|V_i| \geq |V|/C$ and $|W_j| \geq |W|/C$.  In particular, $a,b \leq C$.
\item[(ii)] (Bounded complexity)  The sets $V_1,\ldots,V_a,W_1,\ldots,W_b$ are definable over $\F$ with complexity at most $C$.
\item[(iii)] ($|\F|^{-1/4}$-regularity)  For all $(i,j) \in \{1,\ldots,a\} \times \{1,\ldots,b\}$, and all $A \subset V_i, B \subset W_j$, one has
\begin{equation}\label{elab}
 \left| |E \cap (A \times B)| - d_{ij} |A| |B|\right| \leq C |\F|^{-1/4} |V_i| |W_j|
 \end{equation}
where $d_{ij} := \frac{|E \cap (V_i \times W_j)|}{|V_i| |W_j|}$.
\end{itemize}
\end{lemma}

Comparing this lemma with Lemma \ref{szreg}, we see that one has substantially more regularity (a power gain in $|\F|$), and no exceptional pairs $(i,j)$, thanks to the bounded complexity of the set $E$.  Furthermore, the cells $V_1,\ldots,V_a,W_1,\ldots,W_b$ are not arbitrary, but are themselves definable with bounded complexity.  Let us illustrate this lemma with two simple examples:

\begin{example} Let $V=W=\F \backslash \{0\}$, and let $E$ be the set of all pairs $(v,w)$ such that $vw$ is a quadratic residue; this is clearly a definable set of bounded complexity.  Then we can regularise $E$ by partitioning $V=W$ into the quadratic residues and the non-quadratic residues, with the set $E$ having density either zero or one in each of the four pairs $V_i \times W_j$ created by this partition.
\end{example}

\begin{example}[Paley graph]\label{play} Let $V=W=\F$, and let $E$ be the set of all pairs $(v,w)$ such that $v-w$ is a quadratic residue.  Standard Gauss sum estimates then show that
$$ |E \cap (A \times B)| = \frac{1}{2}|A| |B| + O( |\F|^{2-1/2} )$$
for any $A, B \subset \F$, thus giving \eqref{elab} with the exponent $1/4$ improved to $1/2$, and with $a=b=1$.  (It may be that this improvement of $1/4$ to $1/2$ is in fact true in all cases; our method is unable to show this due to the fact that it invokes the Cauchy-Schwarz inequality at one point, which halves the exponent gain that one expects in bounds such as \eqref{elab}.)
\end{example}

As a very rough first approximation, one can interpret Lemma \ref{algreg} as an assertion that any definable subset of $V \times W$ behaves like some combination of the two basic examples listed above.

\begin{remark} A result in a somewhat similar spirit to Lemma \ref{algreg} was established by Kowalski \cite{kow}.  In our notation, the main result is as follows: if $E$ is a definable subset of complexity at most $M$ over a finite field $\F$ of prime order $p$, $f, g: \F \to \F$ are non-constant polynomials of degree at most $M$, and $\chi: \F \to \C$ is any multiplicative character, then
$$ |\sum_{x \in E} \chi(g(x)) e^{2\pi i f(x)/p}| \leq C_M \sqrt{p}$$
for some quantity $C_M$ depending only on $M$.  This allows for a substantial generalisation of Example \ref{play} to definable Cayley graphs, and also allows for some ``twisting'' of such graphs by multiplicative characters.  It is concievable that one could similarly twist the algebraic regularity lemma by allowing some ``twisted'' generalisation of the notion of definable set, but we will not pursue this issue here.
\end{remark}

\begin{remark}  A recent paper of Malliaris and Shelah \cite{shelah} links the absence of exceptional pairs in the regularity lemma with the concept of \emph{stability} from model theory (roughly speaking, the inability to definably create large induced copies of the half-graph).  The algebraic regularity lemma can thus be viewed as asserting a (somewhat exotic) form of model-theoretic stability for the language of finite fields of large characteristic.
\end{remark}

Actually, in applications we will use an iterated form of the regularity lemma in which one works not with a dense definable subset $E \subset V \times W$ of the product of two definable sets, but a subset $E \subset V_1 \times \ldots \times V_k$ of a bounded number of definable sets (for our application to expansion, we will take $k=4$); see Theorem \ref{threg-higher}.  This lemma is to the $k=2$ case as the Chung hypergraph regularity lemma \cite{chung-hyper} (see also \cite{frankl}) is to the Szemer\'edi regularity lemma.  It seems of interest to obtain stronger hypergraph regularity lemmas, analogous to those in \cite{gowers-hyper}, \cite{rodl}, \cite{rs}, \cite{tao-hyper}, but we will not pursue this matter here.  (See \cite{mubayi} for a discussion of the general poset of hypergraph regularity lemmas.)

The regularity lemma is not directly applicable to the expansion problem, basically because the graph $\{ (x_1,x_2,P(x_1,x_2))\mid x_1,x_2\in \F \}$ of the polynomial is too sparse a subset of $\F^3$ for this lemma to be useful.  However, as observed in \cite{bukh}, if one applies the Cauchy-Schwarz inequality a few times (in the spirit of \cite{gowers-4aps}), one can effectively replace the above graph by the set
$$ \{ (P(x_1,x_2), P(x_1,y_2), P(y_1,x_2), P(y_1,y_2))\mid x_1,x_2,y_1,y_2 \in \F \},$$
which for ``generic'' $P$ will be a dense subset of $\F^4$ to which the (iterated) algebraic regularity lemma can be applied.  (It is these applications of Cauchy-Schwarz that reduce the exponents in our final expansion results to be $1/8$ or $1/16$ instead of $1/4$.) There will be some exceptional cases in which this set fails to be Zariski dense, but we will be able to show (using a Riemann surface\footnote{This argument is closely related to the classical fact that the only one-dimensional algebraic groups over the complex numbers are (up to isomorphism) the additive group $(\C,+)$, the multiplicative group $(\C^\times, \cdot)$, and the elliptic curves.  The elliptic curve case will eventually be eliminated because the underlying map $P$ is polynomial rather than merely rational.} argument!) that those cases only arise when one has additive or multiplicative structure, giving rise to the trichotomies and tetrachotomies in Theorems \ref{expand-thm-modas}, \ref{expand-thm-weak}, \ref{expand-thm-ass}.

Our proof of the regularity lemma will use two somewhat exotic ingredients.  The first is the use of nonstandard analysis, in order to convert a quantitative problem involving finite fields of large characteristic into an equivalent qualitative problem involving pseudo-finite fields of zero characteristic.  The main reason for using the nonstandard formalism is so that we may deploy the second ingredient, which is the theory of the \'etale fundamental group of algebraic varieties over fields of characteristic zero\footnote{One can also define the \'etale fundamental group in positive characteristic, but the theory is significantly less favorable for our purposes; in particular, one cannot guarantee that this group is topologically finitely generated.  This is one of the main reasons why our results are limited to the large characteristic setting.  We will also take advantage of characteristic zero to use the theory of Riemann surfaces in order to analyse algebraic curves, although this is largely for reasons of convenience, as many of the Riemann surface facts we will use have purely algebraic counterparts that are also valid in positive characteristic.}.  The reason that the \'etale fundamental group comes into play is because it plays a key role in counting the number of connected components of certain algebraic varieties that will arise in the argument, and this in turn is needed to count the number of $\F$-points on those varieties thanks to the Lang-Weil estimates \cite{lang}.  Some special cases of this general theme are already visible in the work of Bukh and Tsimerman \cite{bt} on polynomial expansion; the \'etale fundamental group is not explicitly mentioned in their paper, but is implicitly present in some of the algebraic geometry lemmas used in that paper (e.g. \cite[Lemma 21]{bt}).

As a byproduct of our reliance on nonstandard methods, we do not obtain any quantitative bounds in our main theorems; in particular, we cannot explicitly give values for the constants $C$ in those theorems.  In principle such bounds could eventually be extracted from suitable finitisations of the arguments here, but this would require (among other things) effective versions of results on the \'etale fundamental group (in the case of sufficiently large characteristic, rather than in zero characteristic), which seems feasible but only after an enormous amount of effort, which we will not expend here.

\begin{remark} Throughout this paper we shall freely use the axiom of choice.  However, thanks to a well known result of G\"odel \cite{godel}, any result that can be formalized in first-order arithmetic (and the main results of this paper are of this type) and is provable in Zermelo-Frankel set theory with the axiom of choice (ZFC), can also be proven in Zermelo-Frankel set theory without the axiom of choice (ZF).
\end{remark}

\subsection{Acknowledgments}

The author is greatly indebted to Brian Conrad and Jordan Ellenberg for their many patient explanations of the \'etale fundamental group, to Jozsef Solymosi for useful suggestions, and to Antoine Chambert-Loir, Jordan Ellenberg, Norbert Hegyv\'ari, and Van Vu for corrections and comments.  The author was also partially supported by a Simons Investigator award from the Simons Foundation and by NSF grant DMS-0649473.  

\section{Algebraic geometry notation}

In this section we lay out the basic algebraic geometry notation that we will need throughout this paper.

\begin{definition}[Algebraic varieties]  Let $n$ be a natural number, and let $k$ be an algebraically closed field.  An \emph{affine variety} (or more precisely, an \emph{algebraic set}) in $k^n$ over $k$ is a set $V$ of the form
\begin{equation}\label{vdef}
V =\{ x \in k^n\mid P_1(x) = \ldots = P_m(x) = 0 \}
\end{equation}
for some polynomials $P_1,\ldots,P_m: k^n \to k$.    Similarly, we define the \emph{projective space} $\mathbb{P}^n(k)$ to be the space of equivalence classes $[x_1,\ldots,x_{n+1}]$ of tuples $(x_1,\ldots,x_{n+1}) \in k^{n+1} \backslash \{0\}$ after quotienting out by dilations by $k$, and define a \emph{projective variety} in $\mathbb{P}^n(k)$ to be a set $V$ of the form
\begin{equation}\label{vdef-2}
V =\{ x \in \mathbb{P}^n(k)\mid P_1(x) = \ldots = P_m(x) = 0 \}
\end{equation}
for some homogeneous polynomials $P_1,\ldots,P_m: k^{n+1} \to k$ (note that the constraint $P_1(x)=\ldots=P_m(x)=0$ is well-defined in $\mathbb{P}^n(k)$).   We embed $k^n$ in $\mathbb{P}^n(k)$ in the usual manner, identifying $(x_1,\ldots,x_n)$ with $[x_1,\ldots,x_n,1]$; thus for instance every affine variety can be viewed as a subset of an associated projective variety.  

A subset of $\mathbb{P}^n(k)$ is said to be a \emph{quasiprojective variety} if it is the set-theoretic difference of two projective varieties in $\mathbb{P}^n(k)$.  Thus for instance the set-theoretic difference of two affine varieties in $k^n$ is a quasiprojective variety.  
A \emph{constructible set} in $k^n$ is a boolean combination of finitely many affine varieties in $k^n$.  As noted in Section \ref{algreg-sec}, constructible sets are definable over $k$ (indeed, as $k$ is algebraically closed, the two concepts coincide in this setting), and so we can inherit the notion of complexity for such sets.  

We define the \emph{Zariski topology} on $\mathbb{P}^n(k)$ by declaring the projective varieties to be the closed sets; thus, for instance, the \emph{Zariski closure} $\overline{E}$ of a subset $E$ of $\mathbb{P}^n(k)$ is the intersection of all the projective varieties which contain that set.  One can then induce the Zariski topologies on other varieties by restriction.  For instance, in the affine space $k^n$, the Zariski closed sets are given by the affine varieties.

An affine (resp. projective) variety is \emph{geometrically irreducible}, or \emph{irreducible} for short, if it is non-empty and cannot be expressed as the union of two strictly smaller affine (resp. projective) varieties.  We say that a quasiprojective variety (or more generally, a constructible set) is irreducible if its Zariski closure is irreducible.  It is well known (see e.g. \cite[Propositions I.5.2, I.5.3]{mumf}) that any affine (resp. projective) variety can be uniquely decomposed into finitely many irreducible subvarieties, no two of which are contained in each other.  

The \emph{dimension} $\dim(V)$ of a non-empty affine (resp. projective) variety $V$ is the largest natural number $d$ for which there is a chain 
$$ \emptyset \neq V_0 \subsetneq V_1 \subsetneq \ldots \subsetneq V_d \subset V$$
of irreducible affine (resp. projective)  varieties $V_0, \ldots, V_d$; this is always a finite natural number.  We adopt the convention that the empty set has dimension $-\infty$.  The dimension of a quasiprojective variety (in $k^n$ or $\mathbb{P}^n(k)$) is defined to be the dimension of its Zariski closure in the indicated ambient space.  A variety will be called a \emph{curve} if it has dimension one.

If $V,W$ are two varieties with $V \subseteq W$, we say that $V$ is a \emph{subvariety} of $W$.  If $V \subsetneq W$, we say that $V$ is a \emph{strict subvariety} of $W$.

Let $\F$ be a subfield of $k$.  We say that an affine (resp. projective) variety is \emph{defined over $\F$} if one can find polynomials $P_1,\ldots,P_m$ with coefficients in $\F$ for which \eqref{vdef} (resp. \eqref{vdef-2}) holds.  A quasiprojective variety is defined over $\F$ if it is the set-theoretic difference of two affine varieties defined over $\F$.  If $V \subset k^n$ is a quasiprojective variety, we define $V(\F) := V \cap \F^n$ to be the $F$-points of $V$; note that this is a definable subset over $\F$.

Let $V$ be a quasiprojective variety.  If $V \subset k^n$, a \emph{regular function} on $V$ is a function $f: V \to k$ which, at every point $p$ in $V$, agrees with a rational function from (an open dense subset of) $k^n$ to $k$ on an open neighbourhood of $p$ in $V$.  If instead $V \subset \mathbb{P}^n(k)$, a regular function on $V$ is a function which is regular when restricted to the $n+1$ affine subsets $\{ [x_1,\ldots,x_{n+1}] \in \mathbb{P}^n(k)\mid x_i \neq 0 \}$, $i=1,\ldots,n+1$ that cover ${\mathbb P}^n(k)$, each of which can be identified with the affine space $k^n$ in the obvious manner.  The ring of regular functions on $V$ is denoted $k[V]$, and its fraction field (which is well defined for irreducible $V$) is denoted $k(V)$.  A map $\phi: V \to W$ between two quasiprojective varieties is a \emph{regular morphism} if every regular function on $W$ pulls back by $\phi$ to a regular function on $V$.  A \emph{regular isomorphism} is an invertible regular morphism whose inverse is also regular.  
\end{definition}

In most of this paper, we will only need to work with constructible subsets of affine space $k^n$, such as affine varieties.  However in Section \ref{second-sec} we will also need to work with projective varieties (in order to use the theory of compact Riemann surfaces).

If $V$ is an affine variety, then $k[V]$ is just the restriction of the polynomials on $k^n$ to $V$; see \cite[Theorem 3.2]{hart}.  If $V$ is an irreducible projective variety, then $k[V]$ consists only of the constant functions; see \cite[Theorem 3.4]{hart}.  

We recall some basic facts about dimension.  Firstly, we have $\dim(k^n) = n$ for any $n$, and $\dim(V \times W) = \dim(V) + \dim(W)$ for any constructible sets $V,W$; see e.g. \cite[\S I.7]{mumf}.  We clearly also have $\dim(V) \leq \dim(W)$ whenever $V \subset W$, with strict inequality when $W$ is an irreducible affine (resp. projective) variety and $V$ is a strict affine (resp. projective) subvariety.  By construction, all affine or projective varieties of finite non-zero cardinality have dimension zero, and all such varieties of infinite cardinality have dimension greater than zero, and so the same is also true for constructible sets.  We also have the following basic fact:

\begin{proposition}[Projections]\label{proj}  Let $k$ be an algebraically closed field, and let $\pi: V \to W$ be a regular map between two quasiprojective varieties $V,W$ with $V$ irreducible.  Then $\pi(V)$ is an irreducible constructible set, and there is a subvariety $\Sigma$ of $\pi(V)$ of dimension at most $\dim(\pi(V))-1$ such that $\pi^{-1}(\{x\}) \cap V$ has dimension $\dim(V)-\dim(\pi(V))$ for all $x \in \pi(V) \backslash \Sigma$.  Furthermore, the set $\pi^{-1}(\Sigma) \cap V$ has dimension at most $\dim(V)-1$.
\end{proposition}

\begin{proof} See \cite[Lemma A.8]{bgt-product}.
\end{proof}

This has the following consequence.  If $V$ is an irreducible constructible set, we say that a property $P(x)$ of points in $V$ holds \emph{for generic $x \in V$} if there is a subvariety $\Sigma$ of $V$ of dimension at most $\dim(V)-1$ such that $P(x)$ holds for all $x \in V \backslash \Sigma$.

\begin{lemma}[Generic Fubini-type theorem]\label{fubini}  Let $V, W$ be constructible sets, and let $E$ be a constructible subset of $V \times W$.  Then the following are equivalent:
\begin{itemize}
\item[(i)] For generic $v \in V$, one has $(v,w) \in E$ for generic $w \in W$.  
\item[(ii)] For generic $w \in W$, one has $(v,w) \in E$ for generic $v \in V$.
\item[(iii)] For generic $(v,w) \in V \times W$, one has $(v,w) \in E$.
\end{itemize}
\end{lemma}

\begin{proof}  By symmetry, it suffices to show the equivalence of (i) and (iii).  Write $\Sigma := (V \times W) \backslash E$.  If (iii) holds, then 
$\overline{\Sigma}$ has dimension at most $\dim(V)+\dim(W)-1$.  Writing $\pi: \overline{\Sigma} \to \overline{V}$ for the projection map and using Proposition \ref{proj}, we see that the generic fibre of $\pi$ has dimension at most $\dim(W)-1$, giving (i).

Conversely, if (iii) fails, then $\overline{\Sigma}$ has dimension $\dim(V)+\dim(W)$, and so contains at least one of the components $\overline{V}_i \times \overline{W}_j$ where $\overline{V}_i$ is a top-dimensional irreducible component of $\overline{V}$, and similarly for $\overline{W}_j$.  The complement of $\Sigma$ in $\overline{V}_i \times \overline{W}_j$ then has dimension at most $\dim(V)+\dim(W)-1$, and so by the previous arguments, we see that for generic $v \in \overline{V}_i$, one has $w \in \Sigma$ for generic $w \in \overline{W}_j$, which contradicts (i) as required.
\end{proof}

\section{Nonstandard formulation}

As discussed in the introduction, it will be convenient to pass to a nonstandard analysis formalism in order to take full advantage of the existing literature in algebraic geometry and on the \'etale fundamental group, as well as to be able to use some tools from the theory of Riemann surfaces which are only available for varieties when the characteristic is zero.  In this section, we set up this formalism, and give the nonstandard version of the main theorems.

We will assume the existence of a \emph{standard universe} ${\mathfrak U}$ which contains all the objects and spaces that one is interested in (such as the natural numbers $\N$, the real numbers $\R$, finite fields $\F$, constructible sets or varieties defined over $\F$ or $\overline{\F}$, maps between such spaces, etc.).  The precise construction of this universe is not particularly important for our purposes, so long as it forms a set in our external set theory.  We refer to objects and spaces inside the standard universe as \emph{standard objects} and \emph{standard spaces}, with the latter being sets whose elements are in the former category.  Thus, for instance, we refer to elements of $\N$ as standard natural numbers.

For the rest of the paper, we fix a non-principal ultrafilter $\alpha \in \beta \N \backslash \N$ on the natural numbers, that is to say a collection of subsets of $\N$ obeying the following axioms:
\begin{enumerate}
\item If $E, F \in \alpha$, then $E \cap F \in \alpha$.
\item If $E \subset F \subset \N$ and $E \in \alpha$, then $F \in \alpha$.
\item If $E \subset \N$, then exactly one of $E$ and $\N \backslash E$ lies in $\alpha$.
\item No finite subset of $\N$ lies in $\alpha$.
\end{enumerate}
The existence of such a non-principal ultrafilter follows easily from Zorn's lemma.

Throughout the paper, we fix a non-principal ultrafilter $\alpha$.  A property $P(\n)$ depending on a natural number $\n$ is said to hold \emph{for $\n$ sufficiently close to $\alpha$} if the set of $\n$ for which $P(\n)$ holds lies in $\alpha$.  A set of natural numbers lying in $\alpha$ will also be called an \emph{$\alpha$-large set}.

Once we have fixed this ultrafilter, we can define nonstandard objects and spaces.

\begin{definition}[Nonstandard objects and functions]
Given a sequence $(x_\n)_{\n \in \N}$ of standard objects in ${\mathfrak U}$, we define their \emph{ultralimit} $\lim_{\n \to \alpha} x_\n$ to be the equivalence class of all sequences $(y_\n)_{\n \in \N}$ of standard objects in ${\mathfrak U}$ such that $x_\n = y_\n$ for $\n$ sufficiently close to $\alpha$.  Note that the ultralimit $\lim_{\n \to \alpha} x_\n$ can also be defined even if $x_\n$ is only defined for $\n$ sufficiently close to $\alpha$.

An ultralimit of standard natural numbers is known as a \emph{nonstandard natural number}, an ultralimit of standard real numbers is known as a \emph{nonstandard real number}, and so on.

For any standard object $x$, we identify $x$ with its own ultralimit $\lim_{\n \to \alpha} x$.  Thus, every standard natural number is a nonstandard natural number, etc.

Any operation or relation on standard objects can be extended to nonstandard objects in the obvious manner.  Indeed, if $O$ is a $k$-ary operation, we define
$$ O( \lim_{\n \to \alpha} x^1_\n, \ldots, \lim_{\n \to \alpha} x^k_\n ) := \lim_{\n \to \alpha} O( x^1_\n, \ldots, x^k_\n ) $$
and if $R$ is a $k$-ary relation, we define $R(\lim_{\n \to \alpha} x^1_\n, \ldots, \lim_{\n \to \alpha} x^k_\n )$ to be true iff $R(x^1_\n,\ldots,x^k_\n)$ is true for all $\n$ sufficiently close to $\alpha$.  One easily verifies that these nonstandard extensions of $O$ and $R$ are well-defined.
\end{definition}

\begin{example}
The sum of two nonstandard real numbers $\lim_{\n \to \alpha} x_\n$, $\lim_{\n \to \alpha} y_\n$ is the nonstandard real number
$$\lim_{\n \to \alpha} x_\n + \lim_{\n \to \alpha} y_\n = \lim_{\n \to \alpha} x_\n + y_\n,$$
and the statement $\lim_{\n \to \alpha} x_\n < \lim_{\n \to \alpha} y_\n$ means that $x_\n < y_\n$ for all $\n$ sufficiently close to $\alpha$.
\end{example}

We will use the usual asymptotic notation from nonstandard analysis: 

\begin{definition}[Asymptotic notation] A nonstandard real number $x \in \ultra \R$ is said to be \emph{bounded} if one has $|x| \leq C$ for some standard $C>0$, and \emph{unbounded} otherwise.  Similarly, we say that $x$ is \emph{infinitesimal} if $|x| \leq c$ for all standard $c>0$; in the former case we write $x = O(1)$, and in the latter $x=o(1)$.  For every bounded real number $x \in \ultra \R$ there is a unique standard real number $\st(x) \in \R$, called the \emph{standard part} of $\R$, such that $x = \st(x)+o(1)$, or equivalently that $\st(x) - \eps \leq x \leq \st(x)+\eps$ for all standard $\eps>0$.  Indeed, one can set $\st(x)$ to be the supremum of all the real numbers $y$ such that $x>y$ (or equivalently, the infimum of all the real numbers $y$ such that $x<y$).

We write $X = O(Y)$, $X \ll Y$, or $Y \gg X$ if we have $X \leq CY$ for some standard $C$; and we write $X = o(Y)$, $X \lll Y$, or $Y \ggg X$ if we have $X \leq \eps Y$ for every standard $\eps>0$.
\end{definition}

\begin{definition}[Ultraproducts]  Let $(X_\n)_{\n \in \N}$ be a sequence of standard spaces $X_\n$ in ${\mathfrak U}$ indexed by the natural numbers.  The \emph{ultraproduct} $\prod_{\n \to \alpha} X_\n$ of the $X_\n$ is defined to be the space of all ultralimits $\lim_{\n \to \alpha} x_\n$, where $x_\n \in X_\n$ for all $\n$.  We refer to the ultraproduct of standard sets as an \emph{nonstandard set}; in a similar vein, an ultraproduct of standard finite sets is a \emph{nonstandard finite set}, and an ultraproduct of standard finite fields is a \emph{nonstandard finite field}.  We refer to $\ultra X := \prod_{\n \to \alpha} X$ as the \emph{ultrapower} of a standard set $X$; the identification of $x$ with $\lim_{\n \to \alpha} x$ causes $X$ to be identified with a subset of $\ultra X$.  We will refer to the ultrapower  $\ultra {\mathfrak U}$ of the standard universe ${\mathfrak U}$ as the \emph{nonstandard universe}.

In a similar spirit, if $f_\n: X_\n \to Y_\n$ is a collection of standard functions between standard sets $X_\n,Y_\n$, we can form the \emph{ultralimit} $f := \lim_{\n \to \alpha} f_\n$ to be the function from $X := \prod_{\n \to \alpha} X_\n$ to $Y := \prod_{\n \to \alpha} Y_\n$ defined by the formula
$$ f( \lim_{\n \to \alpha} x_n ) := \lim_{\n \to \alpha} f_\n(x_\n).$$
We refer to such functions as \emph{nonstandard functions} (also known as \emph{internal functions} in the nonstandard analysis literature).

As with nonstandard objects, any operation or relation on standard spaces can be converted to a nonstandard analogue in the usual manner.  For instance, the nonstandard cardinality of a nonstandard finite set $X = \prod_{\n \to \alpha} X_\n$ is the nonstandard natural number
$$ |X| := \lim_{\n \to \alpha} |X_\n|.$$
Note that this is a different concept from the usual (or \emph{external}) cardinality of $X$; indeed, nonstandard finite sets usually have an uncountable external cardinality.  Similarly, if $f: X \to \ultra \R$ is a nonstandard function $f = \lim_{\n \to \alpha} f_\n$ defined on a nonstandard finite set $X = \prod_{\n \to \alpha} X_\n$, we can define the nonstandard sum $\sum_{x \in X} f(x)$ to be the nonstandard real number
$$ \sum_{x \in X} f(x) := \lim_{\n \to \alpha} \sum_{x_\n\in X_\n} f_\n(x_\n).$$
\end{definition}

A fundamental property of ultralimits is that they preserve first-order statements and predicates, a fact known as \emph{{\L}os's theorem}:

\begin{theorem}[{\L}os's theorem with parameters and ultraproducts]\label{los-param}  Let $m, k$ be standard natural numbers.  For each $1 \leq i \leq m$, let $x_i = \lim_{\n \to \alpha} x_{i,\n}$ be a nonstandard object, and for each $1 \leq j \leq k$, let $A_j = \prod_{\n \to \alpha} A_{j,\n}$ be a nonstandard set.  If $P(y_1,\ldots,y_m; B_1,\ldots,B_k)$ is a predicate over $m$ objects and $k$ sets, with the sets $A_1,\ldots,A_k$ only appearing in $P$ through the membership predicate $x \in B_j$ for various $j$ and various objects $B_j$, then $P(x_1,\ldots,x_m; A_1,\ldots,A_k)$ is true \textup{(}as quantified over the nonstandard universe $\ultra {\mathfrak U}$\textup{)} if and only if $P(x_{1,\n},\ldots,x_{m,\n}; A_{1,\n},\ldots,A_{k,\n})$ is true for all $\n$ sufficiently close to $\alpha$ \textup{(}as quantified over the standard universe ${\mathfrak U}$\textup{)}.
\end{theorem}

\begin{proof} See e.g. \cite[Theorem A.6]{bgt}.
\end{proof}

Another fundamental property is that of \emph{countable saturation}:

\begin{lemma}[Countable saturation]\label{countsat}  Let $k$ be a standard natural number, and let $X_1,\ldots,X_k$ be nonstandard spaces.  For each standard natural number $n$, let $P_n(x_1,\ldots,x_k)$ be predicates defined for $x_i \in X_i$, using some finite number of nonstandard objects and spaces as constants, and quantified over the nonstandard universe.  Suppose that for any standard natural number $N$, there exists $x_1 \in X_1,\ldots,x_k \in X_k$ such that $P_n(x_1,\ldots,x_k)$ holds for all (standard) $n \leq N$.  Then there exists $x_1 \in X_1,\ldots,x_k \in X_k$ such that $P_n(x_1,\ldots,x_k)$ holds for all $n \in \N$.
\end{lemma}

\begin{proof} Write $X_i = \prod_{\n \to \alpha} X_{i,\n}$ and $P_n = \lim_{\n \to \infty} P_{n,\n}$.  By {\L}os's theorem, we see that for every standard natural number $M$, there exists an $\alpha$-large set $E_M$ and elements $x_{i,\n,M} \in X_{i,\n}$ for $\n \in E_M$ and $i=1,\ldots,k$ such that $P_{n,\n}(x_{1,\n,M},\ldots,x_{k,\n,M})$ holds for all $\n \in E_M$ and $n=1,\ldots,M$.  By shrinking the $E_M$ if necessary, we may assume that the $E_M$ are non-increasing in $M$.  We then define $M_\n$ for any $\n \in E_1$ to be the largest integer in $\{1,\ldots,\n\}$ for which $\n \in E_{M_\n}$.  If we then set $x_i := \lim_{\n \to \alpha} x_{i,\n,M_\n}$ for $i=1,\ldots,k$, we see from {\L}os's theorem that $x_i \in X_i$ for all $i=1,\ldots,k$, and $P_n(x_1,\ldots,x_k)$ holds for all standard $n$, as required.
\end{proof}

A typical application of countable saturation is the following: if $f: X \to \ultra \R$ is a nonstandard function with the property that $f(x)$ is bounded for every $x \in X$, then there is a uniform bound $|f(x)| \leq M$ for some standard $M$ (for otherwise the predicates $|f(x)| \geq n$ would form a counterexample to Lemma \ref{countsat}).  This automatic uniformity is one advantage of the nonstandard framework: one does not need to make as many careful distinctions between the order of various quantifiers in one's arguments (but one instead has to carefully distinguish between standard and nonstandard quantities).  In this paper we will not use countable saturation very often, however, because in our applications such uniform bounds are also often obtainable directly from algebraic geometry methods (e.g. bounding the degree of various varieties or maps).

We will be working extensively with nonstandard finite fields $\F = \prod_{\n \to \alpha} \F_\n$ in this paper, which are examples of what are known as \emph{pseudo-finite fields} in the model theory literature, because (by {\L}os's theorem) they obey all the first-order sentences in the language of fields that hold for all finite fields.   By {\L}os's theorem, the algebraic closure $\overline{\F}$ of a nonstandard finite field $\F$ is contained in the ultraproduct of the algebraic closures $\overline{\F_\n}$ of the associated finite fields; indeed, it is the space of ultralimits $\lim_{\n \to \alpha} x_\n$, where for all $\n$ sufficiently close to $\alpha$, $x_\n$ lies in an extension of $\F_\n$ of degree at most $C$ for some $C$ independent of $\n$.  The nonstandard finite field $\F$ has a nonstandard \emph{Frobenius endomorphism} $\Frob_\F: \overline{\F} \to \overline{\F}$, defined as the restriction to $\overline{\F}$ ultralimit of the standard Frobenius endomorphism $\Frob_{\F_\n}: \overline{\F_\n} \to \overline{\F_\n}$ defined by
$$ \Frob_{\F_\n}(x_\n) := x_\n^{|\F_{\n}|}.$$
(Note that $\Frob_{\F_\n}$ preserves every finite extension of $\F_\n$, and so $\Frob_\F$ is well-defined on $\overline{\F}$.)
Note (again by {\L}os's theorem) that $\F$ can be viewed as the set of fixed points of $\Frob_\F$ in $\overline{\F}$, since $\F_\n$ is the set of fixed points of $\Frob_{\F_\n}$ in $\overline{\F_\n}$ for all $\n$.  Later on, we will use this Frobenius endomorphism to determine which varieties over $\overline{\F}$ are actually defined over $\F$.

Ultraproducts interact well with definable sets.  From {\L}os's theorem, we see that if $\F = \prod_{\n \to \alpha} \F_\n$ is a nonstandard field and $d$ is a standard natural number, then a set $E \subset \F^d$ is definable over $\F$ if and only if $E$ can be expressed as an ultraproduct $E = \prod_{\n \to \alpha} E_\n$, where for all $\n$ sufficiently close to $\alpha$, $E_\n \subset \F_\n^d$ is definable over $\F_n$ with complexity at most $M$, for some $M$ independent of $\n$.

In a similar vein, a function $P: \F \to \F$ on a nonstandard field $\F = \prod_{\n \to \alpha} \F_\n$ is an \emph{external}\footnote{Taking the ultralimit of polynomials whose degree goes to infinity instead of being uniformly bounded will lead to a function which is a nonstandard polynomial, but not an external one.  For instance, the nonstandard Frobenius endomorphism $\Frob_\F$ is a nonstandard polynomial, but is not external if the characteristic of the $\F_\n$ goes to infinity.} polynomial (that is, a polynomial in the usual sense) if and only if it is an ultralimit $P = \lim_{\n \to \alpha} P_\n$ of (standard) polynomials $P_\n: \F_\n \to \F_\n$ of uniformly bounded degree.  

We can now give the nonstandard version of the main theorems stated in the introduction.  We first give the nonstandard version of Theorem \ref{expand-thm-modas}:

\begin{theorem}[Moderate asymmetric expansion, nonstandard formulation]\label{expand-thm-modas-nonst} Let $\F$ be a nonstandard finite field of (external) characteristic zero, and let $P: \F \times \F \to \F$ be an (external) polynomial.  Then at least one of the following statements hold:
\begin{itemize}
\item[(i)] (Additive structure)  One has
$$ P(x_1,x_2) = Q(F_1(x_1)+F_2(x_2))$$
(as a polynomial identity in the indeterminates $x_1,x_2$) for some (external) polynomials $Q, F_1, F_2: \F \to \F$.
\item[(ii)] (Multiplicative structure)  One has
$$ P(x_1,x_2) = Q(F_1(x_1) F_2(x_2))$$
for some (external) polynomials $Q, F_1, F_2: \F \to \F$.
\item[(iii)] (Moderate asymmetric expansion)  One has
$$ |P(A_1,A_2)| \gg |\F|$$
whenever $A_1, A_2$ are nonstandard subsets of $\F$ with $|A_1| |A_2| \ggg |\F|^{2-1/8}$.
\end{itemize}
\end{theorem}

Let us now see why Theorem \ref{expand-thm-modas-nonst} implies Theorem \ref{expand-thm-modas}.   Suppose for contradiction that Theorem \ref{expand-thm-modas-nonst} was true, but Theorem \ref{expand-thm-modas} failed.  Carefully negating the quantifiers (and using the axiom of choice), we conclude that there is a standard natural number $d$ such that, for every standard natural number $\n$, one can find a finite field $\F_\n$ of characteristic at least $\n$, and a polynomial $P_\n: \F_\n \times \F_\n \to \F_\n$ of degree at most $d$, such that $P_\n$ is not expressible in the form \eqref{padd} or \eqref{pmult} for any polynomials $Q_\n$, $F_{1,\n}$, $F_{2,\n}$, and such that there exist subsets $A_{1,\n}$, $A_{2,\n}$ of $\F_\n$ with 
$$ |A_{1,\n}| |A_{2,\n}| \geq \n |\F_\n|^{2-1/8}$$
but
$$ |P_\n(A_{1,\n}, A_{2,\n})| \leq \n^{-1} |\F_\n|.$$
We now take ultralimits, giving the nonstandard finite field $\F := \prod_{\n \to \alpha} \F_\n$ with nonstandard subsets $A_i := \prod_{\n \to \alpha} A_{i,\n}$ for $i=1,2$ and the map $P := \lim_{\n \to \alpha} P_\n$.  For any standard natural number $k$, $\F_\n$ has characteristic greater than $k$ for all but finitely many $\n$, so by {\L}os's theorem, $\F$ does not have characteristic $k$ for any positive $k$, and thus has characteristic zero.
Because the $P_\n$ are polynomials of degree at most $d$, $P$ is an (external) polynomial also.  From {\L}os's theorem, we have $|A_1| |A_2| \ggg |\F|^{2-1/18}$ and $|P(A_1,A_2)| \lll |\F|$, and $P$ cannot be expressed in either of the two forms \eqref{padd}, \eqref{pmult}.  This gives a counterexample to Theorem \ref{expand-thm-modas-nonst}, and the claim follows.

It is also not difficult to show that Theorem \ref{expand-thm-modas} implies Theorem \ref{expand-thm-modas-nonst}, but we will not need this implication here and so will leave it to the interested reader.

In a similar vein, Theorem \ref{expand-thm-weak} and Theorem \ref{expand-thm-ass} follow from these nonstandard counterparts:

\begin{theorem}[Weak expansion, nonstandard formulation]\label{expand-thm-weak-nonst} Let $\F$ be a nonstandard finite field of (external) characteristic zero, and let $P: \F \times \F \to \F$ be an (external) polynomial.   Then at least one of the following statements hold:
\begin{itemize}
\item[(i)] (Additive structure)  One has
$$ P(x_1,x_2) = Q(aF(x_1)+bF(x_2))$$
for some polynomials $Q, F: \F \to \F$, and some elements $a,b \in \F$.
\item[(ii)] (Multiplicative structure)  One has
$$ P(x_1,x_2) = Q(F(x_1)^a F(x_2)^b)$$
for some polynomials $Q, F: \F \to \F$, and some standard natural numbers $a,b$.
\item[(iii)] (Weak expansion)  One has
$$ |P(A,A)| \gg |\F|^{1/2} |A|^{1/2}$$
whenever $A \subset \F$ is a nonstandard subset with $|A| \ggg |\F|^{1-1/16}$.
\end{itemize}
\end{theorem}

\begin{theorem}[Almost strong asymmetric expansion, nonstandard formulation]\label{expand-thm-ass-nonst} Let $\F$ be a nonstandard finite field of (external) characteristic zero, and let $P: \F \times \F \to \F$ be an (external) polynomial.   Then at least one of the following statements hold:
\begin{itemize}
\item[(i)] (Additive structure)  One has
$$ P(x_1,x_2) = Q(F_1(x_1)+F_2(x_2))$$
for some polynomials $Q, F_1, F_2$.
\item[(ii)] (Multiplicative structure)  One has
$$ P(x_1,x_2) = Q(F_1(x_1) F_2(x_2))$$
for some polynomials $Q, F_1, F_2$.
\item[(iii)] (Algebraic constraint) One has irreducible affine curves $V,W$ defined over $\overline{\F}$ and the constraint
$$ P( f(x_1), g(x_2) ) = h( Q(x_1,x_2) )$$
for all $x_1 \in V, x_2 \in W$ and some polynomials $f: V \to \overline{\F}$, $g: W \to \overline{\F}$, $h: \overline{\F} \to \overline{\F}$, $Q: V \times W \to \overline{\F}$ defined over $\overline{\F}$, with $f,g$ non-constant and $h$ having degree at least two.
\item[(iv)] (Almost strong asymmetric expansion)  One has
$$ |\F \backslash P(A_1,A_2)| \ll |\F| \left(\frac{|A_1| |A_2|}{|\F|^{2-1/8}}\right)^{-1/2}$$
whenever $A_1, A_2$ are non-empty nonstandard subsets of $\F$.
\end{itemize}
\end{theorem}

The derivations of Theorem \ref{expand-thm-weak}, \ref{expand-thm-ass} from Theorem \ref{expand-thm-weak-nonst}, \ref{expand-thm-ass-nonst} are closely analogous to the derivation of Theorem \ref{expand-thm-modas} from Theorem \ref{expand-thm-modas-nonst} and are omitted.

Finally, Lemma \ref{algreg} also follows from a nonstandard counterpart:

\begin{lemma}[Algebraic regularity lemma, nonstandard formulation]\label{algreg-nonst}  Let $\F$ be a nonstandard finite field of (external) characteristic zero, let $V, W$ be non-empty definable subsets over $\F$, and let $E \subset V \times W$ be another definable set.  Then there exists partitions $V = V_1 \cup \ldots \cup V_a$, $W = W_1 \cup \ldots \cup W_b$ into a (standard) finite number of definable sets, with the following properties:
\begin{itemize}
\item (Largeness) For all $i \in \{1,\ldots,a\}$ and $j \in \{1,\ldots,b\}$, one has $|V_i| \gg |V|/C$ and $|W_j| \gg |W|/C$.  
\item ($|\F|^{-1/4}$-regularity)  For all $(i,j) \in \{1,\ldots,a\} \times \{1,\ldots,b\} \backslash I$, and all $A \subset V_i, B \subset W_j$, one has
$$ \left| |E \cap (A \times B)| - d_{ij} |A| |B|\right| \ll |\F|^{-1/4} |V_i| |W_j|$$
where $d_{ij} := \frac{|E \cap (V_i \times W_j)|}{|V_i| |W_j|}$.
\end{itemize}
\end{lemma}

Again, we omit the derivation, as it is closely analogous to the previous derivations.

It remains to establish Theorem \ref{expand-thm-modas-nonst}, Theorem \ref{expand-thm-ass-nonst}, Theorem \ref{expand-thm-weak-nonst}, and Lemma \ref{algreg-nonst}.  This will be the focus of the remainder of the paper, with the latter lemma being used as a crucial tool to prove the first three theorems.

As mentioned previously, there are two main reasons why we move to a nonstandard framework.  The first is that one no longer has to explicitly keep track of the complexity of various definable sets or algebraic varieties that one will shortly encounter in the argument.  This allows one to use many existing results from algebraic geometry without modification, as these results are usually phrased qualitatively rather than quantitatively, and so do not come with explicit bounds on complexity.  The other reason is that now that we have passed to a field of characteristic zero, many aspects of algebraic geometry become simpler; in particular, varieties are generically smooth, and \'etale fundamental groups are topologically finitely generated. Furthermore, we can take advantage of embeddings into the complex field $\C$ (Lefschetz principle) in order to exploit the theory of Riemann surfaces.  It would be rather difficult (though not entirely impossible) to replicate these facts in the original setting of finite fields of large characteristic.

\section{Definable sets and Lang-Weil type bounds}

In order to prove the algebraic regularity lemma, we will need some results in the model theory literature \cite{ax}, \cite{kiefe}, \cite{fried}, \cite{cherlin}, \cite{cher2}, \cite{kow} on definable subsets of nonstandard finite fields.   

We will need the fact that definability over a nonstandard finite field $\F$ can be detected using the Frobenius map:

\begin{lemma}\label{ose}  Let $\F$ be a nonstandard finite field, with algebraic closure $\overline{\F}$, and let $V \subset \overline{\F}^n$ be a quasiprojective variety.  Then $V$ is defined over $\F$ if and only if it is invariant with respect to the action of the nonstandard Frobenius map $\Frob_\F$ (which acts componentwise on $\overline{\F}^n$).
\end{lemma}

\begin{proof}  The ``only if'' part is clear, so we focus on the ``if'' part.  First suppose that $V$ is an affine variety that is invariant under $\Frob_\F$.  Then the ideal $I(V)$ of polynomials that vanish on $V$ is also $\Frob_\F$-invariant (where $\Frob_\F$ acts on each coefficient of a given polynomial separately).  If we let $P_1,\ldots,P_m$ be the reduced Gr\"obner basis of $I(V)$ with respect to lexicographical ordering (see e.g. \cite{cox}), then this basis is unique, and is thus also $\Frob_\F$-invariant, that is to say the coefficients of the $P_i$ lie in $\F$.  Thus $V$ is defined over $\F$ as claimed.

Now suppose $V$ is a quasiprojective variety that is invariant under $\Frob_\F$.  The Zariski closure $\overline{V}$ of $V$ in $\overline{\F}^n$ is then also invariant under $\Frob_\F$, and by the preceding discussion is thus defined over $\F$; similarly for the affine variety $\overline{V} \backslash V$.  The claim follows.
\end{proof}

We will also need the fact that definable sets over nonstandard finite fields are projections of the $\F$-points of varieties, and their cardinality is comparable to a power of $|\F|$.  More precisely, we have

\begin{theorem}[Almost quantifier elimination]\label{aqe}  Let $\F$ be a nonstandard finite field, and let $E$ be a subset of $\F^n$ for some (standard) natural number $n$.  Then $\E$ is a definable set if and only if it can be expressed as the intersection of finitely many sets, each of the form
\begin{equation}\label{vaf}
 \{ x \in \F^n \mid \exists t \in \F: P(x,t)=0\}
\end{equation}
for some polynomial $P: \F^n \times \F \to \F$ with coefficients in $\F$.  

Furthermore, if $E$ is definable, the Zariski closure $\overline{E}$ of $E$ in $\overline{F}^n$ is the union of finitely many geometrically irreducible affine varieties defined over $\F$, and the nonstandard cardinality $|E|$ of $E$ is given by
\begin{equation}\label{asy}
 |E| = (\sigma + O(|\F|^{-1/2})) |\F|^{\dim(\overline{E})}
\end{equation}
for some standard positive rational number $\sigma$ (with the convention that $|\F|^{-\infty}=0$).  In particular, we have
\begin{equation}\label{fee}
 |\F|^{\dim(\overline{E})} \ll |E| \ll |\F|^{\dim(\overline{E})}.
\end{equation}
\end{theorem}

\begin{proof} These are the main results of \cite{cher2}.  The fact that $\overline{E}$ consists only of varieties defined over $\F$ follows from Lemma \ref{ose}, since $E$ and hence $\overline{E}$ is Frobenius-invariant.  The fact that the exponent of $\F$ is the dimension of the Zariski closure of $E$ is \cite[Proposition 4.9]{cher2}.
\end{proof}

We illustrate the almost quantifier elimination \eqref{aqe} with some simple examples.  The space of quadratic residues of $\F$ that include $0$ can be expressed as
$$ \{ x \in \F \mid \exists t \in F: x-t^2 = 0 \},$$
while the set of non-zero elements of $F$ can be expressed as
$$ \{ x \in \F \mid \exists t \in F: xt-1 = 0 \},$$
the singleton set $\{0\}$ can be expressed as
$$ \{ x \in \F \mid \exists t \in F: x = 0 \},$$
and the set of nonquadratic residues of $\F$ (again including $0$) can be expressed as
$$ \{ x \in \F \mid \exists t \in F: ax-t^2 = 0 \},$$
where $a$ is an invertible quadratic non-residue.  By intersecting these sets together, we can then create other definable sets, such as the set of all non-zero quadratic residues.  These examples also show that the quantity $\sigma$ appearing in \eqref{asy} need not be an integer (for instance, in the case of quadratic residues, $\sigma$ is $1/2$ when the characteristic is not equal to $2$).

Theorem \ref{aqe} gives the order of magnitude on the cardinality $|E|$ of a definable set, but does not specify exactly what the rational constant $\sigma$ in the asymptotic \eqref{asy}.  The following bound computes this constant in the case of a quasiprojective variety.

\begin{lemma}[Lang-Weil bound]\label{sz}  Let $\F$ be a nonstandard finite field, and let $V \subset \overline{\F}^n$ be a quasiprojective variety defined over $\overline{\F}$.  Then one has
$$ |V(\F)| = (c + O(|\F|^{-1/2})) |\F|^{\dim(V)}$$
where $c$ are the number of top-dimensional geometrically irreducible components of $V$ (i.e. components of dimension exactly $\dim(V)$) which are defined over $\F$.  
\end{lemma}

\begin{proof}  By Theorem \ref{aqe} (or by cruder estimates, such as \cite[Lemma 1]{lang}), any affine variety of dimension strictly less than $\dim(V)$ has at most $O(|\F|^{\dim(V)-1})$ $\F$-points.  Thus, by replacing $V$ by its Zariski closure and then removing all lower dimensional components, we may assume without loss of generality that $V$ is a geometrically irreducible affine variety. 

If $V$ is not defined over $\F$, then by Lemma \ref{ose}, $\Frob_\F(V)$ is a different variety from $V$, and so $V \cap \Frob_\F(V)$ has dimension strictly less than $\dim(V)$.  But this variety contains all the $\F$ points of $\F$, and so $V$ only has at most $O(|\F|^{\dim(V)-1})$ $\F$-points, and the claim follows in this case (with $c=0$).

Finally, if $V$ is geometrically irreducible and defined over $\F$, the claim follows from \cite[Theorem 1]{lang}.
\end{proof}

\section{Proof of regularity lemma}

We now prove Lemma \ref{algreg-nonst}.  The first step is to pass from the set $E$ to a more tractable counting function (essentially the ``square'' of $E$), as follows.

\begin{proposition}[First reduction]\label{First-reg}  Let $\F$ be a nonstandard finite field of characteristic zero, let $V, W$ be definable sets over $\F$ with $\overline{V}, \overline{W}$ geometrically irreducible, and let $E$ be a definable subset of $V \times W$.  Let $\mu: W \times W \to \ultra \N$ denote the nonstandard counting function defined by the formula
\begin{equation}\label{mud}
 \mu(w,w') := | \{ v \in \overline{V}(\F) \mid (v,w), (v,w') \in E \}|
 \end{equation}
  for all $w,w' \in W$.
Then one can partition $W$ into a (standard) finite number of definable subsets $W_1,\ldots,W_m$ such that, for any $1 \leq i,j \leq m$, there is a standard rational $c_{ij}$ such that $\mu(w,w') = (c_{ij} + O( |\F|^{-1/2})) |V|$ for all but $O( |\F|^{-1/2} |W|^2 )$ of the pairs $(w,w') \in W_i \times W_j$.
\end{proposition}

Let us now see how Proposition \ref{First-reg} implies Lemma \ref{algreg-nonst}.  This will be an application of the ``$TT^*$ method'' from harmonic analysis.  Let $F, V, W, E$ be as in Lemma \ref{algreg-nonst}.  By decomposition we may take $\overline{V}, \overline{W}$ to be geometrically irreducible. We use Proposition \ref{First-reg} to partition $W$ into finitely many definable components $W_i$ with the stated properties.  Observe that at least one of the $W_i$ must have a Zariski closure of dimension $\dim(\overline{W})$ (and thus equal to $\overline{W}$, by the irreducibility of $\overline{W}$); and if any component has Zariski closure with dimension strictly less than $\dim(\overline{W})$, it may be safely absorbed into one of the other components without affecting the conclusion of the proposition (thanks to \eqref{fee}).  Thus we may assume that all $W_i$ have the same Zariski closure as $\overline{W}$, which among other things implies that $|W_i| \gg |W|$, thanks to \eqref{fee}.  Next, by passing to just one of these components, we may assume that $m=1$, thus we may reduce without loss of generality to the case where
\begin{equation}\label{muw}
 \mu(w,w') = (c + O( |\F|^{-1/2})) |V|
 \end{equation}
for all but $O( |\F|^{-1} |W|^2 )$ pairs $(w,w') \in W \times W$, and some standard rational $c$.  We can of course assume that $W$ is non-empty, as the claim is vacuously true otherwise.

Now let $f: W \to \ultra \R$ be any nonstandard function of mean zero and bounded in magnitude by $1$, and consider the nonstandard sum
$$ \sum_{v \in \overline{V}(\F)} |\sum_{w \in W} 1_E(v,w) f(w)|^2.$$
We may rewrite this expression as
$$ \sum_{w, w' \in W} f(w) f(w') \mu(w,w').$$
Applying \eqref{muw}, we have $\mu(w,w') = (c + O( |\F|^{-1/2})) |V|$ for all but $O( |\F|^{-1/2} |W|^2 )$ pairs $(w,w')$.  For these exceptional pairs, we use the crude estimate $\mu(w,w') = O(|V|) = (c+O(1)) |V|$.  We conclude that
$$ \sum_{v \in \overline{V}(\F)} |\sum_{w \in W} 1_E(v,w) f(w)|^2 = 
\sum_{w, w' \in W} f(w) f(w') c + O( |\F|^{-1/2} |V| |W|^2 ).$$
But as $f$ was assumed to have mean zero, the first sum vanishes, and so
$$ \sum_{v \in \overline{V}(\F)} |\sum_{w \in W} 1_E(v,w) f(w)|^2 \ll |\F|^{-1/2} |V| |W|^2.$$
In particular
$$ \sum_{v \in V} |\sum_{w \in W} 1_E(v,w) f(w)|^2 \ll |\F|^{-1/2} |V| |W|^2.$$
Next, we apply Proposition \ref{First-reg} again, but with the roles of $V$ and $W$ reversed, to partition $V$ into finitely many definable components $V_1,\ldots,V_m$ such that, for any $1 \leq i,j \leq m$, there is a standard rational $c_{ij}$ such that
$$ |\{ w \in W \mid (v,w), (v',w) \in E \}| = (c_{ij} + O( |\F|^{-1/2})) |W|$$
for all $(v,v') \in V_i \times V_j$ outside of a subvariety of $\overline{V} \times \overline{V}$ of dimension strictly less than $2\dim(\overline{V})$.  As before, we may assume that each $V_i$ has Zariski closure of dimension $\dim(\overline{V})$, so that $|V_i| \gg |V|$ thanks to \eqref{fee}.  By arguing as above, we conclude that
$$ \sum_{w \in W} |\sum_{v \in V_i} 1_E(v,w) g(v)|^2 \ll |\F|^{-1/2} |V|^2 |W|$$
whenever $1 \leq i \leq m$ and $g: V_i \to \ultra \R$ is a nonstandard function of mean zero bounded in magnitude by $1$.  By Cauchy-Schwarz, we conclude that
$$ |\sum_{v \in V_i} \sum_{w \in W} g(v) f(w) 1_E(v,w)| \ll |\F|^{-1/4} |V| |W|$$
whenever $1 \leq i \leq m$ and $f: W \to \ultra \R$, $g: V_i \to \ultra \R$ are nonstandard functions bounded in magnitude by $1$, with at least one of $f, g$ having mean zero.  If we now let $A$, $B$ be arbitrary nonstandard subsets of $V_i, W$ respectively, we can decompose $1_A$ into a constant component $|A|/|V_i|$ and a mean zero component $1_A - |A|/|V_i|$, and similarly decompose $1_B$ into $|B|/|W|$ and $1_B - |B|/|W|$; applying the above estimate to three of the four resulting terms, we conclude that
$$ \sum_{v \in V_i} \sum_{w \in W} 1_A(v) 1_B(w) 1_E(v,w) = \theta_i |A| |B| + O( |\F|^{-1/4} |V| |W| )$$
where $\theta_i := |E|/|V| |W|$ is the density of $E$ in $V_i \times W$.  Note from Theorem \ref{aqe} that $\theta_i$ lies within $O(|\F|^{-1/2})$ of a standard rational, and by replacing $A,B$ with $V_i,W$ we see that $\theta_i$ lies within $O(|\F|^{-1/4})$ of $|E \cap (V_i \times W)|/|V_i| |W|$.  Lemma \ref{algreg-nonst} follows.  

It remains to establish Proposition \ref{First-reg}.  To do this, we will (after some basic reductions) use the Lang-Weil bound (Lemma \ref{sz}) to compute $\mu(w,w')$ in terms of a counting function $c(w,w')$ that counts the number of top-dimensional geometrically irreducible components of a certain variety $U_w \times_{\overline{V}} U'_{w'}$ that are defined over $\F$.

We turn to the details.  Observe that any component $E'$ of $E$ that lies in a proper subvariety of $\overline{V} \times \overline{W}$ has cardinality at most $O(|\F|^{-1} |V| |W|)$ by Theorem \ref{aqe}.  By Chebyshev's inequality, we thus see that
$$ | \{ v \in \overline{V}(\F) \mid (v,w) \in E' \}| = O( |\F|^{-1/2} |V| )$$
for all but $O(|\F|^{-1/2} |W|)$ elements $w \in W$.  From this, we see that $E'$ has a negligible impact on the conclusions of Theorem \ref{First-reg}.  Thus we may freely delete any strict subvariety of $\overline{V} \times \overline{W}$ from $E$ if we wish (i.e. we may work with generic subsets of $\overline{V} \times \overline{W}$).

By Theorem \ref{aqe}, we may write $E$ in the form
\begin{equation}\label{vw}
E = \{ (v,w) \in \overline{V}(\F) \times \overline{V}(\F) \mid \exists t_1,\ldots,t_m \in \F: P_i(v,w,t_i) = 0 \hbox{ for all } i=1,\ldots,m\}
\end{equation}
for some finite collection of polynomials $P_1,\ldots,P_m: \overline{V} \times \overline{W} \times \overline{\F} \to \overline{\F}$ defined over $\F$.  For any one of these polynomials $P_i$, consider the set of pairs $(v,w)$ for which the one-dimensional polynomial $t_i \mapsto P_i(v,w,t_i)$ vanishes.  This is a subvariety of $\overline{V} \times \overline{W}$ defined over $\F$.  If it is all of $\overline{V} \times \overline{W}$, then the polynomial $P_i$ is redundant in \eqref{vw} and can be deleted, so we may assume that it is a strict subvariety of $\overline{V} \times \overline{W}$.  Thus, by deleting all such varieties as discussed previously, we may assume that $E$ actually takes the form
\begin{equation}\label{vw-2}
E = \{ (v,w) \in \Omega(\F) \mid \exists t_1,\ldots,t_m \in \F: P_i(v,w,t_i) = 0 \hbox{ for all } i=1,\ldots,m\}
\end{equation}
where $\Omega$ is a Zariski-dense subvariety of $\overline{V} \times \overline{W}$ defined over $\F$, and the polynomials $t_i \mapsto P_i(v,w,t_i)$ are non-vanishing for any $(v,w) \in \Omega$.  If we then define
$$ U := \{ (v,w,t_1,\ldots,t_m) \in \Omega \times \overline{\F}^m\mid P_i(v,w,t_i) = 0 \hbox{ for all } i=1,\ldots,m \}$$
then $U$ is a quasiprojective subvariety of $\Omega \times F^m$ defined over $\F$, and we have
\begin{equation}\label{elo}
E = \pi(U(\F))
\end{equation}
where $\pi: \Omega \times F^m \to \Omega$ is the projection map, which is \emph{quasi-finite} in the sense that the fibres $\pi^{-1}(\{(v,w)\})$ of $\pi$ are finite (hence zero dimensional) for all $(v,w) \in \Omega$.  In particular, $U$ has dimension at most $\dim(\overline{V}) + \dim(\overline{W})$, thanks to Lemma \ref{proj}.  

For any $w \in W$, we may form the quasiprojective variety
$$ U_w := \{ (v,t) \in \overline{V} \times \overline{\F}^m\mid (v,w,t) \in U \},$$
and for any $w,w' \in W$ we may then form the fibre product
$$ U_w \times_{\overline{V}} U_{w'} := \{ (v,t,t') \in \overline{V} \times \overline{\F}^m \times \overline{\F}^m\mid (v,t) \in U_w, (v,t') \in U_{w'} \}.$$
These are quasiprojective varieties which are quasi-finite over $\overline{V}$ and so have dimension at most $\dim(\overline{V})$.  We form the counting functions
$$ \nu_{w,w'}(v) := |\{ (t,t') \in \F^m \times \F^m\mid (v,t,t') \in U_w \times_{\overline{V}} U_{w'} \}|.$$
Then $\nu_{w,w'}(v)$ is finite for all $v \in \overline{V}(\F)$, and thus (by countable saturation\footnote{One can also use bounds on the degrees of the algebraic varieties involved here, if desired.}, see Lemma \ref{countsat}) is uniformly bounded.  On the other hand, from \eqref{elo} we have $\nu_{w,w'}(v) \neq 0$ if and only if $(v,w), (v,w') \in E$.  Thus by \eqref{mud}, we have
$$ \mu(w,w') = \sum_{v \in \overline{V}(\F)} 1_{\nu_{w,w'}(v) \neq 0}.$$
As $\nu_{w,w'}(v)$ is uniformly bounded, we can express $1_{c_{w,w'}(v) \neq 0}$ as a standard linear combination of $\nu_{w,w'}(v)^k$ for finitely many standard natural numbers $k$, and so $\mu(w,w')$ is a standard linear combination of the moments $\sum_{v \in \overline{V}(\F)} \nu_{w,w'}(v)^k$.  Thus, to prove Proposition \ref{First-reg}, it suffices (by Theorem \ref{aqe}) to show that for each standard natural number $k$, we can partition $W$ into a (standard) finite number of definable subsets $W_1,\ldots,W_m$ such that, for any $1 \leq i,j \leq m$, there is a standard rational $c_{ijk}$ such that 
$$ \sum_{v \in \overline{V}(\F)} \nu_{w,w'}(v)^k = (c_{ijk} + O( |\F|^{-1/2})) |\F|^{\dim(\overline{V})}$$
for all but $O( |\F|^{-1/2} |W|^2 )$ of the pairs $(w,w') \in W_i \times W_j$.

Observe that
$$ \sum_{v \in \overline{V}(\F)} \nu_{w,w'}(v)^k = |(U_{w,k} \times_{\overline{V}} U_{w',k})(\F)|$$
where $U_{w,k}$ is the $k$-fold fibre product of $U_w$ over $\overline{V}$,
$$ U_{w,k} := \{ (v,t_1,\ldots,t_k) \in \overline{V} \times (\overline{\F}^m)^k\mid (v,w,t_i) \in U \hbox{ for all } i=1,\ldots,k \},$$
and 
$$ U_{w,k} \times_{\overline{V}} U_{w',k} := \{ (v,t,t') \in \overline{V} \times (\overline{\F}^m)^k \times (\overline{\F}^m)^k\mid (v,t) \in U_{w,k}, (v,t') \in U_{w',k} \}.$$
By Lemma \ref{sz}, we conclude that
$$ \sum_{v \in \overline{V}(\F)} \nu_{w,w'}(v)^k = (c(w,w') + O(|\F|^{-1/2})) |\F|^{\dim(\overline{V})}$$
where $c(w,w')$ are the number of geometrically irreducible components of $U_{w,k} \times_{\overline{V}} U_{w',k}$ which are defined over $\F$.  To establish Proposition \ref{First-reg}, it thus suffices to establish the following claim:

\begin{proposition}[Second reduction]\label{Second-reg}  Let $\F$ be a nonstandard finite field of characteristic zero, and let $V, W, W'$ be definable sets over $\F$.  Let $m$ be a standard natural number, let $U$ be a subvariety of $\overline{V} \times \overline{W} \times \overline{\F}^m$ defined over $\F$ which is quasi-finite over $\overline{V} \times \overline{W}$, and let $U'$ be a subvariety of $\overline{V} \times \overline{W'} \times \overline{\F}^m$ defined over $\F$ which is quasi-finite over $\overline{V} \times \overline{W'}$.  For any $w \in W$ and $w' \in W'$, set
$$ U_w := \{ (v,t) \in \overline{V} \times \overline{\F}^m\mid (v,w,t) \in U \},$$
$$ U'_{w'} := \{ (v,t') \in \overline{V} \times \overline{\F}^m\mid (v,w',t') \in U' \}$$
and
$$ U_w \times_{\overline{V}} U'_{w'} := \{ (v,t,t') \in \overline{V} \times \overline{\F}^m \times \overline{\F}^m\mid (v,t) \in U_w, (v,t') \in U'_{w'} \}$$
and let $c(w,w')$ be the number of $\dim(\overline{V})$-dimensional geometrically irreducible components of $U_w \times_{\overline{V}} U'_{w'}$ that are defined over $\F$.  Then one can partition $W$ into a (standard) finite number of definable subsets $W_1,\ldots,W_m$ and $W'$ into a (standard) finite number of definable subsets $W'_1,\ldots,W'_{m'}$ such that, for any $1 \leq i \leq m$ and $1 \leq i' \leq m'$, there is a standard natural number $c_{ii'}$ such that $c(w,w') = c_{ii'}$ for all but $O(|\F|^{-1/2} |W| |W'|)$ of the pairs $(w,w') \in W_i \times W'_{i'}$.
\end{proposition}

Indeed, after replacing $U$ and $U'$ with the $k$-fold fibre product
$$ U_k := \{ (v,w,t_1,\ldots,t_k) \in \Omega \times (F^m)^k\mid (v,w,t_i) \in U \hbox{ for all } i=1,\ldots,k \} $$
and applying the above proposition once for each $k$ (and with $W=W'$), we obtain the required structural decomposition of $\sum_{v \in V} \nu_{w,w'}(v)^k$ (after intersecting together all the partitions obtained).

It remains to establish Proposition \ref{Second-reg}.  The next reduction is to remove the requirement that the sets $W_1,\ldots,W_m$ and $W'_1,\ldots,W'_{m'}$ in the partitions of $W,W'$ are themselves definable or even nonstandard, at the slight cost of upgrading the $O(|\F|^{-1/2} |W| |W'|)$ error to $O(|\F|^{-1} |W| |W'|)$.  In other words, we will deduce Proposition \ref{Second-reg} from the following assertion.

\begin{proposition}[Third reduction]\label{Third-reg}  Let $\F,V,W,W'm,U,U',c$ be as in Proposition \ref{Second-reg}.  Then one can partition $W$ into a (standard) finite number of subsets $W_1,\ldots,W_m$, and $W'$ into a (standard) finite number of subsets $W'_1,\ldots,W'_{m'}$, such that, for any $1 \leq i \leq m$ and $1 \leq i' \leq m'$, there is a standard natural number $c_{ii'}$ such that $c(w,w') = c_{ii'}$ for all\footnote{Because $W_i$ and $W'_{i'}$ are not assumed to be nonstandard sets, one has to be careful about what this means, since $W_i$ and $W'_{i'}$ need not have a well-defined cardinality.  What we mean here is that the set of exceptions $(w,w') \in W_i \times W'_{i'}$ for which $c(w,w') \neq c_{ii'}$ has an \emph{outer cardinality} of $O(|\F|^{-1} |W||W'|)$, in the sense that it is \emph{contained} in a nonstandard set of cardinality $O(|\F|^{-1} |W| |W'|)$.} but $O(|\F|^{-1} |W||W'|)$ of the pairs $(w,w') \in W_i \times W'_{i'}$.  
\end{proposition}

The freedom to allow the partitions of $W,W'$ to not be definable or nonstandard will be of technical importance later in the argument, when we will use the axiom of choice to force a ``coordinate system'' on various relevant objects needed to compute $c(w,w')$; such coordinate systems will not necessarily be ``definable'' or even ``nonstandard'', but thanks to the above reduction, this will not be an issue.

Let us assume Proposition \ref{Third-reg} for now and see how it implies Proposition \ref{Second-reg}.  The key observation is that the level sets $\{ (w,w') \in W \times W'\mid c(w,w') = c_0\}$ of the function $c$ are themselves definable subsets over $\F$, as the property of a set cut out by a number of polynomial equations being geometrically irreducible and definable over $\F$ can be expressed as a first-order sentence in the coefficients of these equations.  (Note from countable saturation that that the complexities of all the irreducible varieties involved in a decomposition of a given variety of bounded complexity is necessarily bounded, and so the first-order sentence involved is finite in length.)  Next, we apply Proposition \ref{Third-reg} to partition $W$ and $W'$ into finitely many pieces $W_1,\ldots,W_m$ and $W'_1,\ldots,W'_{m'}$, not necessarily definable or nonstandard.   By hypothesis, we can find a nonstandard subset $\Sigma$ of $W \times W'$ of cardinality $O( |\F|^{-1} |W| |W'|)$ with the property that whenever $1 \leq i \leq m$, $1 \leq i' \leq m'$ and $(w,w') \in (W_i \times W'_{i'}) \backslash \Sigma$, we have $c(w,w') = c_{ii'}$.

By Markov's inequality, we see that outside of a exceptional subset $E'$ of $W'$ of cardinality $O( |\F|^{-1/2} |W'|)$, we have $(w,w') \not \in \Sigma$ for all but $O(|\F|^{-1/2} |W|)$ elements of $w \in W$.  For any $1 \leq i' \leq m$, we set $w'_{i'}$ to be an arbitrarily chosen element of $W'_{i'} \backslash E'$ if this set is non-empty, or an arbitrarily chosen element of $W$ otherwise.  By construction, we see that for all $w' \in W' \backslash E'$, there exists $1 \leq i' \leq m$ such that 
$$ c(w,w') = c(w,w'_{i'})$$
for 
$$ c(w,w'_{i'}) = c(w,w')$$
for all but $O( |\F|^{-1/2} |W| )$ values of $w \in W$.  

For each $1 \leq i' \leq m'$ and natural number $c_0$, the level sets $\{ w \in W\mid c(w,w'_{i'}) = c_0 \}$ is a definable set.  These definable sets generate a partition of $W$ into finitely many definable subsets $\tilde W_1,\ldots,\tilde W_{\tilde m}$.  By construction, we see that for all but $O(|\F|^{-1/2} |W'|)$ values of $w' \in W'$, the function $w \mapsto c(w,w')$ is constant outside of a set of cardinality $O(|\F|^{-1/2} |W|)$ on each of the $\tilde W_1,\ldots,\tilde W_{\tilde m}$.  By symmetry, we may also partition $W'$ into finitely many definable subsets $\tilde W'_1,\ldots,\tilde W'_{m'}$ with the property that for all but $O(|\F|^{-1/2} |W|)$ values of $w \in W$, the function $w' \mapsto c(w,w')$ is constant outside of a set of cardinality $O(|\F|^{-1/2} |W'|)$ on each of the $\tilde W'_1,\ldots,\tilde W'_{m'}$.  

By Theorem \ref{aqe}, the definable sets $\tilde W_i$ either have cardinality $\gg |W|$ or $O(|\F|^{-1} |W|)$.  Any sets of the latter form can be harmlessly absorbed into one of the sets of the former form, so we may assume that all sets $\tilde W_i$ have cardinality $\gg |W|$.  Similarly we may assume that all the sets $\tilde W'_{i'}$ have cardinality $\gg |W'|$.

Now we double-count.  for any $1 \leq i \leq \tilde m$ and $1 \leq i' \leq \tilde m'$, we see that
$$ | \{ (w_1,w_2,w'_1, w'_2) \in \tilde W_{i} \times \tilde W_{i} \times \tilde W'_{i'} \times \tilde W'_{i'}\mid c(w_1,w'_1) \neq c(w_2,w'_1) \}| \ll |\F|^{-1/2} |W|^2 |W'|^2$$
(because of the constancy properties of $w \mapsto c(w,w'_1)$ for most $w'_1$) and
$$ | \{ (w_1,w_2,w'_1, w'_2) \in \tilde W_{i} \times \tilde W_{i} \times \tilde W'_{i'} \times \tilde W'_{i'}\mid c(w_2,w'_1) \neq c(w_2,w'_2) \}| \ll |\F|^{-1/2} |W|^2 |W'|^2$$
(because of the constancy properties of $w' \mapsto c(w_2,w')$ for most $w_2$) and thus
$$ | \{ (w_1,w_2,w'_1, w'_2) \in W'_{i'} \times W'_{i'} \times W''_{i''} \times W''_{i''}\mid c(w_1,w'_1) = c(w_2,w'_1) = c(w_2,w'_2) \}| = |\tilde W_{i}|^2 |\tilde W'_{i'}|^2 - O( |\F|^{-1/2} |W|^4 )$$
and thus by the pigeonhole principle we can find $w_2, w'_2$ such that
$$ | \{ (w_1,w'_1) \in \tilde W_{i} \times \tilde W'_{i'}\mid c(w_1,w'_1) = c(w_2,w'_2) \}| = |\tilde W_{i}| |\tilde W'_{i'}| - O( |\F|^{-1/2} |W| |W'| )$$
and so $c$ is constant on $\tilde W_{i} \times \tilde W'_{i'}$ outside of a set of cardinality $O( |\F|^{-1/2} |W| |W'| )$, and the claim follows.

It remains to establish Proposition \ref{Third-reg}.  We can now remove all references to definability by passing to Zariski closures, and reduce to establishing the following fact:

\begin{proposition}[Fourth reduction]\label{Four-reg}   Let $\F$ be a nonstandard finite field of characteristic zero, and let $V, W, W'$ be affine varieties defined over $\F$.  Let $d$ be a natural number, and let $U, U'$ be subvarieties of $V \times W \times \overline{\F}^d$ and $V \times W' \times \overline{\F}^d$ respectively which are defined over $\F$ and quasi-finite over $V \times W$ and $V \times W'$ respectively.  For any $w \in W$ and $w' \in W'$, set
\begin{align*}
 U_w &:= \{ (v,t) \in V \times \overline{\F}^d\mid (v,w,t) \in U \},\\
U'_{w'} &:= \{ (v,t') \in V \times \overline{\F}^d\mid (v,w',t') \in U' \}\\
 U_w \times_V U'_{w'} &:= \{ (v,t,t') \in V \times \overline{\F}^d \times \overline{\F}^d: (v,t) \in U_w, (v,t') \in U'_{w'} \},
\end{align*}
and let $c(w,w')$ be the number of $\dim(\overline{V})$-dimensional geometrically irreducible components of $U_w \times_{\overline{V}} U'_{w'}$ that are defined over $\F$.  Then one can partition $W$ into a (standard) finite number of subsets $W_1,\ldots,W_m$ (not necessarily definable or nonstandard) and $W'$ into a (standard) finite number of subsets $W'_1,\ldots,W'_{m'}$ such that, for any $1 \leq i \leq m$ and $1 \leq i' \leq m'$, the function $c$ is generically constant on $W_i \times W'_{i'}$ (i.e. it is constant in $W_i \times W'_{i'}$ outside of a subvariety of $W \times W'$ of dimension strictly smaller than $\dim(W)+\dim(W')$).
\end{proposition}

Indeed, Proposition \ref{Third-reg} follows from Proposition \ref{Four-reg} by specialising to a definable subset over $\F$ and using Lemma \ref{sz} to control the (outer cardinality of the) exceptional set.

Now we prove Proposition \ref{Four-reg}.  Our strategy is to work generically and improve the nature of the varieties $U_w, U'_{w'}$ lying above $V$, until they become \emph{finite \'etale covers} of certain Zariski-dense subvarieties of $V$.  At that point, we can use the theory of the \'etale fundamental group (Appendix \ref{etale-app}) to obtain the required local generic constancy of the counting function $c$.

We turn to the details.  First, we may decompose $V$ into geometrically irreducible components.  Any component which has dimension less than $\dim(\overline{V})$, or which is not defined over $\F$, gives a zero contribution to $c$.  Thus we may discard these components, and reduce to the case when $V$ is a single geometrically irreducible affine variety defined over $\F$.  For similar reasons, we may also reduce to the case where $W, W'$ are geometrically irreducible affine varieties defined over $\F$.

Next, we observe that we may freely delete any closed subvariety from $U$ of dimension at most $\dim(V)+\dim(W)-1$ without affecting the conclusion of the proposition.  Indeed, for generic $w \in W$, this deletion will only remove a set of dimension at most $\dim(V)-1$ from $U_w$ and hence from $U_w \times_V U'_{w'}$ for any $w' \in W'$, and hence will not affect $c(w,w')$ for generic $w$.  Similarly, we may delete any closed subvariety from $U'$ of dimension at most $\dim(V)+\dim(W')-1$.

Next, we work to make $U$ smooth.  Given an affine variety $V \subset k^n$ and a point $p$ in $V$, define the \emph{tangent space} $T_p V$ of $V$ at $p$ to be the vector space ${\mathfrak m}/{\mathfrak m}^2$, where ${\mathfrak m}$ is the space of polynomials in $k[V]$ that vanish at $p$.  We say that $p$ is a \emph{smooth} point of $V$ if $T_p V$ has dimension $\dim(V)$, and a \emph{singular} point otherwise.  A quasiprojective variety $U$ is said to be \emph{smooth} if every point of $U$ is a smooth point of $\overline{U}$.  Note that a point that lies in two or more components of an affine variety cannot be a smooth point of that variety, so the irreducible components of a smooth quasiprojective variety are necessarily disjoint.

The variety $U$ has dimension at most $\dim(V)+\dim(W)$.  As is well known, the set of singular points of $U$ must have dimension strictly less than this (see e.g. \cite[Theorem 5.6.8]{taylor}); here is one place where we crucially use the hypothesis that $\F$ has characteristic zero.  By deleting these points, we may thus assume that $U$ is smooth and has dimension exactly $\dim(V)+\dim(W)$; in particular, the irreducible components of $U$ are now disjoint.  Similarly, we may assume that $U'$ is smooth and has dimension $\dim(V)+\dim(W')$.  In particular, the projections of $U$ and $U'$ to $V \times W$ and $V \times W'$ respectively are now \emph{dominant maps}, in the sense that their images are Zariski dense.

By again using the hypothesis that $\F$ has characteristic zero, the set of points $u \in U$ where the derivative $d\pi(u)$ of the projection map $\pi: U \to V \times W$ does not have full rank, has dimension strictly less than $\dim(V)+\dim(W)$ (see e.g. \cite[III 10.7]{hart}), so by deleting these points we may assume that $d\phi$ is everywhere non-singular, or in other words that $\pi: U \to V \times W$ is an \emph{\'etale map}.   Similarly, we may assume that the projection $\pi': U' \to V \times W'$ is also \'etale.

The projections $\pi, \pi'$ are currently quasi-finite and \'etale.  We will need to upgrade the quasi-finiteness property to the stronger property of \emph{finiteness}.  We quickly review the relevant definitions:

\begin{definition}  Let $V \subset k^n$ be quasiprojective variety over an algebraically closed field $k$.  A quasiprojective variety is \emph{abstractly affine} if there is a regular isomorphism between it and an affine variety.  A regular morphism $\phi: V \to W$ is \emph{finite} if one can cover $W$ by open, abstractly affine subvarieties $W_i$, such that $\phi^{-1}(W_i)$ is also abstractly affine, and the ring $k[\phi^{-1}(W_i)]$ is a finite $k[W_i]$-algebra (where we use $\phi$ to pull $k[W_i]$ back into $k[\phi^{-1}(W_i)]$ in the obvious manner).
\end{definition}

\begin{example}  The inclusion of $k \backslash \{0\}$ into $k$ is quasi-finite and \'etale, but not finite, because the ring $k[ k \backslash \{0\} ] = k[x, \frac{1}{x}]$ is not finite over $k[k] = k[x]$.  A finite morphism in algebraic geometry is analogous to the notion of a covering space (with finite fibres) in topology; note for instance that the inclusion of $\C \backslash \{0\}$ into $\C$ is also not a covering space.
\end{example}

We have the following basic fact:

\begin{lemma}\label{lemf}  Let $\phi: U \to V$ be a quasi-finite regular morphism between two quasiprojective varieties $U, V$ which is dominant.  Then there exists an open dense subvariety $V'$ of $V$ such that the restricted map $\phi: \phi^{-1}(V') \to V'$ is finite.
\end{lemma}

\begin{proof}  We adapt the proof of \cite[Theorem I.5.3.6]{shaf}.  By passing to an open dense abstractly affine variety of $V$ (such as $\overline{V}$ with a codimension one closed subvariety removed) we may assume that $V$ is abstractly affine.  Let $k(V)$ denote the field of fractions of $k[V]$ (i.e. the rational functions on $V$), and similarly define $k(U)$.  By the hypotheses on $\phi$, $k(U)$ is an algebraic extension of $k(V)$ (pulled back by $\phi$, of course), and so on clearing denominators one can find a finite set of generators for $k[U]$ that become integral over $k[V]$ after multiplying by a non-zero regular function $f$ in $k[V]$.  Removing the zeroes of $f$ from $V$ (thus adding $1/f$ to $k[V]$, and keeping $V$ abstractly affine), and removing the corresponding preimage from $U$, we obtain the claim.
\end{proof}

Applying this lemma to the map $\pi: U \to V \times W$, we can remove a lower-dimensional piece from $U$ and assume without loss of generality that $\pi: U \to \pi(U)$ is not just \'etale, but is finite \'etale, thus the smooth variety $U$ is a finite \'etale cover of $\pi(U)$.  Similarly, we may make $\pi': U' \to \pi'(U')$ a finite \'etale covering map.

Since $\pi: U \to \pi(U)$ is a finite \'etale covering map of smooth varieties, the restriction $\pi_V: U_w \to \pi_V(U_w)$ is also a finite \'etale map of smooth varieties for any $w \in W$, where $\pi_V: V \times F^m \to V$ is the projection onto $V$.  Similarly, $\pi_V: U'_{w'} \to \pi_V(U'_{w'})$ is finite \'etale for any $w' \in W'$, which implies that the fibre product $\pi_V: U_w \times_V U'_{w'} \to \pi_V(U_w \times_V U'_{w'})$ is also finite \'etale.  (Here we have used the fact that finiteness and the \'etale property are both preserved with respect to base change and composition; see e.g. \cite[Propositions I.1.3, I.3.3]{milne-etale}.)

The set $\phi(U)$ is a Zariski-dense subvariety of $V \times W$.  Applying Lemma \ref{fubini}, we conclude that for generic $v \in V$, one has $(v,w) \in \phi(U)$ (or equivalently, $v \in \phi_V(U_w)$) for generic $w \in W$.  Similarly, for generic $v \in V$, one has $v \in \phi_V(U'_{w'})$ for generic $w' \in W'$.  Thus, we may find a point $p \in V$ such that $p \in \phi_V(U_w \times_V U'_{w'}) = \phi_V(U_w) \cap \phi_V(U'_{w'})$ for generic $w \in W$ and $w' \in W'$.  Indeed, by Lemma \ref{sz}, we may take $p$ to be an $F$-point of $V$.

Fix this point $p$.  For generic $w \in W$ and $w' \in W'$, the fibre of of $U_w \times_V U'_{w'}$ over $p$ is non-empty, and may be identified with the Cartesian products $S_w \times S'_{w'}$, where $S_w$, $S'_{w'}$ are the finite sets
$$ S_w := \{ t \in \overline{\F}^d\mid (v,w,t) \in U \}$$
and
$$ S'_{w'} := \{ t \in \overline{\F}^d\mid (v,w',t) \in U' \}.$$
As $U, U'$ are defined over $\F$, and $p$ is an $F$-point of $V$, $S_w$ and $S'_{w'}$ are defined over $\F$.  In particular, the nonstandard Frobenius map $\Frob_\F$ acts on $S_w$ and $S'_{w'}$, and thus also acts on the product $S_w \times S'_{w'}$ by the diagonal action.

As $\phi_V: U_w \times_V U'_{w'} \to \phi_V(U_w) \cap \phi_V(U'_{w'})$ is a finite \'etale covering, the \'etale fundamental group $\pi_1( \phi_V(U_w) \cap \phi_V(U'_{w'}), p)$ acts on the fibre $S_w \times S'_{w'}$, by a product of its actions on the individual fibres $S_w$ and $S'_{w'}$; see Appendix \ref{etale-app}.  As noted in that appendix, each orbit of this action is the fibre of exactly one of the irreducible components of $U_w \times_V U'_{w'}$. Thus, the number $c(w,w')$ of such components that are defined over $\F$ is equal to the number of orbits of this action that are invariant with respect to the nonstandard Frobenius action.

We view this number as a combinatorial quantity, which currently depends (for generic $w,w'$) in a rather entangled fashion on several objects: the variety $\phi_V(U_w) \cap \phi_V(U'_{w'})$ (and more specifically, its \'etale fundamental group over $p$); the fibres $S_w$ and $S'_{w'}$; the action of the \'etale fundamental group on these fibres; and the action of the Frobenius map on these fibres.  Our goal is to decouple the role of $w$ and $w'$ in forming $c(w,w')$, so that (after partitioning $W$ and $W'$ into finitely many subsets, and working on a single subset of $W$ and a single subset of $W'$), the quantity $c(w,w')$ becomes generically constant.

We achieve this as follows.  The first step is to (crudely) ``trivialise the bundles'' of $U$ over $V \times W$ and $U'$ over $V \times W'$ in a set-theoretic sense.  Observe that for generic $w$, the fibre $S_w$ has constant cardinality $M$ for some standard natural number $M$ (indeed, $M$ is just the degree of $U$ divided by the degree of $V \times W$).  Thus, by the axiom of choice, we may enumerate $S_w = \{ t_{w,1}, \ldots, t_{w,M} \}$.  By fixing such an enumeration, we can thus (non-canonically) identify $S_w$ with $\{1,\ldots,M\}$ for generic $w$.  Similarly, we may non-canonically identify $S'_{w'}$ with $\{1,\ldots,M'\}$ for generic $w'$ and some standard natural number $M'$ by using the axiom of choice to select an enumeration $S'_{w'} = \{ t'_{w',1},\ldots,t'_{w',M'}\}$.  Note that as we appeal to the axiom of choice here to build this enumeration, we do not claim or expect these identifications to be definable, or even nonstandard; but thanks to our reduction of Proposition \ref{Second-reg} to Proposition \ref{Third-reg}, such definability and nonstandardness properties will not be needed\footnote{One could avoid the appeal to the axiom of choice here by working with all enumerations at once, and quotienting out the objects constructed at the end of the day the equivalence relation given by all possible relabelings; similarly for some further invocations of the axiom of choice later in this argument.  However, we will not choose this ``coordinate-free'' route here as it requires the use of more complicated notation, opting instead for a less elegant, but more direct ``coordinate-heavy'' approach, which is more unnatural from an algebraic geometry perspective, but more convenient from a combinatorial one.}.  

We now fix the above enumerations of $S_w$ and $S'_{w'}$.  For generic $w$, the action of the nonstandard Frobenius map $\Frob_\F$ on $S_w \equiv \{1,\ldots,M\}$ is now given by a permutation $\sigma_w$ in the symmetric group $\operatorname{Sym}(M)$ on $M$ elements.  This permutation depends on $w$, so by partioning $W$ into finitely many subsets (which need not be definable or nonstandard), we can ensure that the map $w \mapsto \sigma_w$ is constant on each such subset.  We now pass to one of these subsets of $W$; thus for generic $w$ in this subset, the action of $\Frob_\F$ on $S_w \equiv \{1,\ldots,M\}$ is now independent of $w$, when viewed in coordinates.  Similarly, by partitioning $W'$ into finitely many subsets and passing to any one of these subsets, we may assume that for generic $w'$ in this subset, the action of $\Frob_\F$ on $S'_{w'} \equiv \{1,\ldots,M'\}$ is independent of $w'$.  Thus, for generic $w,w'$ in the respective subsets of $W,W'$, the product action of $\Frob_\F$ on $S_w \times S'_{w'} \equiv \{1,\ldots,M\} \times \{1,\ldots,M'\}$ is independent of both $w$ and $w'$.  To obtain the desired local constancy of $c(w,w')$, it thus suffices to show (perhaps after further finite partition of $W$ and $W'$) that for generic $w,w'$ in their respective subsets of $W,W'$, the set of orbits of the \'etale fundamental group $\pi_1( \phi_V(U_w) \cap \phi_V(U'_{w'}), p)$ on $S_w \times S_{w'} \equiv \{1,\ldots,M\} \times \{1,\ldots,M'\}$ is actually independent of the choice of $w$ and $w'$.  

The main difficulty here, of course, is that the group $\pi_1( \phi_V(U_w) \cap \phi_V(U'_{w'}), p)$ depends on both $w$ and $w'$ in a coupled fashion.  To decouple the role of $w$ and $w'$ here, we would like to use the \'etale van Kampen theorem (Theorem \ref{evk}), but first we must understand how the sets $\phi_V(U_w)$ and $\phi_V(U'_{w'})$ intersect each other.

The set $\phi(U)$ is an open dense subvariety of $V \times W$, and can thus be written as $\phi(U) = (V \times W) \backslash \Sigma$ for some closed subvariety of $V \times W$ of dimension at most $\dim(V)+\dim(W)-1$.  We split $\Sigma = (\Sigma_0 \times W) \cup \Sigma_1$, where $\Sigma_0 \times W$ is the union of all the irreducible components of $\Sigma$ that are of the form $H \times W$ for some closed subvariety $H$ of $V$, and $\Sigma_1$ is the union of all the other irreducible components of $\Sigma$.  Informally, $\Sigma_0$ represents the portion of $\Sigma$ that does not depend on the $w \in W$ coordinate, while $\Sigma_1$ represents the portion which is non-trivially dependent on this coordinate.  Note that $\Sigma_0$ has dimension at most $\dim(V)-1$, and $\Sigma_1$ has dimension at most $\dim(V)+\dim(W)-1$.  We then have
$$ \phi_V(U_w) = V \backslash (\Sigma_0 \cup \Sigma_{1,w})$$
for any $w \in W$, where
$$ \Sigma_{1,w} := \{ v \in V\mid (v,w) \in \Sigma_1 \}$$
is a slice of $\Sigma_{1}$.  Similarly, we may write
$$ \phi_V(U'_{w'}) = V \backslash (\Sigma'_0 \cup \Sigma'_{1,w'})$$
for all $w' \in W'$, where $\Sigma'_0$ is a closed subvariety of $V$ of dimension at most $\dim(V)-1$, $\Sigma'_1$ is a closed subvariety of $V \times W'$ of dimension at most $\dim(V)+\dim(W')-1$ consisting entirely of components that are not of the form $H \times W'$ for any $H$, and 
$$ \Sigma'_{1,w'} := \{ v \in V\mid (v,w') \in \Sigma'_1 \}$$
is a slice of $\Sigma'_{1}$.  

We now have
$$ 
\phi_V(U_w) \cap \phi_V(U'_{w'}) = V' \backslash (\Sigma_{1,w} \cup \Sigma'_{1,w'})
$$
where $V'$ is the open dense subvariety of $V$ defined by
$$ V' := V \backslash (\Sigma_0 \cup \Sigma'_0).$$

We can now apply the \'etale van Kampen theorem (Theorem \ref{evk}) and conclude that for generic $w,w'$, $\pi_1( \phi_V(U_w) \cap \phi_V(U'_{w'}), p )$ surjects onto the fibre product of $\pi_1( V' \backslash \Sigma_{1,w}, p )$ and $\pi_1( V' \backslash \Sigma'_{1,w'}, p )$ over $\pi_1( V' \backslash (\Sigma_{1,w} \cap \Sigma'_{1,w'}), p )$ (with respect to the obvious homomorphisms between these groups). 

Next, we make the crucial observation that for generic $w, w'$, the set $\Sigma_{1,w} \cap \Sigma'_{1,w'}$ has dimension at most $\dim(V)-2$ (i.e. it has codimension at least $2$ in $V$).  Indeed, for generic $w$, $\Sigma_{1,w}$ has dimension at least $\dim(V)-1$.  Given a top dimensional component $H_w$ of $\Sigma_{1,w}$, $H_w$ will not be a component of $\Sigma'_{1,w'}$ for generic $w'$ unless $H_w \times W'$ is contained in $\Sigma'_1$, which contradicts the construction of $\Sigma'_1$.  Thus the intersection of $H_w$ with any component of $\Sigma'_{1,w'}$ will generically have dimension at most $\dim(V)-2$, and the claim follows\footnote{Here we have used the obvious fact that the set of pairs $(w,w')$ for which $\Sigma_{1,w} \cap \Sigma'_{1,w'}$ has dimension more than $\dim(V)-2$ is a constructible set, so that we may apply Lemma \ref{fubini}.}.  

Because the codimension of $\Sigma_{1,w} \cap \Sigma'_{1,w'}$ is generically at least $2$, we may invoke Lemma \ref{highd} and conclude the isomorphism
$$\pi_1( V' \backslash (\Sigma_{1,w} \cap \Sigma'_{1,w'}), p ) \equiv \pi_1( V', p ) $$
for generic $w,w'$, using the obvious homomorphism from the former group to the latter.  We conclude that for generic $w,w'$, $\pi_1( \phi_V(U_w) \cap \phi_V(U'_{w'}), p )$ surjects onto the fibre product of $\pi_1( V' \backslash \Sigma_{1,w}, p )$ and $\pi_1( V' \backslash \Sigma'_{1,w'}, p )$ over $\pi_1( V', p )$, with respect to the obvious homomorphisms between these groups.  On the other hand, observe that $\pi_1( V' \backslash \Sigma_{1,w}, p )$ acts on $S_w$ (as $U_w$ is a finite \'etale covering over $V' \backslash \Sigma_{1,w}$), and $\pi_1( V' \backslash \Sigma'_{1,w'}, p )$ acts on $S'_{w'}$, and so the fibre product of $\pi_1( V' \backslash \Sigma_{1,w}, p )$ and $\pi_1( V' \backslash \Sigma'_{1,w'}, p )$ over $\pi_1( V', p )$ acts on $S_w \times S'_{w'}$ by the product action.  From the compatibility of the \'etale fundamental group actions on fibres (see Appendix \ref{etale-app}) we see that the action of $\pi_1( \phi_V(U_w) \cap \phi_V(U'_{w'}), p )$ on $S_w \times S'_{w'}$ factors through this product action.  From the surjectivity mentioned earlier, we conclude an important fact:

\begin{proposition} For generic $w, w'$, the set of orbits of $\pi_1( \phi_V(U_w) \cap \phi_V(U'_{w'}), p )$ on $S_w \times S'_{w'}$ is equal to the set of orbits of the fibre product of $\pi_1( V' \backslash \Sigma_{1,w}, p )$ and $\pi_1( V' \backslash \Sigma'_{1,w'}, p )$ over $\pi_1( V', p )$.
\end{proposition}

In view of this proposition, the only remaining task needed to establish Proposition \ref{Four-reg} (and thus Lemma \ref{algreg-nonst}) is to show that, after further finite subdivision of $W$ and $W'$ into subsets, and for generic $w,w'$ in respective subsets of $W,W'$, the set of orbits of the fibre product of $\pi_1( V' \backslash \Sigma_{1,w}, p )$ and $\pi_1( V' \backslash \Sigma'_{1,w'}, p )$ over $\pi_1( V', p )$ on $S_w \times S'_{w'} \equiv \{1,\ldots,m\} \times \{1,\ldots,m'\}$ is independent of both $w$ and $w'$.

At first glance, it may seem that one would need a rather precise understanding of the nature of the \'etale fundamental group $\pi_1( V' \backslash \Sigma_{1,w}, p )$, and how it sits over $\pi_1( V', p )$ by the obvious surjective homomorphism, and how it acts on $S_w$.  Fortunately, however, we only need a small amount of information on this group and this action.  Namely, let $H_w$ be the kernel of the surjective homomorphism from $\pi_1( V' \backslash \Sigma_{1,w}, p )$ to $\pi_1( V', p )$.  This normal subgroup of $\pi_1( V' \backslash \Sigma_{1,w}, p )$ acts on $S_w \equiv \{1,\ldots,M\}$; let $\sim_w$ be the equivalence relation on $\{1,\ldots,M\}$ induced by this action (so that two elements of $\{1,\ldots,M\}$ are equivalent by $\sim_w$ if there is an element of $H_w$ that moves one to the other).  There are only finitely many possibilities for this equivalence relation, so by partitioning $W$ further into finitely many subsets and passing to one of these subsets, we may assume that $\sim_w = \sim$ is independent of $w$ for generic $w$ in this subset.  Similarly, letting $H'_{w'}$ denote the kernel of the homomorphism from $\pi_1( V' \backslash \Sigma'_{1,w'}, p )$ to $\pi_1( V', p )$, we may assume that the equivalence relation $\sim'$ on $\{1,\ldots,M'\} \equiv S_{w'}$ induced by $H'_{w'}$ is generically independent of $w'$, after passing to one of the finitely many subsets partitioning $W'$.

For generic $w$, the action of $\pi_1( V' \backslash \Sigma_{1,w}, p )$ on $\{1,\ldots,M\}$ now descends to an action $\rho_w$ of the quotient group $\pi_1(V',p)$ on the quotient space $\{1,\ldots,M\}/\sim$, and similarly for generic $w'$ we have an action $\rho'_{w'}$ of $\pi_1(V',p)$ on $\{1,\ldots,M'\}/\sim'$.  An orbit of the fibre product $\pi_1( V' \backslash \Sigma_{1,w}, p )$ and $\pi_1( V' \backslash \Sigma'_{1,w'}, p )$ over $\pi_1( V', p )$ in $\{1,\ldots,M\} \times \{1,\ldots,M'\}$ can now be written in the form
$$ \bigcup_{g \in \pi_1(V',p)} \Pi^{-1}( \rho_w(g)x ) \times (\Pi')^{-1}(\rho'_{w'}(g) y)$$
where $x$ is a point in $\{1,\ldots,M\}/\sim$, $y$ is a point in $\{1,\ldots,M'\}/\sim'$, and $\Pi: \{1,\ldots,M\} \to \{1,\ldots,M\}/\sim$ and
$\Pi': \{1,\ldots,M'\} \to \{1,\ldots,M'\}/\sim'$ are the quotient maps.  Such orbits are almost independent of $w$ and $w'$, save for the need to specify the actions $\rho_w$, $\rho'_{w'}$ of $\pi_1(V',p)$ on $\{1,\ldots,M\}/\sim$ and $\{1,\ldots,M'\}/\sim'$.  But now we use the crucial fact (see Proposition \ref{tpg}) that $\pi_1(V',p)$ is topologically finitely generated, so that in order to specify an action such as $\rho_w$ on a finite set such as $\{1,\ldots,M\}/\sim$, it suffices to specify the action of a finite number of topological generators.  There are only finitely many such possibilities for such actions, so after partitioning $W$ further into finitely many subsets and passing to one of these subsets, we may assume that $\rho_w$ is in fact independent of $w$ for generic $w$ in these subsets; similarly we may assume that $\rho'_{w'}$ is independent of $w'$ for generic $w'$ in one of the finitely many subsets partitioning $W'$.  Now, the orbits in $\{1,\ldots,M\} \times \{1,\ldots,M'\}$ are completely independent of $w,w'$ for generic $w,w'$ in their respective subsets, giving Proposition \ref{Four-reg} and hence Lemma \ref{algreg-nonst}.

\section{Extension to higher dimensions}

We can iterate Lemma \ref{algreg-nonst} to obtain a higher dimensional version:

\begin{theorem}[Regularity lemma, higher dimensional version]\label{threg-higher}  Let $\F$ be a nonstandard finite field of characteristic zero, let $d,k \ge 1$ be a standard natural number, let $V_1,\ldots,V_d$ be definable sets over $\F$, and let $E_1,\ldots,E_k$ be definable subsets of $V_1 \times \ldots \times V_d$.  Then for each $1 \leq i \leq d$, one can partition $V_i$ into a finite number of definable sets $V_{i,1},\ldots,V_{i,a_i}$ for some standard natural number $a_i$, with the following property: for any natural numbers $j_1,\ldots,j_d$ with $1 \leq j_i \leq a_i$ for each $1 \leq i \leq d$ and $1 \leq l \leq k$, there exists a standard rational number $0 \leq \sigma_{l,j_1,\ldots,j_d} \leq 1$ with the property that
\begin{equation}\label{rag}
 |E_l \cap (A_1 \times \ldots \times A_d)| = \sigma_{l,j_1,\ldots,j_d} |A_1| \ldots |A_d| + O(  |\F|^{-1/4} |V_1| \ldots |V_d| ).
\end{equation}
for all nonstandard sets $A_1,\ldots,A_d$ with $A_i \subset V_{i,j_i}$ for all $1 \leq i \leq d$.

Furthermore, we may ensure that $|V_{i,j}| \gg |V_i|$ for all $1 \leq i \leq d$ and $1 \leq j \leq a_i$. 
\end{theorem}

\begin{remark} We stress that this lemma is \emph{not} analogous to the full ``hypergraph regularity lemma'' that generalises the Szemer\'edi regularity lemma \cite{rs}, \cite{rodl}, \cite{gowers-hyper}, \cite{tao-hyper}, but is instead more analogous to the earlier hypergraph regularity lemma of Chung \cite{chung-hyper} (see also \cite{frankl}); see the recent paper \cite{mubayi} for some discussion of the hierarchy of different hypergraph regularity lemmas.
\end{remark}

\begin{proof} We first remark that the final conclusion $|V_{i,j}| \gg |V_i|$ can be obtained ``for free'' as follows: by Lemma \ref{sz}, we have either $|V_{i,j}| \gg |V_i|$ or $|V_{i,j}| \ll |\F|^{-1} |V_i|$ for each $i,j$.  Any $V_{i,j}$ that obeys the latter bound instead of the former can be absorbed without difficulty into one of the sets $V_{i,j}$ in the partition that obeys the former bound (and, by the pigeonhole principle, at least one of the $V_{i,j}$ will obey that bound), without affecting the regularity property \eqref{rag}.  Thus, to prove Theorem \ref{threg-higher}, we may ignore the final requirement that $|V_{i,j}| \gg |V_i|$ for all $i,j$.

We now induct on $d$.  The case $d=1$ is trivial, so suppose $d \geq 2$, and the claim has already been proven for $d-1$.

Next, we observe that to prove the theorem for a given $d$, it suffices to do so when $k=1$, as the higher $k$ case follows by applying the theorem to each $E_l$ separately and then intersecting together all the definable subsets $V_{i,j}$ of $V_i$ produced by this theorem.  Thus we may assume $k=1$, and abbreviate $E_1$ as $E$.  Applying Lemma \ref{algreg-nonst} with $V$ and $W$ set equal to $V_1 \times \ldots \times V_{d-1}$ and $V_d$, we may partition
$$ V_1 \times \ldots \times V_{d-1} = E'_1 \cup \ldots \cup E'_{k'}$$
and
$$ V_d = V_{d,1} \cup \ldots \cup V_{d,a_d}$$
where $k', a_d$ are standard natural numbers and the $E'_j$, $V_{d,j}$ are definable sets with the property that for each $1 \leq j' \leq k'$ and $1 \leq j_d \leq a_d$ there exists a standard rational number $0 \leq \sigma_{j',j_d} \leq 1$ such that
\begin{equation}\label{aoid}
 |E \cap (A' \times A_d)| = \sigma_{j',j_d} |A'| |A_d| + O( |\F|^{-1/4} |V_1| \ldots |V_d| )
\end{equation}
for all nonstandard sets $A' \subset E'_{j'}$ and $A_d \subset V_{d,j}$.

Next, we apply the induction hypothesis to the $E'_1,\ldots,E'_{k'}$ to obtain partitions $V_i = V_{i,1} \cup \ldots \cup V_{i,a_i}$ for $i=1,\ldots,d-1$ into definable sets with the property that for any $j_1,\ldots,j_{d-1}$ with $1 \leq j_i \leq a_i$ for each $1 \leq i \leq d-1$, and each $1 \leq j' \leq k'$, there exists a standard rational $0 \leq \sigma_{j',j_1,\ldots,j_{d-1}} \leq 1$ such that
\begin{equation}\label{load}
 |E'_{j'} \cap (A_1 \times \ldots \times A_{d-1})| = \sigma_{j',j_1,\ldots,j_{d-1}} |A_1| \ldots |A_{d-1}| + O( |\F|^{-1/4} |V_1| \ldots |V_{d-1}| )
\end{equation}
whenever $A_1,\ldots,A_{d-1}$ are nonstandard sets with $A_i \subset V_{i,j_i}$ for all $1 \leq i \leq d-1$.

Now suppose that $1 \leq j_i \leq a_i$ for $i=1,\ldots,d$, and let $A_1,\ldots,A_d$ are nonstandard sets with $A_i \subset V_{i,j_i}$ for all $1 \leq i \leq d$.  We compute the quantity
$$ |E \cap (A_1 \times \ldots \times A_d)|.$$
Intersecting $A_1 \times \ldots \times A_{d-1}$ with each of the $E'_{j'}$ and using \eqref{aoid}, we can write this expression as
$$ \sum_{j'=1}^{k'} \sigma_{j',j_d} |E'_{j'} \cap (A_1 \times \ldots \times A_{d-1})| |A_d| + O( |\F|^{-1/4} |V_1| \ldots |V_d| ).$$
Applying \eqref{load}, we can simplify this to
$$ \sigma_{j_1,\ldots,j_d} |A_1| \ldots |A_d| + O( |\F|^{-1/4} |V_1| \ldots |V_d| )$$
where
$$ \sigma_{j_1,\ldots,j_d} := \sum_{j'=1}^{k'} \sigma_{j',j_d}  \sigma_{j',j_1,\ldots,j_{d-1}}$$
and the claim follows (note that $\sigma_{j_1,\ldots,j_d}$ can be adjusted if necessary to not exceed $1$, since $|E \cap (A_1 \times \ldots \times A_d)|$ is trivially bounded by $|A_1| \ldots |A_d|$).
\end{proof}

For computational purposes, it is convenient to rephrase \eqref{rag} as follows.

\begin{corollary}\label{long}  Let the notation, hypotheses, and conclusion be as in Theorem \ref{threg-higher}. Then for any nonstandard functions $f_i: V_i \to \ultra \C$ for $i=1,\ldots,d$ with $|f_i(x)| \ll 1$ for all $x \in V_i$, and any $1 \leq l \leq k$, the expression
$$
\sum_{(x_1,\ldots,x_d) \in V_1\times \ldots \times V_d} 1_{E_l}(x_1,\ldots,x_d) f_1(x_1) \ldots f_d(x_d)
$$
is equal to
$$ \sum_{j_1=1}^{a_1} \ldots \sum_{j_d=1}^{a_d} \sigma_{l,j_1,\ldots,j_d} \prod_{i=1}^d (\sum_{x_i \in V_{i,j_i}} f_i(x_i)) + 
O(  |\F|^{-1/4} |V_1| \ldots |V_d| ).$$
\end{corollary}

\begin{proof}  By decomposing each $f_i$, we may assume that $f_i$ is real, non-negative, and bounded by $1$, and supported on a single set $V_{i,j_i}$ for some $1 \leq j_i \leq a_i$.  Our task is then to show that
$$
\sum_{(x_1,\ldots,x_d) \in V_{1,j_1}\times \ldots \times V_{d,j_d}} 1_{E_l}(x_1,\ldots,x_d) f_1(x_1) \ldots f_d(x_d)
$$
is equal to
$$ \sigma_{l,j_1,\ldots,j_d} \prod_{i=1}^d (\sum_{x_i \in V_{i,j_i}} f_i(x_i)) + 
O(  |\F|^{-1/4} |V_1| \ldots |V_d| ).$$
By expressing each $f_i$ as a (nonstandard) integral $f_i = \int_0^1 1_{f_i \geq t}\ dt$ of (nonstandard) indicator functions, we may reduce to the case where each $f_i$ is an indicator function.  But the claim then follows from \eqref{rag}.
\end{proof}

\subsection{Expanding definable maps}

We now apply the above regularity lemma to establish the following dichotomy for definable maps.  Given two definable sets $V, W$ over a field $F$, call a function $f: V \to W$ \emph{definable} if its graph $\{ (v, f(v))\mid v \in V \}$ is a definable set.

\begin{theorem}[Expansion dichotomy]\label{expand-thm} Let $\F$ be a nonstandard finite field of characteristic zero, and let $V, W, U$ be geometrically irreducible quasiprojective varieties defined over $\overline{\F}$.  Let $P: V \times W \to U$ be a regular map defined over $\F$.  Then at least one of the following statements hold:
\begin{itemize}
\item[(i)] (Algebraic constraint)  The set
$$ \{ (P(v,w), P(v,w'), P(v',w), P(v',w'))\mid v,v' \in V; w,w' \in W \}$$
is not Zariski dense in $U^4$.
\item[(ii)]  (Moderate expansion) There exists a partition of $U(\F)$ into finitely many definable subsets $U(\F) = U_1 \cup \ldots \cup U_m$ with $|U_i| \gg |U(\F)|$ for all $1 \leq j \leq m$, with the property that for any nonstandard sets $A \subset V(\F)$, $B \subset W(\F)$, there exists $1 \leq j \leq m$ with
$$ |U_j \backslash P(A,B)| \ll |\F|^{-1/16} (|V(\F)|/|A|)^{1/2} (|W(\F)|/|B|)^{1/2} |U_j|.$$
In particular, we have the following moderate expansion property: if $|A||B| \ggg |\F|^{-1/8} |V(\F)| |W(\F)|$, then $|P(A,B)| \gg |U(\F)|$.
\end{itemize}
\end{theorem}

\begin{proof}
Assume that conclusion (i) of that theorem fails.  We consider the set
$$ \Sigma := \{ (v,v',w,w',P(v,w), P(v,w'), P(v',w), P(v',w'))\mid v,v' \in V; w,w' \in W \} \subset V^2 \times W^2 \times U^4.$$
This is a graph of a regular map from $V^2 \times W^2$ to $U^4$ and is thus (by Proposition \ref{proj}) an irreducible constructible set of dimension $2\dim(V)+2\dim(W)$.  By hypothesis, the projection of this set to $U^4$ is Zariski dense, and thus the projection map $\pi$ from $\Sigma$ to $U^4$ is dominant.  Thus, outside of a subvariety $\Lambda$ of $\overline{U}^4$ of dimension strictly less than $4\dim(V)$, the fibres of $\pi$ are $2\dim(V)+2\dim(W)-4\dim(U)$-dimensional.  Furthermore, $\pi^{-1}(\Lambda)$ has dimension strictly less than $2\dim(V)+2\dim(W)$.

By Lemma \ref{sz}, for any $F$-point $x \in U(\F)^4$ that does not lie in $\Lambda$, the $F$-points $\pi^{-1}(\{x\})(\F)$ of the fibre at $x$ have cardinality 
\begin{equation}\label{Sang}
|\pi^{-1}(\{x\})(\F)| = (c(x) + O(|\F|^{-1/2})) |\F|^{2\dim(V)+2\dim(W)-4\dim(U)},
\end{equation}
where $c(x)$ is the number of top-dimensional geometrically irreducible components of the fibre $\pi^{-1}(\{x\})$ which are defined over $\F$.  As this is finite for every $x$, we see from countable saturation (or from degree considerations) that $c(x)$ is uniformly bounded in $x$.  Also, from the definition of $c(x)$, it is clear that the level sets $E_{c_0} := \{ x \in U(\F)^4 \backslash \Lambda(\F)\mid c(x) = c_0 \}$ are definable sets for each standard natural number $c_0$ (and, by the preceding discussion, are empty for sufficiently large $c_0$).  Applying Theorem \ref{threg-higher} (and combining the four partitions of $U$ obtained by that theorem), we may thus find a partition $U(\F) = U_1 \cup \ldots \cup U_m$ into finitely many definable subsets with $|U_i| \gg |U(\F)|$ for all $1 \leq i \leq m$, such that for any $1 \leq j_1,j_2,j_3,j_4 \leq m$ and natural number $c_0$, there exists a standard rational number $0 \leq \sigma_{c_0,j_1,j_2,j_3,j_4} \leq 1$ such that
$$ |E_{c_0} \cap (A_1 \times A_2 \times A_3 \times A_4)| = \sigma_{c_0,j_1,j_2,j_3,j_4} |A_1| |A_2| |A_3| |A_4| + O( |\F|^{-1/4} |U(\F)|^4 )$$
for all nonstandard subsets $A_1,A_2,A_3,A_4$ of $U_{j_1}, U_{j_2}, U_{j_3}, U_{j_4}$ respectively.  From Corollary \ref{long} we see that
\begin{equation}\label{esum}
\sum_{(u_1,u_2,u_3,u_4) \in E_{c_0}} f_1(u_1) f_2(u_2) f_3(u_3) f_4(u_4) = 
\sigma_{c_0,j_1,j_2,j_3,j_4} \prod_{i=1}^4 (\sum_{u_i \in U_{j_i}} f_i(u_i)) + O( |\F|^{-1/4} |U(\F)|^4 )
\end{equation}
whenever, for each $i=1,2,3,4$, $f_i: U(\F) \to \ultra \R$ is a nonstandard function supported on $U_{j_i}$ bounded in magnitude by $1$.

Fix the partition $U(\F) = U_1 \cup \ldots \cup U_m$, let $1 \leq j \leq m$ be an index, and let $f: U(\F) \to \R^+$ be a nonstandard function bounded in magnitude by $1$ that is supported on $U_j$ and has mean zero.  We consider the quantity
\begin{equation}\label{able}
|\sum_{v \in A} \sum_{w \in B} f( P( v, w ) )|.
\end{equation}
By the Cauchy-Schwarz inequality, we may bound this expression by
$$ |A|^{1/2} (\sum_{v \in V(\F)} |\sum_{w \in B} f( P( v, w ) )|^2)^{1/2}$$
which we can rewrite as
$$ |A|^{1/2} |\sum_{w,w' \in B} \sum_{v \in V(\F)} f( P(v,w) ) f( P(v,w') )|^{1/2}.$$
By a second application of Cauchy-Schwarz, we can bound this expression by
$$ |A|^{1/2} |B|^{1/2} (\sum_{w,w' \in W(\F)} |\sum_{v \in V(\F)} f( P(v,w) )f( P(v,w') )|^2)^{1/4}$$
which we can rearrange as
$$|A|^{1/2} |B|^{1/2} |\sum_{(v,v',w,w') \in V(\F)^2 \times W(\F)^2} f( P(v,w) )f( P(v,w') ) f( P(v',w) ) f( P(v',w') )|^{1/4}$$
or equivalently
$$|A|^{1/2} |B|^{1/2} |\sum_{s \in \Sigma(\F)} f^{\otimes 4}( \pi( s ) )|^{1/4}$$
where the tensor power $f^{\otimes 4}: U(\F)^4 \to \R^+$ of $f$ is defined by the formula 
$$ f^{\otimes 4}( u_1,u_2,u_3,u_4) := f(u_1) f(u_2) f(u_3) f(u_4)$$
for $u_1,u_2,u_3,u_4 \in U(\F)$.  

Since $\pi^{-1}(\Lambda)$ has dimension strictly less than $2\dim(V)+2\dim(W)$, we see from Lemma \ref{sz} that at most $O( |\F|^{2\dim(V)+2\dim(W)-1} )$ of the points $s \in \Sigma(\F)$ lie in $\pi^{-1}(\Lambda)$.  Thus, by the boundedness of $f$, we may bound \eqref{able} by
$$|A|^{1/2} |B|^{1/2} |\sum_{s \in \Sigma(\F) \backslash \pi^{-1}(\Lambda)} f^{\otimes 4}( \pi( s ) ) + O(|\F|^{2\dim(V)+2\dim(W)-1}) |^{1/4}$$
which we can rewrite as
\begin{align*}
&|A|^{1/2} |B|^{1/2} |\sum_{(u_1,u_2,u_3,u_4) \in U(\F)^4 \backslash \Lambda} |\pi^{-1}(\{(u_1,u_2,u_3,u_4)\})(\F)| \times \\
&\quad \times f(u_1) f(u_2) f(u_3) f(u_4) + O(|\F|^{2\dim(V)+2\dim(W)-1})|^{1/4}.
\end{align*}
Applying \eqref{Sang} (and Lemma \ref{sz} to bound $|U(\F)|$), we can bound this by
\begin{align*}
&|A|^{1/2} |B|^{1/2} | |\F|^{2\dim(V)+2\dim(W)-4\dim(U)} \sum_{(u_1,u_2,u_3,u_4) \in U(\F)^4 \backslash \Lambda} \\
&\quad c(u_1,u_2,u_3,u_4) f(u_1) f(u_2) f(u_3) f(u_4) + O(|\F|^{2\dim(V)+2\dim(W)-1/2}) )|^{1/4}.
\end{align*}
We can rewrite this as
\begin{equation}\label{daisy}
\begin{split}
&|A|^{1/2} |B|^{1/2} | |\F|^{2\dim(V)+2\dim(W)-4\dim(U)} \sum_{c_0 \leq C_0} c_0 \sum_{(u_1,u_2,u_3,u_4) \in E_{c_0}} \\
&\quad f(u_1) f(u_2) f(u_3) f(u_4) + O(|\F|^{2\dim(V)+2\dim(W)-1/2}) )|^{1/4}.
\end{split}
\end{equation}
where $C_0$ is the largest value of $c_0$ for which $E_{c_0}$ is non-empty (as mentioned previously, $C_0$ is a standard natural number).  Applying \eqref{esum}, we can bound the above expression by
$$|A|^{1/2} |B|^{1/2} | O( |\F|^{2\dim(V)+2\dim(W)-1/4} ) |^{1/4},$$
and thus (by Lemma \ref{sz})
$$ 
|\sum_{v \in A} \sum_{w \in B} f( P( v, w ) )|
\ll |\F|^{-1/16} |A|^{1/2} |V(\F)|^{1/2} |B|^{1/2} |W(\F)|^{1/2} $$
whenever $f: U(\F) \to \ultra \R$ is a nonstandard function supported on $U_j$, bounded in magnitude by $1$ and of mean zero.

For each $u \in U(\F)$, define the multiplicity function
$$ \mu(u) := |\{ (v,w) \in A \times B\mid P(v,w) = u \}|,$$
then the above bound can be rewritten as
$$ |\sum_{u \in U_j} f(u) \mu(u)| \ll |\F|^{-1/16} |A|^{1/2} |V(\F)|^{1/2} |B|^{1/2} |W(\F)|^{1/2}$$
whenever $1 \leq j \leq a$ and $f: U(\F) \to \ultra \R$ is a nonstandard function supported on $U_j$, bounded in magnitude by $1$ and of mean zero.  In particular, one has
\begin{equation}\label{mujo}
 \sum_{u \in U_j} |\mu(u) - \mu_j| \ll |\F|^{-1/16} |A|^{1/2} |V(\F)|^{1/2} |B|^{1/2} |W(\F)|^{1/2}
\end{equation}
for all $1 \leq j \leq a$ and $\mu_j \in \ultra \R$ is the average value of $\mu(u)$ on $U_j$, as can be seen by taking $f$ to be the signum function of $\mu(u)-\mu_j$, normalised to have mean zero and bounded in magnitude by $1$.  

On the other hand, by double counting we have
$$ |A| |B| = \sum_{u \in U} \mu(u) = \sum_{j=1}^m |U_j| \mu_j.$$
By the pigeonhole principle, we can find $1 \leq j \leq m$ such that
\begin{equation}\label{muj}
 \mu_j \gg |U_j|^{-1} |A| |B|.
\end{equation}
From this and \eqref{mujo} we see that
\begin{equation}\label{loe}
 |\{ u \in U_j\mid \mu(u) = 0 \}| \ll |\F|^{-1/16} (|V(\F)|/|A|)^{1/2} (|W(\F)|/|B|)^{1/2} |U_j|.
 \end{equation}
Since $\{ u \in U_j\mid \mu(u) = 0 \} = U_j \backslash P(A,B)$, the claim follows. 
\end{proof}

\begin{remark}\label{three}  From \eqref{mujo} (and bounding $\mu_j$ crudely by $|A| |B|/ |U_j|$), we conclude the additional bound
$$ |\{ (a,b) \in A \times B\mid P(a,b) \in C \}| \ll \frac{|A| |B| |C|}{|U(\F)|} + |\F|^{-1/16} |A|^{1/2} |V(\F)|^{1/2} |B|^{1/2} |W(\F)|^{1/2}$$
whenever $A \subset V(\F), B \subset W(\F), C \subset U(\F)$ are nonstandard sets and conclusion (i) of Theorem \ref{expand-thm} fails.  A variant of the above argument gives the more general bound
$$ | (A \times B \times C) \cap S| \ll \frac{|A| |B| |C|}{|U(\F)|} + |\F|^{-1/16} |A|^{1/2} |V(\F)|^{1/2} |B|^{1/2} |W(\F)|^{1/2}$$
whenever $S$ is a subvariety of $V \times W \times U$ with the property that the fibres $\{ u \in U: (v,w,u) \in S \}$ are finite for all $v \in V, w \in W$, and such that the set
$$ \{ (u_1,u_2,u_3,u_4) \in U^4\mid \exists v,v' \in V; w,w' \in W: (u_1,v,w), (u_2,v,w'), (u_3,v',w), (u_4,v',w') \in S \}$$
is Zariski dense in $U^4$, by replacing \eqref{able} with the more general expression
$$ |\sum_{v \in A} \sum_{w \in B} \sum_{u \in U: (v,w,u) \in S} f( u )|.$$
We leave the details of the above generalisation to the interested reader.  Such bounds, in the context of subsets of $\C$ rather than of $\F$, were studied in \cite{elekes-szabo}.  It is likely that one could use the techniques in this paper to then establish analogues of the main results of \cite{elekes-szabo} in the context of large subsets of finite fields of large characteristic, but we will not pursue this issue here.
\end{remark}

We also have a variant that gives stronger expansion provided that one can rule out a second constraint:

\begin{theorem}[Second expansion dichotomy]\label{expand2-thm} Let $\F$ be a nonstandard finite field, let $V, W, U$ be geometrically irreducible quasiprojective varieties defined over $\F$.  Let $P: V \times W \to U$ be a regular map defined over $\F$.  Then at least one of the following statements hold:
\begin{itemize}
\item[(i)] (Algebraic constraint)  The set
$$ \{ (P(v,w), P(v,w'), P(v',w), P(v',w'))\mid v,v' \in V; w,w' \in W \}$$
is not Zariski dense in $U^4$.
\item[(ii)]  (Second algebraic constraint)  There exist geometrically irreducible quasiprojective varieties $V',W',U'$ defined over $\F$ with the same dimensions as $V,W,U$ respectively, and dominant regular maps $f: V' \to V$, $g: W' \to W$, $h: U' \to U$ defined over $\F$ such that the variety
$$  \{ (v',w',u') \in V' \times W' \times U'\mid P( f(v'), g(w') ) = h(u') \}$$
is not irreducible.
\item[(iii)]  (Strong expansion) For any non-empty nonstandard sets $A \subset V(\F)$, $B \subset W(\F)$, one has the strong expansion property
$$ |U(\F) \backslash P(A,B)| \ll |\F|^{-1/16} (|V(\F)|/|A|)^{1/2} (|W(\F)|/|B|)^{1/2} |U|.$$
\end{itemize}
\end{theorem}

\begin{proof} Now we use a variant of the previous argument.  We repeat all the construction and argument in the proof of Theorem \ref{expand-thm}.  If we have \eqref{muj} for all $1 \leq j \leq m$, then we have conclusion (iii) by summing \eqref{loe} in $j$.  Thus we may assume that
$$
 \mu_j = o(|U_j|^{-1} |A| |B| )
$$
for some $j$.  Thus, for this $j$, we have
$$ \sum_{v \in A} \sum_{w \in B} 1_{U_j}(P(v,w)) = o(|A| |B|).$$
Write $E := \{ (v,w) \in V(\F) \times W(\F)\mid P(v,w) \in U_j \}$, then $E$ is a definable subset of $V(\F) \times W(\F)$ and
\begin{equation}\label{loo}
 |E \cap (A \times B)| = o( |A| |B| ).
 \end{equation}
By Lemma \ref{algreg-nonst}, we can partition $V(\F)$ into a finite number of definable sets $V_1,\ldots,V_a$ and $W(\F)$ into a finite number of definable sets $W_1,\ldots,W_b$, such that for any $1 \leq i \leq a$ and $1 \leq i' \leq b$, there exists a standard rational number $0 \leq \sigma_{ii'} \leq 1$ with the property that
\begin{equation}\label{lan}
 |E \cap (A' \times B')| = \sigma_{ii'} |A'| |B'| + O( |\F|^{-1/4} |V(\F)| |W(\F)| ).
\end{equation}
for all nonstandard sets $A' \subset V_i$ and $B' \subset W_{i'}$; by concatenating the $V_i, W_{i'}$ if necessary we may assume that $|V_i| \gg |V(\F)|$ and $|W_j| \gg |W(\F)|$.  

Specialising to $A' := A \cap V_i$ and $B' := B \cap W_{i'}$ and using \eqref{loo} we see that
$$ \sigma_{ii'} |A'| |B'| \ll |\F|^{-1/4} |V(\F)| |W(\F)| + o( |A| |B| ).
$$

By the pigeonhole principle we can find $i,i'$ such that
$$ |A \cap V_i| \gg |A|$$
and
$$ |B \cap W_{i'}| \gg |B|,$$
and for this choice of $i,i'$ we thus have
$$ \sigma_{ii'} \ll |\F|^{-1/4} (|V(\F)|/|A|) (|W(\F)|/|B|) + o(1).$$
We may assume that $(|V(\F)|/|A|) (|W(\F)|/|B|) \ll |\F|^{1/8}$, otherwise conclusion (iii) is vacuously true; and so $\sigma_{ii'} \ll o(1)$.  Since $\sigma_{ii'}$ is standard rational, we conclude $\sigma_{ii'} = 0$.  Going back to \eqref{lan} and now setting $A' := V_i$ and $B' := W_{i'}$, we conclude that
$$ |E \cap (V_i \times W_{i'})| \ll |\F|^{-1/4} |V(\F)| |W(\F)|,$$
or in other words that
$$ \sum_{v \in V} \sum_{w \in W} 1_{V_i}(v) 1_{W_{i'}}(w) 1_{U_j}(P(v,w)) \ll |\F|^{-1/4} |V(\F)| |W(\F)|.$$

Using Theorem \ref{aqe}, we can write $V_i = f( V'(\F) )$ for some variety $V'$ defined over $\F$ with the same dimension as $V$, and some dominant map $f$ from $V'$ to $V$, also defined over $\F$;.  Similarly $W_{i'} = g(W'(\F))$ and $U_j = h(U'(\F))$; thus
$$ \sum_{v \in V} \sum_{w \in W} 1_{f(V'(\F))}(v) 1_{g(W'(\F))}(w) 1_{h(U'(\F)}(P(v,w)) \ll |\F|^{-1/4} |V(\F)| |W(\F)|.$$
From Proposition \ref{proj}, we can find an open dense subvariety $V''$ of $V'$ such that $f$ is quasi-finite on $V''$, and thus (by countable saturation, or degree considerations) one has a uniform bound $|f^{-1}(\{v\})| \leq C$ for all $v \in f(V'')$.  By intersecting $V''$ with its Galois conjugates, we may assume that $V''$ is defined over $\F$.  Similarly we can find open dense subvarieties $W'', U''$ of $W', U'$ respectively defined over $\F'$ such that $g, h$ are quasi-finite on $W'', U''$ respectively.  Then
$$ \sum_{v \in V} \sum_{w \in W} 1_{f(V''(\F))}(v) 1_{g(W''(\F))}(w) 1_{h(U''(\F)}(P(v,w)) \ll |\F|^{-1/4} |V(\F)| |W(\F)|$$
and hence
$$ |\{ (v', w', u' ) \in V''(\F) \times W''(\F) \times U''(\F) \mid h(u') = P(f(v'), g(w')) \}| \ll |\F|^{-1/4} |V(\F)| |W(\F)|.$$
In particular, the variety
$$ \{ (v', w', u') \in V'' \times W'' \times U'' \mid h(u') = P(f(v'),g(w')) \}$$
has $o( |V(\F)| |W(\F)| ) = o( |\F|^{\dim(V)+\dim(W)})$ $\F$-points.  On the other hand, this variety is quasi-finite over $\{ (v',w',P(f(v'),g(w')))\mid v \in V'', w \in W'' \}$ and thus has dimension $\dim(V)+\dim(W)$.  The only way this is consistent with Lemma \ref{sz} is if the variety is not irreducible, and the claim follows.
\end{proof}

\section{Solving the algebraic constraint}\label{algo}

In this section we solve the algebraic constraint that emerges in Theorem \ref{expand-thm}, in the case when $V=W=k$ is just an affine line and $P$ is polynomial.  More precisely, we show

\begin{theorem}\label{const-0}  Let $k$ be an algebraically closed field of characteristic zero.  Let $P: k^2 \to k$ be a polynomial with the property that the set
\begin{equation}\label{lka}
 \{ (P(a,c), P(a,d), P(b,c), P(b,d)) \mid  a,b,c,d \in k  \}
\end{equation}
is not Zariski dense in $k^4$.  Then one of the following statements hold:
\begin{itemize}
\item[(i)] (Additive structure)  There exist polynomials $Q, F, G: k \to k$ such that $P(x,y) = Q(F(x)+G(y))$ for all $x,y \in k$.
\item[(ii)] (Multiplicative structure) There exist polynomials $Q, F, G: k \to k$ such that $P(x,y) = Q(F(x) G(y))$ for all $x,y \in k$.
\end{itemize}
\end{theorem}

This is not quite what we need for Theorem \ref{expand-thm-modas-nonst} because the polynomials obtained here are defined over an algebraically closed field $k$, rather than over the nonstandard finite field $\F$; we address this issue at the end of this section.

The strategy of proof of Theorem \ref{const-0} will be to use the Lefschetz principle to reduce to the complex case $k=\C$, and then use complex analytic methods to analyse certain Riemann surfaces associated with $P$.

We begin with the reduction to the complex case $k=\C$.  This will be a standard ``Lefschetz principle'' argument.  We first observe that to prove Theorem \ref{const-0}, it suffices to verify the case when $k$ has finite transcendence degree over the rationals $\Q$, as can be seen by passing to the field of definition of $P$.  In particular, we may identify $k$ with a subfield of $\C$.  If the set \eqref{lka} is not Zariski dense in $k^4$, then it is contained in a proper subvariety of $k^4$, and hence
$$
 \{ (P(a,c), P(a,d), P(b,c), P(b,d)) \mid  a,b,c,d \in \C \}
$$
is contained in the complexification of that variety and is thus also not Zariski dense.  Applying Theorem \ref{const-0} with $k=\C$, we conclude that we may find polynomials $Q_0,F_0,G_0: \C \to \C$ such that $P(x,y) = Q_0(F_0(x)+G_0(y))$ or $P(x,y) = Q_0(F_0(x)G_0(y))$ for all $x,y \in \C$.  For sake of discussion let us work with the additive case $P(x,y) = Q_0(F_0(x)+G_0(y))$, as the multiplicative case is analogous.  We are not quite done yet, because $Q_0,F_0,G_0$ have coefficients in $\C$ rather than in $k$.  We claim however that we may find polynomials $Q,F,G: k \to k$ of the same degree as $Q_0,F_0,G_0$ respectively and with coefficients in $k$ such that $P(x,y) = Q(F(x)+G(y))$ for all $x,y \in k$.  Indeed, this can be viewed as an algebraic constraint satisfaction problem in the coefficients of $Q,F,G$, where the constraints are defined over $k$.  If this problem has no solutions over the algebraically closed field $k$, then by Hilbert's nullstellensatz we see that the problem has no solution over any extension of $k$ either, and in particular has no solutions in $\C$, contradicting the existence of the factorisation $P(x,y)=Q_0(F_0(x)+G_0(y))$.  This concludes the reduction to the $k=\C$ case.

Henceforth $k=\C$.  Let $P: \C^2 \to \C$ be a polynomial such that
$$
 \{ (P(a,c), P(a,d), P(b,c), P(b,d)) \mid  a,b,c,d \in \C \}
$$
is not Zariski dense, i.e. it is contained in a proper subvariety of $\C^4$.  We conclude that for generic $(a,b,c,d) \in \C^4$, the derivative of the map
$$ (a,b,c,d) \mapsto (P(a,c), P(a,d), P(b,c), P(b,d))$$
from $\C^4$ to $\C^4$ is singular.  As this is a closed condition, this in fact holds for \emph{all} $(a,b,c,d) \in \C^4$.  Taking determinants, we conclude the constraint
$$
\begin{pmatrix}
P_1(a,c) & 0 & P_2(a,c) & 0 \\
P_1(a,d) & 0 & 0 & P_2(a,d) \\
0 & P_1(b,c) & P_2(b,c) & 0 \\
0 & P_1(b,d) & 0 & P_2(b,d)
\end{pmatrix} = 0$$
for all $a,b,c,d \in \C$, where $P_1, P_2$ are the partial derivatives of $P$ with respect to the first and second variable respectively.  Expanding this out, we conclude that
$$ P_1(a,c) P_2(a,d) P_2(b,c) P_1(b,d) = P_2(a,c) P_1(a,d) P_1(b,c) P_2(b,d)$$
for all $a,b,c,d \in \C$.

If one of $P_1$ or $P_2$ is identically zero, then $P(x,y)$ is a function of just one of the two variables $x,y$, and one trivially has both additive and multiplicative structure, so we henceforth assume that $P_1,P_2$ are not identically zero.  We can then rearrange the above identity as
$$ \frac{P_1}{P_2}(a,c) \frac{P_1}{P_2}(b,d) = \frac{P_1}{P_2}(a,d) \frac{P_1}{P_2}(b,c)$$
for generic $a,b,c,d \in \C$.  Fixing a generic choice of $b,d$, we conclude in particular (after relabeling $c$ as $b$) that
\begin{equation}\label{pilo}
 \frac{P_1}{P_2}(a,b) = \frac{f(a)}{g(b)}
\end{equation}
for generic $a,b \in \C$ and some rational functions $f, g: \C \cup \{\infty\} \to \C \cup \{\infty\}$, not identically zero or identically infinite.

To motivate the argument that follows, suppose that we could make a change of variables $a \mapsto z$, $b \mapsto w$ such that
$$ dz = f(a) da; \quad dw = g(b) db,$$
thus $z$ is a primitive of $f$ applied to $a$, and $w$ is a primitive of $g$ applied to $b$.  Then, by the chain rule, the constraint \eqref{pilo} simplifies to
$$ \frac{\partial P}{\partial z} - \frac{\partial P}{\partial w} = 0$$
and so $P$ must be a function of $z+w$.  This already resembles the conclusion of Theorem \ref{const-0} quite closely (particularly if one formally rewrites the multiplicative conclusion $P(x,y) = Q(F(x) G(y))$ as $P(x,y) = Q \circ \exp( \log \circ F(x) + \log \circ G(y) )$).  

Of course, the rational functions $f, g$ may contain simple poles, which by the residue theorem implies that they do not have single-valued primitives taking values in $\C$.  However, through monodromy one can still define a primitive taking values in some Riemann surface covering the (punctured) complex plane.  So, it is natural to try to execute the above strategy in the framework of Riemann surfaces rather than on the complex plane.  This is essentially what we will do next, except that we will not explicitly use the abstract language of Riemann surfaces and instead work with the more concrete machinery of paths in the complex plane, in order to easily take advantage of the additive structure of $\C$.

We turn to the details.  By collecting the residues of the poles of $f$, we may write
$$ f(a) = \sum_{j=1}^m \frac{\alpha_j}{a-a_j} + \tilde f(a)$$
for some finite number of distinct simple poles $a_1,\ldots,a_m \in \C$, some non-zero complex residues $\alpha_1,\ldots,\alpha_m \in \C$, and a rational function $\tilde f(a)$, such that all poles of $\tilde f$ have zero residue.  In particular, $\tilde f$ has a primitive $F$, which is a rational function, so that
\begin{equation}\label{faa}
 f(a) = \sum_{j=1}^m \frac{\alpha_j}{a-a_j} + F'(a)
\end{equation}
for all but finitely many $a \in \C$.  By translation, we may assume without loss of generality that $a_1,\ldots,a_m \neq 0$, and that $F(0)=0$.  

Now let $\gamma: [0,1] \to \C \backslash \{a_1,\ldots,a_m\}$ be a smooth curve avoiding the poles $a_1,\ldots,a_m$ of $f$ that starts at $\gamma(0)=0$.  From the fundamental theorem of calculus, we can evaluate the contour integral $\int_\gamma f = \int_\gamma f(a)\ da$ as
$$ \int_\gamma f = \sum_{j=1}^m \alpha_j \operatorname{Log} \frac{\gamma(1)-a_j}{a_j} + F(\gamma(1))$$
where (by abuse of notation) $\operatorname{Log} \frac{\gamma(1)-a_j}{a_j}$ is one of the logarithms $\log \frac{\gamma(1)-a_j}{a_j}$ of $\frac{\gamma(1)-a_j}{a_j}$, with the exact branch of logarithm used depending on the homotopy class of $\gamma$ in $\C \backslash \{a_1,\ldots,a_m\}$.  (Here we adopt the convention that $\int_\gamma f = \infty$ if $\gamma$ terminates at a (zero-residue) pole of $f$, and evaluate contour integrals that pass through (zero-residue) poles of $f$ by arbitrarily perturbing the contour around such poles.) In particular, we see that
$$ \int_\gamma f \in c_{\gamma(1)} + \Gamma_1$$
where $\Gamma_1 \leq \C$ is the finitely generated subgroup of $\C$ generated by $2\pi i \alpha_1,\ldots,2\pi i \alpha_m$, and for any $a \in \C \backslash \{a_1,\ldots,a_m\}$, $c_a+\Gamma_1$ denotes the coset
\begin{equation}\label{lor}
 c_a + \Gamma_1 := \sum_{j=1}^m \alpha_j \log \frac{a-a_j}{a_j} + F(a);
\end{equation}
thus $c_a$ is only defined up to an additive error in $\Gamma_1$.  Conversely, for any given choice of endpoint $a \in \C \backslash \{a_1,\ldots,a_m\}$, and any element of the coset $z = c_a + \Gamma_1$, one can find a smooth curve $\gamma:[0,1] \to \C \backslash \{a_1,\ldots,a_m\}$ from $0$ to $a$ with $\int_\gamma f = z$.

We need the following technical lemma, reminiscent of the Picard theorem on the possible values of entire functions.

\begin{lemma}[Almost surjectivity]\label{ass}  For all complex numbers $z$ outside of at most one coset of $\Gamma_1$, there exists at least one smooth curve $\gamma: [0,1] \to \C \backslash \{a_1,\ldots,a_m\}$ starting at $0$ with $\int_\gamma f = z$.  If $\Gamma_1$ is not trivial and is not a rank one lattice $\Gamma_1 = 2\pi i \alpha \Z$, then the caveat ``outside of at most one coset of $\Gamma_1$'' in the previous claim may be deleted.
\end{lemma}

To see why it is necessary sometimes to exclude a coset of $\Gamma$, consider the case
$$ f(a) := \frac{1}{a-1} - \frac{1}{a-2}$$
so that $\Gamma$ is the rank one lattice $\Gamma = 2\pi i \Z$ and
$$ \int_\gamma f = \operatorname{Log}(1-\gamma(1)) - \operatorname{Log}(1-\frac{\gamma(1)}{2}).$$
This quantity can attain all complex values except for those on the coset $\log 2 + \Gamma$ (this would require $\gamma$ to terminate at infinity, which is not possible since $\gamma$ has to stay inside $\C \backslash \{a_1,\ldots,a_m\}$).  Another example is when
$$ f(a) = \frac{1}{(a-1)^2}$$
so that $\Gamma = \{0\}$ is trivial and
$$ \int_\gamma f = \frac{\gamma(1)}{1-\gamma(1)}.$$
In this case, the coset $1 + \Gamma = \{1\}$ cannot be attained.

\begin{proof}  We first dispose of an easy case when $\Gamma_1$ is trivial (i.e. $f$ has no residues at any pole).  Then $\int_\gamma f = F(\gamma(1))$; as $f$ is not identically zero, the rational function $F$ is not constant, and the claim follows from the fundamental theorem of algebra.  (Note that the rational function $F-z_0$ has a constant numerator for at most one complex number $z_0$.)  Thus we may assume henceforth that $\Gamma_1$ is non-trivial.

Let $\Omega$ denote the set of all $z$ such that $z=\int_\gamma f$ for at least one $\gamma$ as in the lemma.  Clearly $\Omega$ contains $\{0\}$; by the preceding discussion, $\Omega$ is also the union of cosets of $\Gamma_1$.  Our objective is to show that $\Omega$ is either equal to $\C$, or $\C$ with one coset of $\Gamma_1$ deleted, with the latter case only permitted when $\Gamma_1$ is a rank one lattice.

If the rational function $F$ is not identically zero, then it is non-constant (since $F(0) = 0$), and thus has a pole somewhere in $\C \cup \{\infty\}$.  In any neighbourhood of this pole, the quantity \eqref{lor} attains all sufficiently large finite values in $\C$, thanks to Rouche's theorem, thus $\Omega$ contains the exterior of a sufficiently large disk.  Since $\Gamma_1$ is a non-trivial (hence unbounded) subgroup of $\C$, and $\Omega$ is invariant with respect to translations by $\Gamma_1$, we conclude in this case that $\Omega=\C$ as required.  Thus we may assume without loss of generality that $F=0$.

Next, suppose that the residues $\alpha_1,\ldots,\alpha_m$ do not all lie on a single line $\R\alpha$. We work in the neighbourhood of a single pole $a_1$.  From \eqref{lor} and Rouche's theorem, we see that $\Omega$ contains the half-plane $\{ \alpha_1 z \mid \Re(z) < -C \}$ for some sufficiently large $C$.  On the other hand, as the $\alpha_1,\ldots,\alpha_m$ do not all lie in a single line $\R\alpha$, the group $\Gamma_1$ is not contained in $2\pi i \alpha_1 \R$, and so the orthogonal projection onto $\alpha_1 \R$ is unbounded.  From this and the $\Gamma_1$-invariance of $\Omega$ we conclude that $\Omega=\C$ as desired.

The only remaining case is when all the $\alpha_1,\ldots,\alpha_m$ lie on a single line; by rotation, we may assume that the $\alpha_1,\ldots,\alpha_m$ are all real.   Let us first assume that the $\alpha_1,\ldots,\alpha_m$ are not all commensurable (i.e. rational multiples of each other), so that $\Gamma_1$ is a dense subset of $2\pi i\R$.   As non-constant analytic functions are open maps, and $f$ locally has an analytic primitive, we see that $\Omega$ is open; as $\C$ is connected, and $\Omega$ clearly is non-empty, it thus suffices to show that $\Omega$ is closed.  Accordingly, let $z_n$ be a sequence in $\Omega$ converging to a finite limit $z \in \C$; our task is to show that $z$ also lies in $\Omega$.

By definition, $z_n = \int_{\gamma_n} f$ for some sequence of smooth curves $\gamma_n:[0,1] \to \C \backslash \{a_1,\ldots,a_n\}$, all starting at $0$ but possibly having distinct endpoints $\gamma_n(1)$.  By the Bolzano-Weierstrass theorem, we may pass to a subsequence and assume that $a'_n := \gamma_n(1)$ converges to some limit $a' \in \C \cup \{\infty\}$.

If $a' = a_i$ for some pole $a_i$, then from \eqref{lor} and the hypothesis that the $\alpha_i$ are all real, we see that $\Re(z_n) \to -\sgn(\alpha_i) \infty$ as $n \to \infty$, which contradicts the convergence of $z_n$ to a finite limit.  If $a' \in \C\backslash \{a_1,\ldots,a_n\}$, then by choosing suitable branches of the logarithm, one can make the map $a \mapsto c_{a}$ analytic and non-constant (hence continuous and open) in a neighbourhood of $a'$; as $z_n \in c_{a'_n}+\Gamma_1$ for all $n$, we conclude on taking limits that $z \in c_{a'} + \overline{\Gamma}$; since $a \mapsto c_{a}$ is open, this implies that $z \in \Omega$ as required.

Next, we consider the case when $a'=\infty$, thus $\gamma_n(1) \to \infty$.  If $\alpha_1+\ldots+\alpha_m$ is non-zero, then from \eqref{lor} we see that $\Re(z_n) \to \sgn(\alpha_1+\ldots+\alpha_m) \infty$, contradicting the convergence of $z_n$ to a finite limit.  Thus $\alpha_1+\ldots+\alpha_m=0$, and so at least two of the $\alpha_i$ have opposing signs.  Without loss of generality we may assume that $\alpha_1 > 0 > \alpha_2$.  From \eqref{lor}, the real part of $\int_\gamma f$ depends only on the endpoint $\gamma(1)$, and goes to $-\infty$ as $\gamma(1)$ approaches $a_1$ and $+\infty$ as $\gamma(1)$ approaches $a_2$.  By the intermediate value theorem, we may thus find a smooth curve $\gamma$ starting at $0$ such that $\Re \int_\gamma f = \Re z$; as $\Gamma$ is a dense subgroup of $2\pi i \R$, we conclude that $\Omega$ contains a dense subset of the line $z + 2\pi i \R$.  As $\Omega$ is open and $\Gamma_1$-invariant, we conclude that $z \in \Omega$ as desired.  This concludes the proof of $\Omega=\C$ in the case that $\Gamma$ is a dense subgroup of $2\pi i \R$.

Finally, we have to consider the case when $\Gamma_1$ is a rank one lattice in $2\pi i \R$.  If $\alpha_1+\ldots+\alpha_m$ is non-zero, then we may repeat the previous arguments to show that $\Omega=\C$, so we may assume that $\alpha_1+\ldots+\alpha_m = 0$.  In this case, we see from \eqref{lor} that if $\gamma_n: [0,1] \to \C \backslash \{a_1,\ldots,a_m\}$ is a family of paths from the origin with $a'_n := \gamma_n(1)$ going to infinity, and $z_n := \int_{\gamma_n} f$, then
$$ \operatorname{dist}( z_n, c_\infty + \Gamma_1 ) \to 0$$
as $n \to \infty$, where $c_\infty + \Gamma_1$ is the coset
$$ c_\infty + \Gamma_1 := - \sum_{j=1}^m \alpha_j \operatorname{log} a_j + \Gamma_1.$$
From this and the previous arguments, we see that any limit point $z$ of $\Omega$ will also lie in $\Omega$ unless $z$ lies in $c_\infty + \Gamma_1$.  Thus the set $\Omega \backslash (c_\infty+\Gamma_1)$ is open, closed, and non-empty in $\C \backslash (c_\infty+\Gamma_1)$; as the latter set is connected, we conclude that $\Omega$ contains $\C \backslash (c_\infty+\Gamma_1)$, as required.
\end{proof}

All the above analysis involving the poles and residues of $f$ may similarly be applied to $g$.  More specifically, we may write
\begin{equation}\label{glib}
 g(b) = \sum_{k=1}^n \frac{\beta_k}{b-b_j} + G'(b)
\end{equation}
for all but finitely many $b \in \C$, where $b_1,\ldots,b_n \in \C$ are distinct and non-zero, $\beta_1,\ldots,\beta_k$ are non-zero, and $G$ is a rational function with $G(0)=0$.  We let $\Gamma_2$ be the subgroup of $\C$ generated by $2\pi i \beta_1,\ldots,2\pi i \beta_k$, and have the analogue of Lemma \ref{ass}:

\begin{lemma}[Almost surjectivity]\label{ass-2}  For all complex numbers $w$ outside of at most one coset of $\Gamma_2$, there exists at least one smooth curve $\gamma: [0,1] \to \C \backslash \{b_1,\ldots,b_n\}$ starting at $0$ with $\int_\gamma g = w$.  If $\Gamma_2$ is not trivial and is not a rank one lattice $\Gamma_2 = 2\pi i \alpha \Z$, then the caveat ``outside of at most one coset of $\Gamma_2$'' in the previous claim may be deleted.
\end{lemma}

Now, we are able to solve the differential equation \eqref{pilo}, obtaining a rigorous version of the heuristic that $P$ is ``a function of $z+w$'' in some sense:

\begin{proposition}[Additive structure in $z,w$]\label{Add}  There exists an entire function $H: \C \to \C$ with the property that
\begin{equation}\label{pae}
 P( \gamma_1(1), \gamma_2(1) ) = H( \int_{\gamma_1} f + \int_{\gamma_2} g )
\end{equation}
for all smooth curves $\gamma_1: [0,1] \to \C \backslash \{a_1,\ldots,a_m\}$ and $\gamma_2: [0,1] \to \C \backslash \{b_1,\ldots,b_n\}$ with $\gamma_1(0)=\gamma_2(0)=0$ and $\int_{\gamma_1} f, \int_{\gamma_2} g \neq \infty$.
\end{proposition}

\begin{proof}  We begin by establishing a preliminary property of $P$ implied by \eqref{pae}, namely that
\begin{equation}\label{pww}
 P( \gamma_1(1), b ) = P( \tilde \gamma_1(1), b )
\end{equation}
whenever $b \in \C$ and $\gamma_1, \tilde \gamma_1: [0,1] \to \C \backslash \{a_1,\ldots,a_m\}$ are smooth curves with $\gamma_1(0)=\tilde \gamma_1(0)$ and $\int_{\gamma_1} f = \int_{\tilde \gamma_1} f \neq \infty$.

Fix $\gamma_1, \tilde \gamma_1$ as above, and let $w_0$ be a complex number to be chosen later.  By perturbing $\gamma_1, \tilde \gamma_1$ without moving the endpoints, we may assume that $\gamma_1, \tilde \gamma_1$ avoid all the poles of $f$.  By unique continuation, it suffices to establish \eqref{pww} for all $b$ in some non-empty open subset of $\C$.

By the connectedness of $\C \backslash \{a_1,\ldots,a_n\}$, even after deleting the remaining poles of $f$, we can find a smooth homotopy $\gamma^*_1: [0,1] \times [0,1] \to \C \backslash \{a_1,\ldots,a_n\}$ with $\gamma^*_1(t,0) = \gamma_1(t)$ and $\gamma^*_1(t,1) = \tilde \gamma_1(t)$ for all $t \in [0,1]$, and $\gamma^*_1(0,s) = 0$ for all $s \in [0,1]$, and such that $\gamma^*_1$ avoids all the poles of $f$.

The quantity $\int_{\gamma^*_1(\cdot,s)} f$ will, in general, not be constant in $s$; however, it varies smoothly in $s$, and thus lies in a ball $B(0,R)$.  From Cauchy's theorem, we see that
$$ \int_{\gamma^*_1(\cdot,s+ds)} f = \int_{\gamma^*_1(\cdot,s)} f + \int_{\gamma^*(1,[s,s+ds])} f$$
for any $0 \leq s < s+ds \leq 1$, and thus
\begin{equation}\label{dds}
 \frac{d}{ds} \int_{\gamma^*_1(\cdot,s)} f = f( \gamma^*(1,s) ) \frac{d}{ds} \gamma^*(1,s)
\end{equation}
for all $s\in [0,1]$.

Suppose that we can find a smooth function $\tilde b: [0,1] \to \C \backslash \{b_1,\ldots,b_n\}$ such that $\tilde b(0)=\tilde b(1)=b$, $\tilde b$ avoids all the poles of $g$, and
\begin{equation}\label{geo}
 g(\tilde b(s)) \frac{d}{ds} \tilde b(s) = - f( \gamma^*(1,s) ) \frac{d}{ds} \gamma^*(1,s).
\end{equation}
From \eqref{pilo} and the chain rule, we then have
$$ \frac{d}{ds} P( \gamma^*(1,s), b(s) ) = 0$$
for all $s\in [0,1]$, and thus by the fundamental theorem of calculus
$$P( \gamma_1(1), b ) = P( \gamma^*(1,0), b(0) ) = P( \gamma^*(1,1), b(1) ) = P( \tilde \gamma(1), b )$$
which is \eqref{pww}.  Thus, it will suffice to construct such a smooth function $\tilde b$ for all $b$ in a non-empty open subset of $\C$.

Suppose that we can find a ball $B(w_0,2R+1)$ and a complex analytic local diffeomorphism $\Phi: B(w_0,2R+1) \to \C$ which avoids all the poles and zeroes of $g$ in its image, and is an inverse primitive of $g$ in the sense that
\begin{equation}\label{ddt}
 \Phi'(w) = g( \Phi(w) )^{-1}
\end{equation}
for all $w \in B(w_0,2R+1)$.  Then for any $b$ in the open set $\Phi(B(w_0,1))$, so that $b = \Phi(w)$ for some $w \in B(w_0,1)$, we can construct the desired function $\tilde b$ by the formula
$$ \tilde b(s) := \Phi( w + \int_{\gamma_1} f - \int_{\gamma^*_1(\cdot,s)} f ).$$
Indeed, from the chain rule and \eqref{dds}, \eqref{ddt} we obtain \eqref{geo}, and the condition $\tilde b(0)=\tilde b(1)=b$ follows from the hypothesis $\int_{\gamma_1} f = \int_{\tilde \gamma_1} f$.  Thus, to finish the proof of \eqref{pww}, it suffices to obtain a local diffeomorphism $\Phi: B(w_0,2R+1) \to \C$ with the stated properties.

Suppose that the function $G$ is not identically zero, and thus has a pole at some point $b_* \in \C \cup \{\infty\}$.  If $b_*$ is distinct from $b_1,\ldots,b_m$.  By \eqref{glib}, $g$ then has a meromorphic primitive $\tilde G$ in a neighbourhood of $b_*$ which has a pole at $b_*$ but is otherwise holomorphic, and which will be a local diffeomorphism if the neighbourhood is small enough.  In particular, by Rouche's theorem, we see that if $w_0$ is a complex number of sufficiently large magnitude, $\tilde G$ is a diffeomorphism between some open subset of this neighbourhood and $B(w_0,2R+1)$.   Taking $\Phi$ to be the inverse of $\tilde G$, we obtain the claim.

If the pole $b_*$ coincides with one of the $b_i$, then we do not have a meromorphic primitive in a neighbourhood of $b_*$ any more, but we still have a primitive on sufficiently small semi-neighbourhood such as $\{ b \in\C\mid |b-b_*| < \eps; \Re(b-b_*) > 0 \}$ which is still a local diffeomorphism.  By Rouche's theorem (the point being that the pole of $G$ dominates the logarithmic factors), one can still find a diffeomorphism from some subset of this neighbourhood and $B(w_0,2R+1)$ for a suitably chosen $w_0$, and we can argue as before.

Now suppose that $G$ is identically zero.  As $g$ is not identically zero, we see from \eqref{glib} that $n$ is non-zero, thus $g$ has a simple pole at $b_1$.  From \eqref{glib}, we see that $g$ \emph{formally} has a primitive in a neighbourhood of $b_1$ that is equal to the sum of $\beta_1 \log(b-b_1)$ plus a holomorphic function.  This is however not rigorous because $\log$ is multivalued.  To get around this difficulty, we cover a small punctured neighbourhood $B(b_1,\eps) \backslash \{b_1\}$ of $b_1$ by the half-space $\{ z\mid \Re(z) < \log \eps \}$ via the translated exponential map $z \mapsto b_1 + \exp(z)$.  The differential $g(b) db$ on $B(b_1,\eps) \backslash \{b_1\}$ then pulls back to the differential
$$ g( b_1 + \exp(z) ) \exp(z) dz$$
on the half-space $\{ z\mid \Re(z) < \log \eps \}$.  For $\eps$ small enough, this differential has a primitive $\tilde G$, which is the sum of $\beta_1 z$ plus a holomorphic function of $b_1+\exp(z)$.  In particular, if $w_0$ is a complex number with $\beta_1^{-1} \Re(w_0)$ sufficiently large and negative, we see from Rouche's theorem that $\tilde G$ is a diffeomorphism between some open subset of this neighbourhood and $B(w_0,2R+1)$.  Taking $\Phi$ to be the inverse of $\tilde G$, composed with the translated exponential map $z \mapsto b_1 + \exp(z)$ (i.e. $\Phi(w) := b_1+ \exp( \tilde G^{-1}(w) )$), we obtain the claim.  This concludes the proof of \eqref{pww}.

From \eqref{pww}, we see that for any two smooth curves $\gamma_1:b[0,1] \to \C \backslash \{a_1,\ldots,a_m\}$ and $\gamma_2: [0,1] \to \C \backslash \{b_1,\ldots,b_n\}$ with $\gamma_1(0)=\gamma_2(0)$ and $\int_{\gamma_1} f, \int_{\gamma_2} g \neq \infty$, the quantity $P( \gamma_1(1), \gamma_2(1))$ is a function of $\int_{\gamma_1} f$ and $\gamma_2(1)$.  Applying the analogue of \eqref{pww} with the rules of $f,g$ reversed, we conclude that 
$P( \gamma_1(1), \gamma_2(1))$ is a function of $\int_{\gamma_1} f$ and $\int_{\gamma_2} g$.  Thus, there is a function $Q: \Omega_1 \times \Omega_2 \to \C$ such that
\begin{equation}\label{pgg}
 P( \gamma_1(1), \gamma_2(1) ) = Q( \int_{\gamma_1} f, \int_{\gamma_2} g )
\end{equation}
where $\Omega_1$ is the set of all finite values of $\int_{\gamma_1} f$ for smooth $\gamma_1: [0,1] \to \C \backslash \{a_1,\ldots,a_m\}$, and $\Omega_2$ is the set of all finite values of $\int_{\gamma_2} g$ for smooth $\gamma_2: [0,1] \to \C \backslash \{b_1,\ldots,b_n\}$.  Note from Lemmas \ref{ass}, \eqref{ass-2} that $\Omega_1$, $\Omega_2$ are equal to the entire complex plane $\C$ with at most one coset of $\Gamma_1, \Gamma_2$ respectively deleted.  

Since $f$ and $g$ have local primitives which are analytic and non-constant, and hence open, we see that the function $Q$ is continuous.  By further use of these local primitives, we see that for all $(z_0,w_0) \in \Omega_1 \times \Omega_2$, the function $Q(z,w_0)$ is holomorphic for $z$ in a sufficiently punctured disk $B(z_0,\eps) \backslash \{z_0\}$, and thus (by the continuity of $Q$ at $z_0$ and Morera's theorem) is holomorphic on the unpunctured disk $B(z_0,\eps)$ also.    Thus $Q(z,w)$ is holomorphic in the first variable $z$, and is similarly holomorphic in the second variable; thus it is a smooth biholomorphic function on $\Omega_1 \times \Omega_2$.  

Let $(z_0,w_0) \in \Omega_1 \times \Omega_2$.  For $z \in B(z_0,\eps) \backslash \{z_0\}$ and $w \in B(w_0,\eps) \backslash \{w_0\}$  for $\eps>0$ small enough, we may use local primitives to find smooth $\gamma_1: [0,1] \to \C \backslash \{a_1,\ldots,a_m\}$ and 
$\gamma_2: [0,1] \to \C \backslash \{b_1,\ldots,b_n\}$ with $\gamma_1(0)=\gamma_2(0)=0$, $\int_{\gamma_1} f = z$, $\int_{\gamma_2} g = w$, and $f(\gamma_1(1)), g(\gamma_2(1)) \neq 0$.  If one then elongates $\gamma_1$ by an infinitesimal line segment from $\gamma_1(1)$ to $\gamma_1(1)+f(\gamma_1(1))^{-1} dt$ for some infinitesimal $dt$, and simultaneously elongates $\gamma_2$ by an infinitesimal line segment from $\gamma_2(1)$ to $\gamma_2(1)+g(\gamma_2(1))^{-1} dt$, we see from\footnote{Note that while \eqref{pgg} is stated for smooth curves, it extends automatically to continuous curves (and in particular to the concatenation of two smooth curves) by a limiting argument.} \eqref{pgg} that
$$
 P( \gamma_1(1) + f(\gamma_1(1))^{-1} dt, \gamma_2(1) + g(\gamma_2(1))^{-1} dt ) = Q( z + dt, w + dt ) + o(dt)$$
and thus by the chain rule and sending $dt$ to zero
$$ f(\gamma_1(1))^{-1} P_1( \gamma_1(1), \gamma_2(1) ) + g(\gamma_2(1))^{-1} P_2( \gamma_1(1), \gamma_2(1) ) 
= Q_1(z,w) + Q_2(z,w).$$
By \eqref{pilo}, the left-hand side vanishes, and thus
$$ Q_1(z,w) + Q_2(z,w) = 0$$
for all $z \in B(z_0,\eps) \backslash \{z_0\}$ and $w \in B(w_0,\eps) \backslash \{w_0\}$; by the smoothness of $Q$ and the arbitrariness of $z_0,w_0$ we conclude that $Q_1(z,w) + Q_2(z,w) = 0$ for all $(z,w) \in \Omega_1 \times \Omega_2$.  In particular, for any complex number $\zeta \in \C$, the function
$$ z \mapsto Q( z, \zeta - z ),$$
which is defined outside of a discrete subset of $\C$, has zero derivative and thus extends to a constant function on $\C$.  In other words, we can find a function $H: \C \to \C$ such that
$$ Q(z,w) = H(z+w)$$
for all $(z,w) \in \Omega_1 \times \Omega_2$.  Since $Q$ is biholomorphic and $\Omega_1,\Omega_2$ are complements of discrete sets, we conclude that $H$ is holomorphic on all of $\C$, and is thus entire, and the proposition follows.
\end{proof}

We now analyse the entire function $H$ given by the above proposition.  Observe from \eqref{lor} that one can freely modify $\int_{\gamma_1} f$ by any element of $\Gamma_1$ without affecting the endpoint $\gamma_1(1)$.  As a consequence, we conclude that $H$ must be periodic with respect to translations by $\Gamma_1$.  Similarly it is periodic with respect to translations by $\Gamma_2$, and is thus periodic with respect to the combined subgroup $\Gamma_1+\Gamma_2$.  If this subgroup contains two non-zero elements that are not real multiples of each other, then $H$ thus descends to a holomorphic function on a torus, and is thus constant by Liouville's theorem, which makes $P$ constant, in which case Theorem \ref{const-0} is trivial.  Thus we may assume that $\Gamma_1+\Gamma_2$ does not contain two non-zero elements that are not real multiples of each other, and is thus contained in a line; without loss of generality we may normalise (after rescaling the $a$ and $b$ variables) so that $\Gamma_1+\Gamma_2 \leq 2\pi i \R$.  If $\Gamma_1+\Gamma_2$ is dense in $2\pi i \R$, then $H$ is constant along $2\pi i \R$; since non-constant analytic functions have isolated zeroes, we conclude that $H$ is constant on all of $\C$, so again $P$ is constant and Theorem \ref{const-0} is trivial in this case.  Thus we may assume that $\Gamma_1+\Gamma_2$ is discrete; without loss of generality we may normalise so that either $\Gamma_1+\Gamma_2 = \{0\}$ or $\Gamma_1 + \Gamma_2 = 2\pi i \Z$.  As we shall see shortly, these two cases correspond to the additive and multiplicative cases\footnote{In principle, the case when $\Gamma_1 + \Gamma_2$ is a rank two lattice would correspond to a case in which $P$ is given in terms of an elliptic curve group law instead of addition and multiplication, but such laws can only be expressed in terms of rational functions rather than by polynomials and so do not actually arise in this analysis.  However, if $P$ were generalised to be a regular map on the product of two curves, rather than a polynomial map on the product of two affine lines $k$, then one would have to also consider constructions arising from elliptic curve group laws; and in higher dimensions one would also have to consider more general abelian varieties.  It seems of interest to extend Theorem \ref{const-0} to these settings, but we will not pursue this matter here.} of Theorem \ref{const-0}.

First suppose that $\Gamma_1+\Gamma_2 = \{0\}$, so that $f$ and $g$ have no poles with non-zero residues.  In this case, we see from \eqref{lor} that
$$ \int_{\gamma_1} f = F(\gamma_1(1)); \int_{\gamma_2} f = G(\gamma_2(1))$$
and hence from Proposition \ref{Add}
$$ P( a, b ) = H( F(a) + G(b) )$$
whenever $a,b \in \C$ and $F(a), G(b) \neq \infty$.  If $F$ has a pole at some finite $a_*$, then by Rouche's theorem $F(a)$ can take any sufficiently large value for $a$ in a neighbourhood of $a_*$, and the above equation then forces the entire function $H$ to be bounded in a neighbourhood of infinity, and is thus constant by Liouville's theorem, in which case Theorem \ref{const-0} is trivial.  Thus we may assume that the rational function $F$ has no finite poles, and is thus a polynomial; similarly we may assume that $G$ is a polynomial.  By holding $b$ fixed and sending $a$ to infinity, we conclude that $H$ is of polynomial growth; being entire, we conclude from the generalised Liouville theorem that $H$ is itself a polynomial, and we have obtained the additive conclusion of Theorem \ref{const-0}.  

Now suppose that $\Gamma_1+\Gamma_2 = 2\pi i\Z$.  Then all the $\alpha_i, \beta_j$ are integers, and from \eqref{lor} we have
$$ \int_{\gamma_1} f \in F(\gamma_1(1)) + \log \tilde F( \gamma_1(1) )$$
and
$$ \int_{\gamma_2} g \in G(\gamma_2(1)) + \log \tilde G( \gamma_2(1) )$$
for some rational functions $\tilde F, \tilde G$, with at least one of $\tilde F, \tilde G$ non-constant, and all $\gamma_1, \gamma_2$ for which $F(\gamma_1(1)), G(\gamma_2(1)) \neq \infty$ and $\tilde F(\gamma_1(1)), \tilde G(\gamma_2(1)) \neq 0,\infty$.  From Proposition \ref{Add}, we conclude that
$$ P(a,b) = H( F(a) + G(b) + \log \tilde F(a) \tilde G(b) )$$
whenever $F(a), F(b) \neq \infty$ and $\tilde F(a), \tilde G(b)\neq 0,\infty$; note that the right-hand side is well-defined since $H$ is $2\pi i\Z$-periodic.   We use this periodicity to write $H(z) = \tilde H( \exp(z) )$ for some holomorphic function $\tilde H: \C \backslash \{0\} \to \C$, and conclude that
\begin{equation}\label{pab}
 P(a,b) = \tilde H( \exp(F(a) + G(b)) \tilde F(a) \tilde G(b) ).
\end{equation}
If $F$ has a pole at some finite $a_*$, then for fixed $b$, $\exp(F(a) + G(b)) \tilde F(a) \tilde G(b)$ can take any sufficiently large value in a neighbourhood of $a_*$, and so by arguing as in the additive case $\tilde H$ and hence $P$ is constant, in which case the claim is trivial.  Similarly if $G$, $\tilde F$, $\tilde G$ have finite poles, so we may assume that $F, G, \tilde F, \tilde G$ are all polynomials.

At least one of $\tilde F,\tilde G$ is non-constant; without loss of generality, assume that $\tilde F$ is non-constant, and so has a zero at some finite $a_*$.  Letting $a$ approach $a_*$ while fixing $b$ at some generic value, we see that $\exp(F(a) + G(b)) \tilde F(a) \tilde G(b)$ can attain any sufficiently small non-zero value, and conclude from \eqref{pab} that $\tilde H$ remains bounded in a neighbourhood of the origin, so that the origin is a removable singularity and $\tilde H$ can be extended to an entire function.

Suppose that $F$ is non-constant.  Then, fixing $b$ to be some generic value, we see from Rouche's theorem that any sufficiently large complex number $z$ can be represented as $\exp(F(a) + G(b)) \tilde F(a) \tilde G(b)$ for some $a = O( \log |z| )$.  We conclude that $H$ grows at most like a power of a logarithm (i.e. $H(z) = O( \log^{O(1)} |z| )$) and so by the generalised Liouville theorem, $H$ is constant and we are again done.  Thus we may assume that $F$ is constant, and similarly $G$ is constant.  We may then absorb the $\exp(F(a)+G(b))$ factor in \eqref{pab} into $\tilde F$ or $\tilde G$, and we have obtained the multiplicative conclusion of Theorem \ref{const-0}.  The proof of Theorem \ref{const-0} is now complete.

Now we move from the algebraically closed field setting to the nonstandard finite field setting.

\begin{theorem}\label{const}  Let $\F$ be a nonstandard finite field of characteristic zero.  Let $P: \F^2 \to \F$ be a polynomial with the property that the set
\begin{equation}\label{lka2}
 \{ (P(a,c), P(a,d), P(b,c), P(b,d)) \mid  a,b,c,d \in \F  \}
\end{equation}
is not Zariski dense in $\overline{\F}^4$.  Then one of the following statements hold:
\begin{itemize}
\item[(i)] (Additive structure)  There exist polynomials $Q, F, G: \F \to \F$ such that $P(x,y) = Q(F(x)+G(y))$ for all $x,y \in \F$.
\item[(ii)] (Multiplicative structure) There exist polynomials $Q, F, G: \F \to \F$ such that $P(x,y) = Q(F(x) G(y))$ for all $x,y \in \F$.
\end{itemize}
\end{theorem}

\begin{proof} Set $k := \overline{\F}$.  If \eqref{lka2} is not Zariski dense, then from Lemma \ref{sz} we see that \eqref{lka} is not Zariski dense either.  Thus by Theorem \ref{const-0}, we have $P(x,y) = Q(F(x)+G(y))$ or $P(x,y) = Q(F(x)G(y))$ for some polynomials $Q,F,G: k \to k$ defined over $k$.  The remaining difficulty is to replace these polynomials by polynomials that are defined over $\F$.

Note that we may assume that $F, G$ are non-constant, as the claim is trivial otherwise.

Suppose first that we have the additive structure $P(x,y) = Q(F(x)+G(y))$.  Clearly $P(\F,\F) \subset \F$, and hence $|F(\F) + G(\F)| \ll |\F|$.
The field of definition of $F,G$ is a finite extension $\F'$ of $\F$.  It will be convenient (particularly when we turn to the more difficult multiplicative case) to coordinatise this field $\F'$ as follows.  Let $n := [\F':\F]$ be the index of the extension $\F'$ over $\F$.  Let $S: \F \to \F$ be a generic monic polynomial of degree $n$, so that $S$ is irreducible and the roots $t_1,\ldots,t_n \in \F'$ are distinct.  We can then identify $\F'$ with the field
$$  \{ ( R(t_1), \ldots, R(t_n) )\mid R \in \F[t] \} \subset k^n$$
where $R$ ranges over the polynomials of one variable in $t$, where we give $k^n$ the ring structure of componentwise addition and multiplication.  This can  be viewed as an $n$-dimensional subspace (over $\F$) of $k^n$, with basis $e_0,\ldots,e_{n-1}$ given by
$$ e_i := (t_1^i,\ldots,t_n^i).$$
The restriction $F, G: \F \to \F'$ of $F,G$ to $\F$ can thus be written in components as $F = \sum_{i=0}^{n-1} F_i e_i$, $G = \sum_{i=0}^{n-1} G_i e_i$ for some polynomials $F_i,G_i: \F \to \F$ defined over $\F$.  We can then extend these maps to polynomial maps $\tilde F, \tilde G: k \to k^n$ by the same formulae:
$$ \tilde F := \sum_{i=0}^{n-1} F_i e_i; \quad \tilde G := \sum_{i=0}^{n-1} G_i e_i.$$
As $F, G$ are non-constant, $\tilde F, \tilde G$ are non-constant also.

We then have $|\tilde F(\F) + \tilde G(\F)| \ll |\F|$.  Since
$$ \sum_{v \in \tilde F(\F) + \tilde G(\F)} |\tilde F(\F) \cap (v-\tilde G(\F))| = |\tilde F(\F)| |\tilde G(\F)| \gg |\F|^2$$
and $|\tilde F(\F) \cap (v-\tilde G(\F))| \leq |\F|$, we conclude that
$$ |\tilde F(\F) \cap (v-\tilde G(\F))| \gg |\F|$$
for $\gg |\F|$ values of $v$.  In particular, $\tilde F(k) \cap (v-\tilde G(k))$ is infinite for infinitely many $v \in k^n$.  As $\tilde F, \tilde G$ are non-constant polynomial maps, $\tilde F(k), \tilde G(k)$ are irreducible quasiprojective curves in $k^n$, and so
$$ \overline{\tilde F(k)} = v - \overline{\tilde G(k)}$$
for infinitely many $v$.  In particular, the symmetry group $H := \{ v \in k^n\mid \overline{\tilde F(k)} = v + \overline{\tilde F(k)} \}$ is infinite.  But this symmetry group is an algebraic subgroup of the additive group $k^n$ of dimension at most $1$, and so it is a line.  This implies that $\overline{\tilde F(k)}$ and $\overline{\tilde G(k)}$ are translates of the same line, and so $F(x)+G(y)$ is an affine (over $\F$) function of $F_i(x)+G_i(y)$ for some $0 \leq i \leq n-1$.  From this it is easy to see that
$$ P(x,y) = Q(F(x)+G(y)) = Q_i(F_i(x)+G_i(y))$$
for some polynomial $Q_i: \F \to \F$ defined over $\F$, and conclusion (i) of Theorem \ref{const} follows.

Now we turn to the multiplicative case $P(x,y) = Q(F(x) \cdot G(y))$.  By arguing as before we have
$$|\tilde F(\F) \cdot \tilde G(\F)| \ll |\F|.$$
Since $F, G$ have only boundedly many zeroes, $\tilde F(\F)$ and $\tilde G(\F)$ have all but boundedly many points in $(k^\times)^n$, where $k^\times := k \backslash \{0\}$ is the multiplicative group of $k$.  If we define $\tilde F(\F)^\times := \tilde F(\F) \cap (k^\times)^n$ and $\tilde G(\F)^\times := \tilde G(\F) \cap (k^\times)^n$, then we have
$$|\tilde F(\F)^\times \cdot \tilde G(\F)^\times| \ll |\F|.$$
Arguing as in the additive case, we conclude that the symmetry group
$H := \{ v \in (k^\times)^n\mid \overline{\tilde F(k)}^\times = v \cdot \overline{\tilde F(k)}^\times \}$ is infinite, where $\overline{\tilde F(k)}^\times := \overline{\tilde F(k)} \cap (k^\times)^n$.  This is a algebraic subgroup of the multiplicative group $(k^\times)^n$ that is contained in a dilate of the connected curve $\overline{\tilde F(k)}^\times$, and is therefore a connected curve.  Using the Lefschetz principle to embed the field of definition of $H$ into $\C$, we can thus view $H$ as a connected one-dimensional algebraic subgroup of $(\C^\times)^n$.  From Lie group theory, a one-dimensional connected subgroup of $(\C^\times)^n$ takes the form $\{ (\exp( \alpha_1 z), \ldots, \exp(\alpha_n) z)\mid z \in \C \}$ for some complex numbers $\alpha_1,\ldots,\alpha_n$.  By inspecting the limits as $z \to \infty$, one can check that such subgroups are algebraic only when $\alpha_1,\ldots,\alpha_n$ are \emph{commensurate}, in that they lie in a dilate $\Q$, in which case $H$ can be parameterised as $\{ (t^{m_1},\ldots,t^{m_n})\mid t \in \C^\times\}$ for some integers $m_1,\ldots,m_n$.  As $\overline{\tilde F(k)}^\times$ is contained in a dilate of $H$, we see (after restricting back from $\C$ to $k$) that
$$ \tilde F(k) \subset \{ (c_1 t^{m_1}, \ldots, c_n t^{m_n})\mid t \in k \}$$
for some constants $c_1,\ldots,c_n \in k$.  Factoring the components of $\tilde F$ into monic irreducible polynomials, we conclude that 
$$ \tilde F(x) = (c_1 F'(x)^{a_1}, \ldots, c_n F'(x)^{a_n})$$
for some monic polynomial $F': k \to k$ defined over $k$, and some natural numbers $a_1,\ldots,a_n$ (a scalar multiple of the $m_1,\ldots,m_n$).

The Frobenius endomorphism $\Frob_\F$ generates the Galois group $\operatorname{Gal}(\F'/\F) \equiv \Z/n\Z$, which acts transitively on the roots $t_1,\ldots,t_n$ of $R$.  This group then also acts on the polynomial components $c_i F'(x)^{a_i}$ of $\tilde F$, by acting on the coefficients of these polynomials; by taking degrees, we conclude that $a_1=\ldots=a_n=a$, so that $F'$ is invariant with respect to the Frobenius action and is thus defined over $\F$.  Thus $F$ is a scalar multiple of a polynomial $(F')^a$ defined over $\F$.  A similar argument shows that $G$ is also a scalar multiple of a polynomial $(G')^b$ defined over $\F$, and so $P$ can be written in the form $P(x,y) = Q'((F')^a(x) (G')^b(y))$ for some polynomial $Q'$; as $P, F', G'$ are Frobenius-invariant (with $F',G'$ non-constant), $Q'$ is also Frobenius-invariant and thus defined over $\F$, giving the required representation of $P$.
\end{proof}

Combining Theorem \ref{const} with Theorem \ref{expand-thm}, we immediately obtain Theorem \ref{expand-thm-modas-nonst}.

\section{Weak expansion}\label{weak-sec}

Now we prove Theorem \ref{expand-thm-weak-nonst} (and hence Theorem \ref{expand-thm-weak}), by combining the above arguments with the Fourier-analytic arguments from \cite{hls}.  Suppose for contradiction that we can find $\F, P$
which obey the hypotheses of this theorem, but do not obey any of the four conclusions (i)-(iv).  Applying Theorem \ref{expand-thm-modas-nonst} (and noting that moderate asymmetric expansion certainly implies weak expansion), we see that we must have either additive structure in the sense that
\begin{equation}\label{padd-0}
P(x_1,x_2) = Q(F_1(x_1)+F_2(x_2))
\end{equation}
for some polynomials $Q, F_1, F_2: \F \to \F$, or multiplicative structure in the sense that
\begin{equation}\label{pmult-0}
P(x_1,x_2) = Q(F_1(x_1) F_2(x_2))
\end{equation}
for some polynomials $Q, F_1, F_2: \F \to \F$.

Suppose first that we have additive structure \eqref{padd-0}.  As the conclusion (iii) of Theorem \ref{expand-thm-weak-nonst} is assumed to fail, we can find a nonstandard subset $A \subset \F$ with 
\begin{equation}\label{abig}
|A| \ggg |\F|^{1-1/16}
\end{equation}
and
$$ |P(A,A)| \lll |\F|^{1/2} |A|^{1/2}.$$
We can assume that $P$ is non-constant, so that $Q$ is non-constant.  We can also assume $F_1,F_2$ non-constant, as the claim is immediate otherwise.
In particular, $Q$ has fibres of bounded cardinality and so $|Q(B)| \gg |B|$ for all nonstandard $B \subset \F$.  We conclude that
\begin{equation}\label{f12}
 |F_1(A) + F_2(A)| \lll |\F|^{1/2} |A|^{1/2}.
\end{equation}
If we let $B := F_2(A) \times F_1(A) \subset \F^2$ and $C := (F_1(A)+F_2(A)) \times (F_1(A)+F_2(A)) \subset \F^2$, then we have
$$ (F_1(t), F_2(t)) + B \subset C$$
for all $t \in A$, and hence
\begin{equation}\label{afd}
 \sum_{t \in \F} \sum_{(x,y) \in \F^2} 1_B(x,y) 1_C(x+F_1(t),y+F_2(t)) \geq |A| |B| \gg |A|^3.
\end{equation}
We now use Fourier analysis to expand
\begin{equation}\label{ba}
 1_B(x,y) = \sum_{\chi_1,\chi_2 \in \hat \F} \hat 1_B(\chi_1,\chi_2) \chi_1(x) \chi_2(y)
 \end{equation}
where $\hat \F$ is the space of nonstandard (additive) characters on $\F$, that is to say the nonstandard homomorphisms $\chi: \F \to {}^* S^1$ from $\F$ (viewed as an additive group) to the unit circle, and $\hat 1_B(\chi_1,\chi_2)$ are the Fourier coefficients
$$ \hat 1_B(\chi_1,\chi_2) := \E_{x,y \in \F} 1_B(x,y) \overline{\chi_1(x)} \overline{\chi_2(y)}.$$
Similarly, we may write
\begin{equation}\label{ca}
 1_C(x+F_1(t),y+F_2(t)) = \sum_{\chi_1,\chi_2 \in \hat \F} \hat 1_C(\chi_1,\chi_2) \chi_1(x) \chi_2(y) \chi_1(F_1(t)) \chi_2(F_2(t)).
\end{equation}
Multiplying \eqref{ca} by the complex conjugate of \eqref{ba} and summing using the orthogonality properties of characters, we may expand the left-hand side of \eqref{afd} as
$$ |\F|^2 \sum_{\chi_1,\chi_2 \in \hat \F} \hat 1_B(\chi_1,\chi_2) \hat 1_C(\chi_1,\chi_2) \sum_{t \in\F} \chi_1(F_1(t)) \chi_2(F_2(t)).$$
Using the trivial bound $|\chi_1|, |\chi_2| = 1$, we see that the contribution of any given pair $(\chi_1,\chi_2)$ to the above sum is
$$ |\F|^2 (|B|/|\F|^2) (|C|/|\F|^2) |\F| $$
which by \eqref{f12} and the trivial bound $|B| \leq |A|^2$ is $o(|A|^3)$.  We thus have
$$ \sum_{(\chi_1,\chi_2) \in \hat \F \times \hat \F \backslash E} |\hat 1_B(\chi_1,\chi_2)| |\hat 1_C(\chi_1,\chi_2)| |\sum_{t \in\F} \chi_1(F_1(t)) \chi_2(F_2(t))| \gg |A|^3/|\F|^2$$
for any subset $E \subset \hat \F \times \hat \F$ of bounded cardinality.  On the other hand, from the Plancherel identity we have
$$ \sum_{(\chi_1,\chi_2) \in \hat \F \times \hat \F} |\hat 1_B(\chi_1,\chi_2)|^2  = |B|/|\F|^2 \leq |A|^2/|\F|^2$$
and
$$ \sum_{(\chi_1,\chi_2) \in \hat \F \times \hat \F} |\hat 1_C(\chi_1,\chi_2)|^2  = |C|/|\F|^2 \lll |A|/|\F|$$
and hence by Cauchy-Schwarz
$$ \sum_{(\chi_1,\chi_2) \in \hat \F \times \hat \F \backslash E} |\hat 1_B(\chi_1,\chi_2)| |\hat 1_C(\chi_1,\chi_2)| 
\lll |A|^{3/2}/|\F|^{3/2}.$$
By H\"older's inequality, we conclude the existence of $\chi_1,\chi_2 \in \hat \F$ with $(\chi_1,\chi_2) \not \in E$ such that
$$ |\sum_{t \in\F} \chi_1(F_1(t)) \chi_2(F_2(t))| \ggg |A|^{3/2}/|\F|^{1/2};$$
in particular, from the hypothesis \eqref{abig} we certainly have
\begin{equation}\label{tff}
 |\sum_{t \in\F} \chi_1(F_1(t)) \chi_2(F_2(t))| \ggg |\F|^{1/2}.
 \end{equation}
(with some room to spare).
As $E$ was an arbitrary set of bounded cardinality, we conclude that \eqref{tff} holds for an unbounded number of pairs $(\chi_1,\chi_2) \in \hat \F \times \hat \F$.  In particular, it holds for a pair $(\chi_1,\chi_2)$ with $(\chi_1,\chi_2) \neq (0,0)$.

As is well known, the group $\hat \F$ of additive characters of $\F$ is isomorphic to $\F$ itself, and there exists a non-trivial generator $\chi_0 \in \hat \F$ such that any other character $\chi \in \hat\F$ takes the form $\chi(x) = \chi_0(ax)$ for some $a \in \F$. (Indeed, if $\F$ is a standard finite field of characteristic $p$, one can take $\chi_0$ to be $\chi_0(x) := e^{2\pi i \phi(x)/p}$ where $\phi$ is any linear surjection from $\F$ (viewed as a vector space over $\F_p$) to $\F_p$, and the nonstandard case then follows by {\L}os's theorem.)  We may thus write
$$ \sum_{t \in \F} \chi_1(F_1(t)) \chi_2(F_2(t)) = \sum_{t \in \F} \chi_0( a F_1(t) + b F_2(t) )$$
for some $a,b \in \F$, not both zero.  On the other hand, from the Weil bound on character sums (see e.g. \cite{reiter}) and {\L}os's theorem we have
$$ |\sum_{t \in \F} \chi_0(P(t)) | \ll |\F|^{1/2}$$
whenever $\chi_0$ is a non-trivial additive character and $P$ is a non-constant (external) polynomial.  We conclude that $a F_1(t) + b F_2(t)$ must be constant, and it is then an easy matter to write $P$ in the desired form for Theorem \ref{expand-thm-weak-nonst}(i).

Now we suppose that we are in the case when $P$ has multiplicative structure \eqref{pmult-0}.  Arguing as in the additive case, we then can find a nonstandard subset $A \subset \F$ obeying \eqref{abig} and
$$
 |F_1(A) \cdot F_2(A)| \lll |\F|^{1/2} |A|^{1/2}.
$$
Again, we may assume that $Q,F_1,F_2$ are non-constant.  By removing a bounded number of elements from $A$, we may also assume that $0 \not \in F_1(A), F_2(A)$, thus $F_1(A), F_2(A)$ take values inside the multiplicative group $\F^\times := \F \backslash \{0\}$.

We now run the same Fourier analytic argument used in the additive case, but with the underlying abelian group used for Fourier analysis being the multiplicative group $\F^\times$ rather than the additive group $\F$.  Let $\hat{\F^\times}$ be the space of nonstandard (multiplicative) characters on $\F$, that is to say the nonstandard homomorphisms $\psi: \F^\times \to {}^* S^1$ from $\F^\times$ (viewed as a multiplicative group) to the unit circle.  By repeating the previous arguments with the obvious changes, we conclude that
\begin{equation}\label{psif}
 |\sum_{t \in\F} \psi_1(F_1(t)) \psi_2(F_2(t))| \ggg |\F|^{1/2}
\end{equation}
for an unbounded number of pairs $(\psi_1,\psi_2) \in \hat{\F^\times} \times \hat{\F^\times}$.

We factor $F_1,F_2$ as
$$ F_1 = c_1 P_1^{a_1} \ldots P_r^{a_r}, \quad F_2 = c_2 P_1^{b_1} \ldots P_r^{b_r}$$
where $c_1,c_2 \in \F^\times$, $P_1,\ldots,P_r$ are a bounded number of distinct monic irreducible polynomials, and $a_1,\ldots,a_r, b_1,\ldots,b_r$ are standard natural numbers.  If $(a_1,\ldots,a_r)$ and $(b_1,\ldots,b_r)$ are linearly dependent, then $P$ can be expressed in the desired form for Theorem \ref{expand-thm-weak-nonst}(ii), so suppose that these vectors are linearly independent.   We can rewrite the left-hand side of \eqref{psif} as
$$
|\sum_{t \in \F} \prod_{i=1}^r \psi_1^{a_i} \psi_2^{b_i}( P_i(t) )|.$$
By the Weil bound for multiplicative characters (see e.g. \cite[Corollary 2.3]{wan}), and the fact that $P_1,\ldots,P_r$ are distinct monic irreducible polynomials, this expression is $O( |\F|^{1/2} )$ unless $\psi_1^{a_i} \psi_2^{b_i}$ is trivial for all $i=1,\ldots,r$.  But as $(a_1,\ldots,a_r)$ and $(b_1,\ldots,b_r)$ are linearly independent, this can only occur if $\psi_1,\psi_2$ both have order at most $C$, for some standard $C$.  On the other hand, the multiplicative group $\F_q^\times$ of a finite field of order $q$ is isomorphic as a group to the cyclic group $\Z/(q-1)\Z$, and so the dual group $\hat{\F_q^\times}$ also is isomorphic to this group.  In particular, for any $k$, there are at most $k$ characters in $\hat{\F_q^\times}$ of order $k$.  Applying {\L}os's theorem, we conclude the same statement holds for $\hat{\F}$.  Thus there are only a bounded number of pairs $(\psi_1,\psi_2)$ for which \eqref{psif} holds, giving the desired contradiction.

\section{The second algebraic constraint}\label{second-sec}

We are now ready to establish Theorem \ref{expand-thm-ass-nonst} (and thus Theorem \ref{expand-thm-ass}).  
Suppose for contradiction that we can find $\F, P$ which obey the hypotheses of this theorem, but do not obey any of the four conclusions (i)-(iv).  Applying Theorem \ref{expand2-thm} and Theorem \ref{const}, we see that the only possibility is that
there exist geometrically irreducible one-dimensional quasiprojective varieties $V',W',U'$ defined over $\F$ and dominant regular maps $f: V' \to \overline{\F}$, $g: W' \to \overline{\F}$, $h: U' \to \overline{\F}$ defined over $\F$ such that the variety
\begin{equation}\label{var}
  \{ (v',w',u') \in V' \times W' \times U'\mid P( f(v'), g(w') ) = h(u') \}
\end{equation}
is not irreducible.  As $\F$ has characteristic zero, the curves $U',V',W'$ are generically smooth, so by removing a finite number of points from each curve we may assume that $U',V',W'$ are smooth curves.

Note that if $P$ only depends on the first variable (thus $P(x,y) = Q(x)$ for some polynomial $x$) then we have either conclusion (i) or (ii) of Theorem \ref{expand-thm-ass-nonst}, so we may assume that $P$ does not depend purely on the first variable.  Similarly, we may assume that $P$ does not depend purely on the second variable.  In particular, $P$ is non-constant and thus dominant.

We now claim that the varieties
\begin{equation}\label{zao}
 \{ (v',w') \in V' \times W'\mid P(f(v'),g(w')) = t \}
\end{equation}
are reducible for generic $t \in \overline{\F}$.  Indeed, suppose this were not the case; then (as the set of $c$ for which \eqref{zao} is reducible is constructible) this implies that \eqref{zao} is irreducible for generic $t \in \F$.  Now consider an irreducible component of the variety \eqref{var}.  This can be viewed as the relative product of the varieties
\begin{equation}\label{so}
 \{ (v',w',t) \in V' \times W' \times \overline{\F}\mid P(f(v'),g(w')) = t \}
\end{equation}
and
\begin{equation}\label{what}
 \{ (u,t) \in U \times \overline{\F}\mid h(u') = t \}
 \end{equation}
over $\overline{\F}$.  As discussed previously, the generic fibres of the first factor \eqref{so} of this relative product are irreducible.  Hence, any irreducible component of this relative product must be generically equal to a relative product of \eqref{so} with an irreducible component of \eqref{what}.  But \eqref{what} is isomorphic to $U$ and is thus already irreducible, so that \eqref{var} is irreducible, a contradiction.  Thus \eqref{zao} is generically reducible.  

We may rephrase the previous conclusion in the language of linear systems as the assertion that the linear system $((v',w') \mapsto P(f(v'),g(w'))-t)_{t \in \overline{\F}}$ on $V' \times W'$ is reducible.  We also observe that this system has no fixed components, as the varieties \eqref{zao} are disjoint as $c$ varies.
We may then apply (the generalisation of) Bertini's second theorem (see e.g. \cite[Theorem (5.3)]{klei}) to conclude that this system is composite with a pencil, which means that there is a regular map $Q: (V' \times W' \backslash \Sigma) \to C$ from a dense subvariety $(V' \times W') \backslash \Sigma$ of $V' \times W'$ into an irreducible algebraic curve $U'$ and a regular map $h: U' \to \overline{\F}$ of degree $d \geq 2$ (in the sense that the fibres $k^{-1}(\{t\})$ generically have cardinality $d$) such that
\begin{equation}\label{pvw}
 P(f(v'),g(w')) = h( Q( v',w' ) )
\end{equation}
for generic $(v',w') \in V' \times W'$.  (Indeed, one can simply take $C$ to be the curve given by the Chow coordinates of the irreducible components of generic fibres \eqref{zao}, with the obvious generically defined maps $Q,h$.)  Note that as $P, f, g$ are dominant, the maps $h, Q$ must also be dominant.

We now begin transferring the base field $\overline{\F}$ to the complex field $\C$ in order to use the theory of Riemann surfaces.  As in Section \ref{algo}, we may find an algebraically closed subfield $k$ of $\overline{\F}$ of finite transcendence degree over $\Q$ such that all the varieties $U',V',W',\Sigma$ and regular maps $P,f,g,h,Q$ used above are defined over $k$, and then embed $k$ into the complex field $\C$, so that all these varieties can also be viewed as complex varieties (with the regular maps being complex analytic maps).   

Henceforth we will work over the complex field.  We will exploit the well-known fact (see e.g. \cite[Corollary 6.10]{hart}) that any smooth quasiprojective algebraic curve over an algebraically closed field $k$ is isomorphic to an open dense subset of a smooth projective algebraic curve over $k$, or that is to say a projective algebraic curve with a finite number of points deleted.  Furthermore (see e.g. \cite[Proposition 6.8]{hart}), any regular map from an open dense subset of a projective algebraic curve $C$ to a projective variety $V$ can be uniquely completed to a regular map from $C$ to $W$.  As such, we can view the algebraic curves $U',V',W'$ in the above discussion as open dense subsets of smooth projective curves $\tilde U, \tilde V, \tilde W$ respectively, and we can also view $\C$ as an open dense subset of the projective line $\mathbb{P}^1(\C)$.  We can thus complete the regular maps $f: V' \to \C$, $g: W' \to \C$, $h: U' \to \C$ to regular maps $\tilde f: \overline{V} \to \mathbb{P}^1$, $\tilde g: \overline{W} \to \mathbb{P}^1$, $\tilde h: \tilde{U} \to \mathbb{P}^1(\C)$.  Unfortunately, the regular maps $P, Q$, being defined on the product of curves rather than on a single curve, do not automatically extend to the projective completion.  However, we may obtain a regular extension $\tilde Q: (\tilde V \times \tilde W) \backslash \Sigma \to \tilde U$ of $Q$ defined outside of a finite subset $\Sigma$ of $\tilde V \times \tilde W$ as follows.  Take the graph 
$$\{ (v,w,Q(v,w)): v \in V, w \in W, Q(v,w) \hbox{ well defined}\} \subset \tilde V \times \tilde W \times \tilde U$$
which is an irreducible constructible set by Proposition \ref{proj}. Its closure $S$ is then an irreducible projective variety\footnote{Here we use the fact that the product of projective varieties is (up to isomorphism) again projective, see e.g. \cite[Lemma I.6.3]{mumf}.} of dimension two.  Projecting $S$ back down to $\tilde V \times \tilde W$, we see that outside of a subset $\Delta$ of $S$ of dimension at most one, this projection has zero-dimensional (hence finite) fibres, with the fibres being at least one-dimensional on $\Delta$.  In particular, the projection $\Sigma$ of $\Delta$ to $\tilde V \times \tilde W$ is finite.  Outside of $\Sigma$, the fibres are finite and generically a single point; a local connectedness argument (using the fact that every point in $\tilde V \times \tilde W$ contains arbitrarily small neighbourhoods which are connected even when one removes those points in which $Q$ is undefined) then shows that the fibres are a single point everywhere in $\tilde V \times \tilde W \backslash \Sigma$, and so $S$ is a graph of some function $\tilde Q: \tilde V \times \tilde W \backslash \Sigma \to \C$.  This implies that the projection from $S \backslash \Sigma$ to $\tilde V \times \tilde W \backslash \Delta$ has degree one (because the generic fibre has cardinality equal to the degree in characteristic zero, see e.g. \cite[Proposition 7.16]{harris}), thus $k(S \backslash \Sigma)$ is isomorphic to $k(\tilde V \times \tilde W \backslash \Delta)$.  As the variety  $\tilde V \times \tilde W \backslash \Delta$ is smooth, it is normal (see e.g. \cite[Theorem II.5.1]{shaf}) and so $k[S \backslash \Sigma]$ is isomorphic to $k[\tilde V \times \tilde W \backslash \Delta]$.  Thus the map $(v,w) \mapsto (v,w,Q(v,w))$ is regular on $\tilde V \times \tilde W \backslash \Sigma$, and so $\tilde Q$ is regular on this domain as well.

Similarly, we have a regular extension $\tilde P: (\mathbb{P}^1(\C) \times \mathbb{P}^1(\C)) \backslash \Lambda \to \mathbb{P}^1(k)$ defined outside of a finite subset $\Lambda$ of $\mathbb{P}^1(\C) \times \mathbb{P}^1(\C)$.  By enlarging $\Sigma$ if necessary, we may assume that $(f(v),g(w)) \not \in \Lambda$ whenever $(v,w) \not \in \Sigma$.  Using these regular extensions and \eqref{pvw}, we see that
\begin{equation}\label{pvw-2}
 \tilde P(\tilde f(v),\tilde g(w)) = \tilde h( \tilde Q( v,w ) )
\end{equation}
for all $(v,w) \in (\tilde V \times \tilde W) \backslash \Sigma$.

The smooth projective curves $\tilde V, \tilde W, \tilde U$, when viewed over $\C$, become compact Riemann surfaces, and thus each have a well-defined \emph{genus}, which is a natural number; see e.g. \cite{grif}.  We now split into several cases depending on the genera $g_{\tilde V}, g_{\tilde W}, g_{\tilde U}$ of the curves $\tilde V, \tilde W, \tilde U$ (viewed as curves over $\C$).  The high genus cases will be relatively easy to eliminate, by using existing theorems in the literature that limit the number of regular maps available between high genus Riemann surfaces, and once we reduce to the case when $\tilde U$ has genus zero, we will be able to conclude the final option (iii) of Theorem \ref{expand-thm-ass-nonst}, after some normalisation.  Unfortunately, we were not able to argue to similarly reduce the genus of $\tilde V$ or $\tilde W$ to zero, which is why Theorem \ref{expand-thm-ass-nonst}(iii) still makes reference to curves of arbitrary genus.

We turn to the details.  First suppose that $g_{\tilde U} \geq 2$ and $g_{\tilde V} < 2$.  In this case, we use the Riemann-Hurwitz formula (see \cite[p. 219]{grif}), which among other things implies that there does not exist a non-constant regular map from a Riemann surface of genus $g$ to a Riemann surface of genus $g'$ whenever $g < g'$.  We conclude that the map $(v,w) \mapsto \tilde Q(v,w)$ is constant in $v$ for all $w$ (outside of $\Sigma$, of course), which implies by \eqref{pvw} and the dominance of $\tilde f, \tilde g$ that $P(v,w)$ is a function of $w$ only, contradicting our previous hypothesis about $P$.  

Now suppose that $g_{\tilde U} \geq 2$ and $g_{\tilde V} \geq 2$.  For this case, we use a classical theorem of de Franchis \cite{def}, that asserts that when two Riemann surfaces $\tilde U, \tilde V$ have genus at least two, then there are only finitely many non-constant regular maps $h_1,\ldots,h_n$ from $\tilde V$ to $\tilde U$.  For each $i=1,\ldots,n$, the set of $w \in \tilde W$ such that $\tilde Q(v,w) = h_i(v)$ for all $v\in \tilde V$ for which $\tilde Q(v,w)$ is well-defined is a closed subset of $\tilde W$, as is the set of $w \in \tilde W$ for which $v \mapsto \tilde Q(v,w)$ is constant.  As these sets cover $\tilde W$, we conclude that $\tilde Q$ is either constant in $v$ or constant in $w$ (outside of $\Sigma$), which implies that $P$ is constant in either $v$ or $w$, leading to a contradiction as before.

Now suppose that $g_{\tilde U}=1$ and $g_{\tilde V} \geq 2$.  Here we can use a variant of the de Franchis theorem, namely the theorem of Tamme \cite{tamme} that for any fixed degree $d$, there are only finitely many non-constant regular maps from $\tilde V$ to $\tilde U$ of degree at most $d$.  The maps $v \mapsto \tilde Q(v,w)$ can easily be seen to have uniformly bounded degree if they are non-constant, and so by repeating the previous argument we again obtain a contradiction.

If $g_{\tilde U}=1$ and $g_{\tilde V}=0$, then we can again use the Riemann-Hurwitz formula to show that there are no non-constant regular maps from $\tilde V$ to $\tilde U$, and we can argue as before to reach a contradiction.  

If $g_{\tilde U}=1$, then the previous arguments (together with their counterparts when $\tilde V$ and $\tilde W$ are switched) handle all cases except when $g_{\tilde V}=g_{\tilde W}=1$; thus $\tilde U, \tilde V, \tilde W$ can all be viewed as elliptic curves (after arbitrarily designating one point on each of $\tilde U, \tilde V, \tilde W$ as the origin).  It is a classical fact (see e.g. \cite[p. 238]{grif}) that every elliptic curve over $\C$ has the structure of a complex abelian group, and specifically to a torus $\C/\Gamma$ for some discrete lattice $\Gamma$ of $\C$.  We can thus form the identifications $\tilde U \equiv \C/\Gamma_{\tilde U}$, $\tilde V \equiv \C/\Gamma_{\tilde V}$, $\tilde W \equiv \C/\Gamma_{\tilde W}$ on the level of Riemann surfaces (and complex abelian groups) for some lattices $\Gamma_{\tilde U}, \Gamma_{\tilde V}, \Gamma_{\tilde W}$.  It is then known (see \cite[Lemma 4.9]{hart}) that any regular map from $\tilde V$ to $\tilde U$ corresponds to a map from $\C/\Gamma_{\tilde V}$ to $\C/\Gamma_{\tilde U}$ of the form $z \mapsto z_0 + m z$, where $z_0 \in \C$ and $m$ is a complex number such that $m \Gamma_{\tilde V} \subset \Gamma_{\tilde U}$.  In particular, $m$ is constrained to a discrete subgroup of $\C$ (the rank of which depends on whether $\tilde U, \tilde V$ are isogenous, and whether they have complex multiplication).  For each such $m$, the set of $w \in \tilde W$ for which (the completion of) the map $v \mapsto \tilde Q(v,w)$ corresponds to a map of the form $z \mapsto z_0 + mz$ for some complex number $z_0$ is a Zariski closed subset of $\tilde W$ (this follows from the fact that the group operations on an elliptic curve are given by a regular map), and is thus either finite or all of $\tilde W$.  As $\tilde W$ is uncountable (when viewed over the complex numbers), we conclude that there is a single $m$ for which the above statement holds for all $w \in \tilde W$.  In particular, this implies that $\tilde Q(v,w)$ takes the form 
\begin{equation}\label{quod}
\tilde Q(v,w) = R(v) \oplus_{\tilde U} S(w)
\end{equation}
on the domain of definition of $\tilde Q$, where $R: \tilde V \to \tilde U$ corresponds to the map $z \mapsto mz$, and $S: \tilde W \to\tilde U$ is a map (which is necessarily regular, since $\tilde Q$ and $R$ are regular).  Again, we may assume that $R, S$ are non-constant (hence dominant), as otherwise $P$ depends only on one of $v,w$, contradicting our preceding hypothesis.

At this point we could use Theorem \ref{const} to conclude, but we instead give the following more direct argument.  We start using the hypothesis that $\tilde P: \mathbb{P}^1(\C) \times \mathbb{P}^1(\C) \to \mathbb{P}^1(\C)$ is not just a regular map, but is in fact (the completion of) a polynomial, which implies that $\tilde P(x,y)=\infty$ can only occur if $x=\infty$ or $y=\infty$.  Combining this with \eqref{pvw-2} and \eqref{quod}, we conclude that $R(v) \oplus_{\tilde U} S(w) \in \tilde h^{-1}(\{\infty\})$ can only occur if $v \in \tilde f^{-1}(\{\infty\})$ or $w \in \tilde g^{-1}(\{\infty\})$.  As $\tilde h, \tilde f, \tilde g, R, S$ are dominant, we thus conclude that there are finite subsets $A, B$ of $\tilde U$ such that the only pairs $(a, b) \in \tilde U \times \tilde U$ with $a \oplus_{\tilde U} b \in \tilde h^{-1}(\{\infty\})$ are those with $a \in A$ or $b \in B$.  But as $\tilde U$ is an infinite group, this is only possible of $\tilde h^{-1}(\{\infty\})$ is empty.  But $\tilde h: \tilde U \to \mathbb{P}^1(\C)$ is a projective morphism, hence has Zariski closed image (see e.g. \cite[Theorem I.5.2]{shaf}); since $\tilde h$ is non-constant, it is therefore surjective, a contradiction.  This concludes the treatment of the $g_{\tilde U}=1$ case.

The only remaining cases are when $g_{\tilde U}=0$, thus $\tilde U$ is isomorphic to $\mathbb{P}^1(\C)$ (see \cite[Example 1.3.5]{hart}), and so without loss of generality we can take $\tilde U = \mathbb{P}^1(\C)$.  In particular, $\tilde Q: (\tilde V \times \tilde W) \backslash \Sigma \to \mathbb{P}^1(\C)$ and $\tilde h: \mathbb{P}^1(\C) \to \mathbb{P}^1(\C)$ are now meromorphic functions on Riemann surfaces.

As $\tilde h$ is a non-constant meromorphic function, $\tilde h^{-1}(\{\infty\})$ must contain at least one point.
Suppose first that $\tilde h^{-1}(\{\infty\})$ contains at least two points, which after a M\"obius transformation we may normalise to be $0$ and $\infty$.  From \eqref{pvw-2} we see that if $(v,w) \in (\tilde V \times \tilde W) \backslash \Sigma$ is such that $\tilde Q(v,w)=0$ or $\tilde Q(v,w)=\infty$, then either $\tilde g(w)=\infty$ or $\tilde f(v)=\infty$.  Thus, for all but finitely many $w$, the meromorphic function that is (the completion of) $v \mapsto \tilde Q( v, w )$ has all of its zeroes and poles in $\tilde f^{-1}(\{\infty\})$.  The order of these zeroes and poles is easily seen to be bounded, so there are only finitely many possibilities for the \emph{divisor} of $v \mapsto \tilde Q( v, w )$ (the formal sum of the zeroes, minus the poles).  The set of $w$ for which a given divisor occurs is a constructible subset of $\tilde W$, so there must exist one divisor $D$ which is attained for all but finitely many $w$.  By Liouville's theorem, any two meromorphic functions with the same divisor must differ by a multiplicative constant, so we conclude that $\tilde Q(v,w) = R(v) S(w)$ for all but finitely many $v \in \tilde V$, all but finitely many $\tilde w \in W$, and some meromorphic functions $R: \tilde V \to \mathbb{P}^1(\C)$, $S: \tilde W \to \mathbb{P}^1(\C)$.

As before, we could use Theorem \ref{const} to conclude at this point, but we will again give a more direct argument. Now let $v, v' \in \tilde V$ be such that $\tilde f(v)=\tilde f(v')$.  Excluding finitely many exceptional pairs $(v,v')$ (including those for which one of $R(v), R(v')$ is zero or infinite), we conclude from the above discussion and \eqref{pvw-2} that
$$ \tilde h( R(v) S(w) ) = \tilde h( R(v') S(w) )$$
for all but finitely many $w$; as $S$ is dominant, this implies that
$$ \tilde h( (R(v)/R(v')) z ) = \tilde h( z )$$
for all but finitely many $z \in \C$, and hence for all $z \in \C$.  Consider the group $G$ of complex numbers $u$ such that $\tilde h(uz) = \tilde h(z)$ for all $z \in\C$.  As $\tilde h$ is non-constant, this is a finite subgroup of $\C^\times$ and is thus the $N^{\operatorname{th}}$ roots of unity for some natural number $N$.  Then we may write $\tilde h(z) = \tilde h'(z^N)$ for some regular map $\tilde h': \mathbb{P}^1(\C) \to \mathbb{P}^1(\C)$, so by replacing $\tilde h$ with $\tilde h'$ and $Q$ with $Q^N$ if necessary we may assume that $N=1$.  (Note that this cannot reduce the degree of $\tilde h$ to less than $2$, since $\tilde h$ will still map both $0$ and $\infty$ to $\infty$.)  We conclude that for all but finitely many pairs $(v,v')$, $\tilde f(v)=\tilde f(v')$ implies $R(v) = R(v')$.  Thus the irreducible projective variety $\{ (\tilde f(v), R(v))\mid v \in \mathbb{P}^1(\C) \}$ is a graph outside of a finite set of a function from $\mathbb{P}^1(\C) \to \mathbb{P}^1(\C)$; this function is continuous, generically holomorphic and blows up at most polynomially at any point, and is thus rational, so that $R = Y \circ \tilde f$ for some rational function $Y: \mathbb{P}^1(\C) \to \mathbb{P}^1(\C)$.  Similarly we may assume that $S = Z \circ \tilde g$ for some rational $Z: \mathbb{P}^1(\C) \to \mathbb{P}^1(\C)$.  But then from \eqref{pvw-2} and the dominance of $\tilde f, \tilde g$ we conclude that
$$ 
P(x, y) = \tilde h( Y(x) Z(y) )
$$
for all but finitely many pairs $(x,y) \in \C^2$.  Since $h^{-1}(\{\infty\})$ contains $0$, $\infty$, we see that $Y, Z$ cannot have any poles or zeroes on $\C$ and are thus constant, contradicting the non-constant nature of $P$.

The only remaining case is when $h^{-1}(\{\infty\})$ consists of a single point, which we may normalise to be $\infty$.  Then the meromorphic function $h$ maps $\C$ to $\C$ and is thus a polynomial.  We have now almost reached the conclusion in Theorem \ref{expand-thm-ass-nonst}(iii), with the main thing missing being that the curves $\tilde V, \tilde W$ are not yet affine, and the maps $f, g, Q$ are not yet polynomials.  To address this, we use a variant of the argument that treated the case when $\tilde h^{-1}(\{\infty\}$ contained more than one point.  Namely, we observe from \eqref{pvw-2} that for all but finitely many $w \in \tilde W$, the function $v \mapsto \tilde Q(v,w)$ is defined on all of $\tilde V$, and is finite outside of $\tilde f^{-1}(\{\infty\})$.  This function is a meromorphic function on $\tilde V$ of bounded degree.  Applying the Riemann-Roch\footnote{One does not need the full power of the Riemann-Roch theorem here; the finite dimensionality of this space can also be established from Laurent expansion at each pole together with Liouville's theorem.} theorem (see \cite[p. 245]{grif}), we conclude that there is a finite-dimensional vector space $S$ (over $\C$) of meromorphic functions that only have poles at $\tilde f^{-1}(\{\infty\})$, such that the functions $v \mapsto \tilde Q(v,w)$ lie in $S$ for all but finitely many $w \in \tilde W$.  Letting $e_1,\ldots,e_n: \tilde V \backslash \tilde f^{-1}(\{\infty\}) \to \C$ be a basis for this vector space, we thus have a representation of the form
\begin{equation}\label{lode}
\tilde Q(v,w) = \sum_{i=1}^n Z_i(w) e_i(v)
\end{equation}
for all but finitely many $w$, and all $v \in \tilde V \backslash \tilde f^{-1}(\{\infty\})$.  As $\tilde Q$ is a regular map and the $e_i$ are linearly independent, we conclude that $Z_i$ are regular maps from $W$ (with finitely many points deleted) to $\C$ and are thus meromorphic.

From \eqref{pvw-2}, we see that for all but finitely many $v$, we have $w \mapsto \tilde Q(v,w)$ defined on all of $\tilde W$, and finite outside of $\tilde g^{-1}(\{\infty\})$.  From this, \eqref{lode}, and the linear independence of the $e_i$ we see that the $Z_i$ are finite outside of $\tilde Q(v,w)$.  Also one can easily verify that the $Z_i$ have bounded degree.  Applying the Riemann-Roch theorem, we can find a linearly independent set of regular maps $e'_1,\ldots,e'_m: \tilde W \backslash \tilde g^{-1}(\{\infty\}) \to \C$, such that each $Z_i$ is a linear combination of the $e'_j$, thus we have
\begin{equation}\label{tq}
\tilde Q(v,w) = \sum_{i=1}^n \sum_{j=1}^m c_{ij} e_i(v) e'_j(w)
\end{equation}
for all $v$ outside of a finite subset of $\tilde V$, and all $w$ outside of a finite subset of $\tilde W$, and some complex coefficients $c_{ij}$.

Now let $V \subset \C^{1+n}, W \subset \C^{1+m}$ be the sets
$$ V := \{ (f(v), e_1(v),\ldots,e_n(v))\mid v \in \tilde V \backslash \tilde f^{-1}(\{\infty\}) \}$$
and
$$ W := \{ (g(v), e'_1(w),\ldots,e'_m(w))\mid w \in \tilde W \backslash \tilde g^{-1}(\{\infty\}) \}.$$
These are the irreducible projective curves $\{ (f(v),e_1(v),\ldots,e_n(v))\mid v \in \tilde V \}$, $\{ (g(w),e'_1(w),\ldots,e'_m(w))\mid w \in \tilde W \}$ with a finite number of points deleted, and so are irreducible quasiprojective curves.  If we define the polynomial maps $f': \C^{1+n} \to \C$, $g': \C^{1+m} \to \C$, $Q': \C^{1+n} \times \C^{1+m} \to \C$ by
\begin{align*}
f'( z_0, z_1,\ldots,z_n ) &:= z_0 \\
g'( w_0, w_1,\ldots,w_m ) &:= w_0 \\
Q'( (z_0,\ldots,z_n), (w_0,\ldots,w_m) ) &:= \sum_{i=1}^n \sum_{j=1}^m c_{ij} z_i w_j
\end{align*}
then $f', g'$ are non-constant on $V, W$, and we see from \eqref{pvw}, \eqref{tq} that
$$ P( f'(v), g'(w) ) = j(Q'(v,w) )$$
for all $v$ outside of a finite subset of $V$, and all $w$ outside of a finite subset of $W$.  As all functions involved here are polynomials, we may pass to the affine Zariski closures $\overline{V} \subset \C^{1+n}$, $\overline{W} \subset \C^{1+m}$, which are affine irreducible curves, and conclude that the above identity in fact holds for \emph{all} $v \in \overline{V}$ and $w \in \overline{W}$.  This gives Theorem \ref{expand-thm-ass-nonst}(iii) except for the fact that all polynomials and curves here are defined over $\C$ rather than $k$; but as $P$ was already defined over $k$, we may use the nullstellensatz (or the Lefschetz principle) as in the proof of Theorem \ref{const-0} to locate a choice of $\overline{V}, \overline{W}, f', g', Q'$ obeying the above properties and also defined over the algebraically closed subfield $k$ of $\C$.  This (finally!) concludes the proof of Theorem \ref{expand-thm-ass-nonst}.

\appendix 
\section{The \'etale fundamental group}\label{etale-app}

Throughout this appendix, $k$ is an algebraically closed field of characteristic zero.  (Some portion of the discussion below can be generalised to the positive characteristic setting, but several additional technical complications arise in that case which we do not wish to discuss here.)

Given any smooth irreducible variety $W$ over $k$, and a point $p$ in $W$, one can define an \emph{\'etale fundamental group} $\pi_1(W,p)$ of $W$ at $p$.  The precise construction of this group is not particularly relevant for this discussion, but one can for instance define $\pi_1(W,p)$ as the group of automorphisms of the fibre functor $\phi \mapsto \phi^{-1}(\{p\})$ that maps finite \'etale covers\footnote{A minor technical point here: in the definition of the \'etale fundamental group of a general locally connected Noetherian scheme, one needs to consider covers that are also as general as a locally connected Noetherian scheme.  However, it is known that any finite \'etale cover of a quasiprojective variety is again a quasiprojective variety \cite[Proposition 5.3.2]{ega}, and so one can work entirely within the category of quasiprojective varieties here.  We thank Antoine Chambert-Loir for pointing out this subtlety.} $\phi: V \to W$ of $W$ to their fibre $\phi^{-1}(\{p\})$ (and is thus a functor from the category of finite \'etale covers of $W$ to the category of sets).  See \cite[Chapter V]{sga} or \cite[Chapter 5]{szam} for details of this construction.  (One can also define the \'etale fundamental group for more general connected, locally noetherian schemes, but we will not need this additional level of generality here.)  

In this appendix, we will list some key properties of the \'etale fundamental group that are well established in the literature (such as \cite{sga}), which we will use as ``black boxes''.  Under the running hypothesis that $k$ is algebraically closed and has characteristic zero, these properties are largely analogous to properties of the (profinite completion of the) topological fundamental group of a complex variety.  (The situation is more complicated in positive characteristic, due to the existence of Artin-Schreier coverings, but we will not need to deal with these difficulties here.)

An important property of the \'etale fundamental group for us will be topological finite generation:

\begin{proposition}[Topological finite generation]\label{tpg}  Let $W$ be a smooth variety over an algebraically closed field $k$ of characteristic zero, and let $p$ be a point in $W$.  Then \'etale fundamental group $\pi_1(W,p)$ is a profinite group which is topologically finitely generated, that is to say there is a finitely generated subgroup which is dense in the profinite topology.
\end{proposition}

\begin{proof} The profiniteness of the \'etale fundamental group follows from construction (see \cite[\S V.7]{sga} or \cite[Theorem 5.4.2]{szam}).  Topological finite generation is established in \cite[Theorem II.2.3.1]{sga71} (and can also be deduced from Theorem \ref{equate} below).
\end{proof}   

Given any finite \'etale covering $\phi: V \to W$ of $W$ by a smooth variety $V$, the fibre $\phi^{-1}(\{p\})$ is automatically a finite set, and by construction, the fundamental group $\pi_1(W,p)$ acts on this set, thus each group element $g \in \pi_1(W,p)$ will map a point $v$ in $\phi^{-1}(\{p\})$ to a point $gv$ in $\phi^{-1}(\{p\})$.  This action is natural in the following sense: if one has two finite \'etale coverings $\phi: V \to W$ and $\phi': V' \to W$ and a regular map $f: V \to V'$ with $\phi = \phi' \circ f$, then the actions of $\pi_1(W,p)$ on $\phi^{-1}(\{p\})$ and $(\phi')^{-1}(\{p\})$ are intertwined by $f$ in the sense that
$$ g f(v) = f(gv)$$
for all $g \in \pi_1(W,p)$ and $v \in \phi^{-1}(\{p\})$; see \cite[\S V.7]{sga} or \cite[Theorem 5.4.2]{szam}.  A particular consequence of importance to us occurs when one has two finite \'etale coverings $\phi_1: V_1 \to W$, $\phi_2: V_2 \to W$ by smooth varieties $V_1,V_2$.  Then we may form the the fibre product $\phi_1 \times_W \phi_2: V_1 \times_W V_2 \to W$, where
$$ V_1 \times_W V_2 := \{ (v_1,v_2) \in V_1 \times V_2\mid \phi_1(v_1)= \phi_2(v_2) \}$$
and
$$ \phi_1 \times_V \phi_2(v_1,v_2) = \phi_1(v_1) = \phi_2(v_2)$$
then $\phi_1 \times_W \phi_2$ is easily seen to also be a finite \'etale covering, and the action of $\pi_1(W,p)$ on the product fibre $(\phi_1 \times_W \phi_2)^{-1}(\{p\}) = \phi_1^{-1}(\{p\}) \times \phi_2^{-1}(\{p\})$ is the direct sum of the action of $\pi_1(W,p)$ on the individual fibres $\phi_1^{-1}(\{p\}), \phi_2^{-1}(\{p\})$, thus
$$ g (v_1,v_2) = (gv_1, gv_2)$$
for all $g \in \pi_1(W,p)$, $v_1 \in \phi_1^{-1}(\{p\})$, and $v_2 \in \phi_2^{-1}(\{p\})$.

Givne a finite \'etale covering $\phi: V \to W$ by a smooth variety $V$, we can use the \'etale fundamental group $\pi_1(W,p)$ to relate the connected components of $V$ with the fibre $\phi^{-1}(\{p\})$ (and, for the purposes of this paper, this is the main reason why need the \'etale fundamental group in the first place).  Indeed, $V$ is connected if and only if the action of $\pi_1(W,p)$ on $\phi^{-1}(\{p\})$ is transitive (see \cite[\S V.7]{sga} or \cite[Theorem 5.4.2]{szam}); from this and functoriality, we see that in the more general case when $V$ is not necessarily connected, the orbits of $\pi_1(W,p)$ on $\phi^{-1}(\{p\})$ are nothing more than the fibres over $p$ of the connected components of $V$.

Let $W$ be a smooth irreducible variety, let $U$ be a smooth irreducible subvariety of $W$, and let $p$ be a point in $U$.  Then there is a canonical homomorphism $\eta$ from $\pi_1(U,p)$ to $\pi_1(W,p)$, which is compatible with the actions of these groups on fibres in the following sense: if $\phi: V \to W$ is a finite \'etale covering of smooth varieties, and $\phi': \phi^{-1}(U) \to U$ is the restriction of $\phi$ to $\phi^{-1}(U)$, then $\phi'$ is also a finite \'etale covering, and one has
$$ g v = \eta(g) v$$
for all $g \in \pi_1(U,p)$ and all $v \in \phi^{-1}(\{p\}) = (\phi')^{-1}(\{p\})$; see \cite[\S V.7]{sga}.  

\begin{lemma}[Insensitivity to high codimension sets]\label{highd}  Let $W$ be a smooth irreducible variety, let $S$ be a closed subvariety of $W$, and let $p \in W \backslash S$, and let $\eta: \pi_1(W \backslash S,p) \to \pi_1(W,p)$ be the homomorphism described above.
\begin{itemize}
\item If $S$ has codimension at least $1$ in $W$, then $\eta$ is surjective.
\item If $S$ has codimension at least $2$ in $W$, then $\eta$ is an isomorphism.
\end{itemize}
\end{lemma}

\begin{proof}  See \cite[Corollary V.5.6]{sga} for the first part, and \cite[Corollary X.3.3]{sga} for the second part.  (This second part can also be deduced from the Zariski-Nagata purity theorem \cite[Theorem X.3.4]{sga2}.)
\end{proof}

If $W$ is a smooth irreducible variety over $k$, and $p$ is a point in $W$, then (because $k$ has characteristic zero) one can find an algebraically closed subfield $k'$ of finite transcendence degree over $\Q$ over which $W$ and $p$ are still defined (by taking the algebraic closure of the coefficients of the polynomials that cut out $W$, as well as the coefficients of $p$).  As such, there exists an embedding $\tau: k' \to \C$ of the field $k'$ to the complex numbers $\C$; note that in general that this embedding will not be unique.  Given such an embedding $\tau$, we can then define a complex variety $\tau(W)$ by applying $\tau$ to all the coefficients of the polynomials defining $W$.  Similarly, by applying $\tau$ to $p$ we have a point $\tau(p)$ in $\tau(W)$.  As $W$ was smooth, we see that $\tau(W)$ is smooth and irreducible in the algebraic sense, which implies by the inverse function theorem that $\tau(W)$ is smooth and connected in the analytic sense.  We can then form the ordinary (i.e, topological) fundamental group $\pi_1^{\operatorname{top}}(\tau(W), \tau(p))$.  This group is, in general, not profinite; however, we can form the profinite completion $\pi_1^{\operatorname{top}}(\tau(W), \tau(p))^\wedge$, defined as the Hausdorff completion (i.e. inverse limit) of all the finite quotients of $\pi_1^{\operatorname{top}}(\tau(W), \tau(p))$.  We have the following deep theorem:

\begin{theorem}[Equivalence of \'etale and topological fundamental groups]\label{equate}  Let $W$ be a smooth variety over an algebraically closed field of characteristic zero, let $p$ be a point in $W$, and let $k'$ be a field of finite transcendence degree over $\Q$, such that $W$ and $p$ are defined over $k'$.  Let $\tau: k'\to \C$ be an embedding of the field $k'$ in $\C$.  Then there is a canonical identification
$$ \pi_1(W,p) \equiv \pi_1^{\operatorname{top}}(\tau(W), \tau(p))^\wedge$$
between the \'etale fundamental group $\pi_1(W,p)$ and the profinite completion of the topological fundamental group $\pi_1^{\operatorname{top}}(\tau(W), \tau(p))$.  Furthermore one has the following functorial property: if $S$ is an closed subvariety of $W$ defined over $k'$ that has codimension at least $1$ and avoids $p$, then the above identifications intertwine the canonical homomorphism 
$$
\pi_1(W \backslash S,p) \to \pi_1(W,p)$$
of \'etale fundamental groups described previously, and the profinite completion
$$ 
\pi_1^{\operatorname{top}}(\tau(W \backslash S), \tau(p))^\wedge \to \pi_1^{\operatorname{top}}(\tau(W), \tau(p))^\wedge
$$
of the obvious homomorphism from the topological fundamental group $\pi_1^{\operatorname{top}}(\tau(W \backslash S), \tau(p))$ to the topological fundamental group $\pi_1^{\operatorname{top}}(\tau(W), \tau(p))$.  (It is easy to see that this profinite completion exists and is well defined).
\end{theorem}

\begin{proof} The first part of the theorem is \cite[Corollary XII.5.2]{sga}, which follows directly from the Riemann existence theorem (\cite[Theorem XII.5.1]{sga}), which provides an equivalence of categories between finite analytic covering spaces of $\tau(W)$ and finite \'etale coverings of $W$, with $\pi(W,p)$ being described completely by its action on fibres over $p$ in the latter category, and $\pi_1^{\operatorname{top}}(\tau(W), \tau(p))^\wedge$ by its action on fibres over $\tau(p)$ in the former category.  The second part of the theorem follows from the obvious fact that these equivalences of categories for $W$ and for $W \backslash S$ are intertwined by the restriction maps from $W$ to $W \backslash S$.  (More generally, one can replace the inclusion map from $W \backslash S$ to $W$ by other smooth maps, but we will not need to do so here.)
\end{proof}

\begin{remark} Among other things, this shows that up to isomorphism, the profinite completion of the topological fundamental group $\pi_1^{\operatorname{top}}(\tau(W), \tau(p))^\wedge$ is independent of $\tau$.  However, the topological fundamental group itself can be sensitive to the choice of $\tau$; see \cite{serre}.
\end{remark}

We will use the above equivalence to establish the following result which is crucial for our analysis:

\begin{theorem}[Weak \'etale van Kampen theorem]\label{evk}  
Let $W$ be a smooth variety defined over an algebraically closed field of characteristic zero, and let $W_1,W_2$ be closed subvarieties of $W$ of strictly smaller dimension.  Let $p$ be a point in $W \backslash (W_1 \cup W_2)$.  By Lemma \ref{highd}, we have canonical surjective homomorphisms from $\pi_1(W \backslash (W_1 \cup W_2),p)$ to $\pi_1(W \backslash W_1)$ and $\pi_1(W \cap W_2)$, as well as canonical surjective homomorphisms from $\pi_1(W \backslash W_1,p)$ and $\pi_1(W \backslash W_2,p)$ to $\pi_1(W)$.  

Then $\pi_1(W \backslash (W_1 \cup W_2),p)$ surjects onto the fibre product $\pi_1(W \backslash W_1,p) \oplus_{\pi_1(W,p)} \pi_1(W \backslash W_2,p)$.
\end{theorem}

\begin{proof}  By Lemma \ref{highd}, we may remove $W_1 \cap W_2$ from $W$ without affecting any of the fundamental groups, so we may assume that $W_1$ and $W_2$ are disjoint.

Let $k'$ be an algebraically closed field of finite transcendence degree over $\Q$, over which $W, W_1, W_2, p$ are all defined, and let $\tau: k' \to \C$ be an embedding of fields. By Theorem \ref{equate}, it suffices to show that $G_{12}^\wedge$ surjects onto the fibre product $G_1^\wedge \times_{G^\wedge} G_2^\wedge$, where $G,G_1,G_2,G_{12}$ are the topological fundamental groups
\begin{align*}
G &:= \pi_1(\tau(W),\tau(p)) \\
G_1 &:= \pi_1(\tau(W) \backslash \tau(W_1),\tau(p)) \\
G_2 &:= \pi_1(\tau(W) \backslash \tau(W_2),\tau(p)) \\
G_{12} &:= \pi_1(\tau(W) \backslash (\tau(W_1) \cup \tau(W_2),\tau(p)).
\end{align*}
On the other hand, by the topological van Kampen theorem (see e.g. \cite[\S 1.2]{hatcher}), $G_{12}$ can be canonically identified with the amalgamated free product of $G_1$ and $G_2$ over $G$; in particular, $G_{12}$ surjects onto $G_1 \times_G G_2$, which implies that the image of $G_{12}^\wedge$ in $G_1^\wedge \times_{G^\wedge} G_2^\wedge$ is dense in the relative product of profinite topologies.  On the other hand, as profinite groups are compact Hausdorff with respect to the profinite topology, this image must also be compact, and the surjectivity follows.
\end{proof}

\begin{remark} Actually, the above theorem is valid in arbitrary characteristic, and furthermore $\pi_1(W \backslash (W_1 \cup W_2),p)$ is the coproduct (in the category of profinite groups) of $\pi_1(W \backslash W_1,p)$ and $\pi_1(W \backslash W_2,p)$ over $\pi_1(W,p)$; see \cite{zoon}.  This result can also be deduced from \cite[Theorem IX.5.1]{sga}; ultimately, it is equivalent to the fact that finite \'etale covers over $W \backslash W_1$ and $W \backslash W_2$ which are isomorphic on $W \backslash (W_1 \cup W_2)$ can  be glued together to form a finite \'etale cover of $W$ if $W_1,W_2$ are disjoint.   However, we will not need this stronger version of the \'etale van Kampen theorem here.
\end{remark}


\begin{thebibliography}{10}

\bibitem{adams}
W. Adams, P. Loustaunau. An Introduction to Gr\"obner Bases. AMS, Providence, 1994.
	
\bibitem{ax}
J. Ax, \emph{The elementary theory of finite fields}, Ann. Math. \textbf{88} (1968), 239--271.



\bibitem{borg}
J. Bourgain, \emph{More on the sum-product phenomenon in prime fields and its applications}, Int. J. Number Theory \textbf{1} (2005), no. 1, 1–-32. 

\bibitem{bgt-product}
E. Breuillard, B. Green, T. Tao, \emph{Approximate subgroups of linear groups},   Geom. Funct. Anal. \textbf{21} (2011), no. 4, 774–-819.

\bibitem{bgt}
E. Breuillard, B. Green, T. Tao, \emph{The structure of approximate groups}, preprint, {\tt arXiv:1110.5008}.

\bibitem{bukh}
B. Bukh, \emph{Sums of dilates}, Combin. Probab. Comput. 17 (2008), no. 5, 627–-639. 

\bibitem{bt}
B. Bukh, J. Tsimerman, \emph{Sum-product estimates for rational functions}, Proc. Lond. Math. Soc. (3) \textbf{104} (2012), no. 1, 1–26. 

\bibitem{cherlin}
G. Cherlin, L. van den Dries, A. Macintyre, \emph{Decidability and undecidability theorems for
PAC-fields}, Bull. Amer. Math. Soc. (N.S.) \textbf{4} (1981), no. 1, 101–-104.

\bibitem{cher2}
Z. Chatzidakis, L. van den Dries, A. Macintyre, \emph{Definable sets over finite fields}, J. Reine angew.
Math. 427 (1992), 107–-135

\bibitem{chung-hyper}
F. Chung, \emph{Regularity lemmas for hypergraphs and quasi-randomness}, Random Structures Algorithms \textbf{2} (1991), no. 2, 241–-252. 

\bibitem{cill}
J. Cilleruelo, Y. Hamidoune, O. Serra, \emph{On sums of dilates}, Combin. Probab. Comput. \textbf{18} (2009), no. 6, 871–-880.

\bibitem{chikr}
D. Covert, D. Hart, A. Iosevich, D. Koh, M. Rudnev, \emph{Generalized incidence theorems, homogeneous forms, and sum-product estimates in finite fields}, European J. Combin. \textbf{31} (2010), no. 1, 306–-319.

\bibitem{cox}
D. Cox, J. Little, and D. O'Shea. Ideals, Varieties, and Algorithms: An Introduction to Computational Algebraic Geometry and Commutative Algebra. Springer, New York, 1992.

\bibitem{dave}
H. Davenport, A. Schinzel, \emph{Two problems concerning polynomials}, J. Reine Angew. Math. \textbf{214/215} (1964), 386-–391.

\bibitem{def}
M. De Franchis, \emph{Un teorema sulle involuzioni irrazionali}, Rend. Circ. Mat Palermo \textbf{36} (1913), 368

\bibitem{lipton}
R. DeMillo, R. Lipton, \emph{A probabilistic remark on algebraic program testing}, Information Processing Letters \textbf{7} (1978), 192--195.

\bibitem{elekes-ronyai}
G. Elekes, L. R\'onyai, \emph{A combinatorial problem on polynomials and rational functions}, J. Combin. Theory, (A), \textbf{89} (2000), 1-–20.

\bibitem{elekes-szabo}
G. Elekes, E. Szab\'o, \emph{How to find groups? (And how to use them in Erd\H{o}s geometry?)}, Combinatorica, 2012.

\bibitem{eot}
J. Ellenberg, R. Oberlin, T. Tao, \emph{The Kakeya set and maximal conjectures for algebraic varieties over finite fields}, to appear, Mathematika.

\bibitem{frankl}
P. Frankl, V. R\"odl, \emph{The uniformity lemma for hypergraphs}, Graphs Combin. \textbf{8} (1992), no. 4, 309–-312.

\bibitem{fried}
M. Fried, G. Sacerdote, \emph{Solving diophantine problems over all residue class fields of a number field and all finite fields}, Ann. Math. \textbf{100} (1976), 203--233.

\bibitem{garaev}
M. Garaev, C.-Y. Shen, \emph{On the size of the set $A(A+1)$}, Math. Z. \textbf{263} (2009), no. 94.

\bibitem{godel}
K. G\"odel, \emph{Consistency of the axiom of choice and of the generalized continuum-hypothesis with the axioms of set
theory}, Proc. Nat. Acad. Sci, \textbf{24} (1938), 556–-557.


\bibitem{gowers-sz}
T. Gowers, \emph{Lower bounds of tower type for Szemer\'edi's uniformity lemma}, Geom. Func. Anal. \textbf{7} (1997), 322--337.

\bibitem{gowers-4aps} W.~T.~Gowers, \emph{A new proof of Szemer\'edi's theorem for progressions of length four,} GAFA \textbf{8} (1998), no. 3, 529--551.

\bibitem{gowers-hyper}
T. Gowers, {Hypergraph regularity and the multidimensional Szemer\'edi theorem}, Ann. of Math. (2) \textbf{166} (2007), no. 3, 897–-946.

\bibitem{grif}
P. Griffiths, J. Harris, Principles of algebraic geometry. Reprint of the 1978 original. Wiley Classics Library. John Wiley \& Sons, Inc., New York, 1994. 

\bibitem{ega}
A. Grothendieck, J. Dieudonn\'e, \emph{\'El\'ements de g\'eom\'etrie alg\'ebrique (r\'edig\'es avec la collaboration de Jean Dieudonné\'e) : II. \'Etude globale \'el\'ementaire de quelques classes de morphismes}, Publications Math\'ematiques de l'IH\'ES \textbf{8} (1961), 5-–222.

\bibitem{sga}
A. Grothendieck, M. Raynaud, S\'eminaire de G\'eom\'etrie Alg\'ebrique du Bois Marie - 1960-61 - Rev\^etements \'etales et groupe fondamental - (SGA 1) (Documents Math\'ematiques 3), Paris: Soci\'et\'e Math\'ematique de France, pp. xviii+327, (2003) [1971].

\bibitem{sga2}
A. Grothendieck, M. Raynaud, Laszlo, Yves, ed., Cohomologie locale des faisceaux coh\'erents et th\'eor\`emes de Lefschetz locaux et globaux (SGA 2), Documents Math\'ematiques (Paris), 4, Paris: Soci\'et\'e Math\'ematique de France, (2005) [1968].

\bibitem{sga71}
A. Grothendieck, S\'eminaire de G\'eom\'etrie Alg\'ebrique du Bois Marie - 1967-69 - Groupes de monodromie en g\'eom\'etrie alg\'ebrique - (SGA 7) - vol. 1. Lecture notes in mathematics. 288. Berlin; New York: Springer-Verlag.


\bibitem{gyar}
K. Gyarmati, A. S\'ark\H{o}zy, \emph{Equations in finite fields with restricted solution sets. II}, Acta Math. Hungar. 119 (2008), 259–-280.


\bibitem{harris}
J. Harris, Algebraic geometry, A first course.  Springer-Verlag, 1992.

\bibitem{hi}
D. Hart, A. Iosevich, \emph{Sums and products in finite fields: an integral geometric viewpoint}, Radon transforms, geometry, and wavelets, 129–135, Contemp. Math., 464, Amer. Math. Soc., Providence, RI, 2008.

\bibitem{hart}
R. Hartshorne, Algebraic geometry. Graduate Texts in Mathematics, No. 52. Springer-Verlag, New York-Heidelberg, 1977. 

\bibitem{hls}  
D. Hart, L. Li, C.-Y. Shen, \emph{Fourier analysis and expanding phenomena in finite fields}, Proc. Amer. Math. Soc. \textbf{141} (2013), no. 2, 461–-473.

\bibitem{hatcher}
A. Hatcher, Algebraic topology. Cambridge University Press, Cambridge, 2002.

\bibitem{heg}
N. Hegyv\'ari, F. Hennecart, \emph{Explicit constructions of extractors and expanders}, Acta Arith. \textbf{140} (2009), 233--249



\bibitem{ir}
A. Iosevich, M. Rudnev, \emph{Erd\H{o}s distance problem in vector spaces over finite fields}, Trans.
Amer. Math. Soc. 359 (2007), 6127–-6142.

\bibitem{kiefe}
C. Kiefe, \emph{Sets definable over finite fields: their zeta-functions}, Trans. Amer. Math. Soc. \textbf{223} (1976), 45–-59.

\bibitem{klei}
S. Kleiman, \emph{L. Bertini and his two fundamental theorems}, Studies in the history of modern mathematics, III. Rend. Circ. Mat. Palermo (2) Suppl. No. 55 (1998), 9–37.

\bibitem{laba}
S. Konyagin, I. Laba, \emph{Distance sets of well-distributed planar sets for polygonal norms}, Israel J. Math. \textbf{152} (2006), 157–-179. 

\bibitem{kow}
E. Kowalski, \emph{Exponential sums over definable subsets of finite fields}, Israel J. Math. \textbf{160} (2007), 219–-251. 

\bibitem{lang}
S. Lang, A. Weil, \emph{Number of points of varieties in finite fields}, Amer. J. Math. \textbf{76} (1954), 819–-827. 

\bibitem{mubayi}
J. Lenz, D. Mubayi, \emph{The Poset of Hypergraph Quasirandomness}, preprint, {\tt arXiv:1208.5978}.

\bibitem{reiter}
R. Lidl, H. Niederreiter, Finite Fields, second edition, Cambridge University Press, 1997.

\bibitem{shelah}
M. Malliaris, S. Shelah, \emph{Regularity lemmas for stable graphs}, preprint, {\tt arXiv:1102.3904}.

\bibitem{milne-etale}
J. Milne, Etale cohomology. Princeton Mathematical Series, 33. Princeton University Press, Princeton, N.J., 1980.

\bibitem{mumf}
D. Mumford, The red book of varieties and schemes. Second, expanded edition. Includes the Michigan lectures (1974) on curves and their Jacobians. With contributions by Enrico Arbarello. Lecture Notes in Mathematics, 1358. Springer-Verlag, Berlin, 1999.

\bibitem{plagne}
A. Plagne, \emph{Sums of dilates in groups of prime order}, Combin. Probab. Comput. \textbf{20} (2011), no. 6, 867–-873. 

\bibitem{rs}
V. R\"odl, M. Schacht, \emph{Regular partitions of hypergraphs: regularity lemmas}, Combin. Probab. Comput. \textbf{16} (2007), no. 6, 833-–885.

\bibitem{rodl}
V. R\"odl, J. Skokan, \emph{Regularity lemma for $k$-uniform hypergraphs}, to appear, Random Structures and Algorithms.

\bibitem{shin}
A. Schinzel, Polynomials with special regard to reducibility. With an appendix by Umberto Zannier. Encyclopedia of Mathematics and its Applications, 77. Cambridge University Press, Cambridge, 2000.

\bibitem{schwartz}
J. Schwartz, \emph{Fast probabilistic algorithms for verification of polynomial identities}, Journal of the ACM \textbf{27} (1980), 701-–717.

\bibitem{serre}
J. P. Serre, \emph{Exemples de vari\'et\'es projectives conjugu\'ees non hom\'eomorphes}, C. R. Acad. Sci. Paris \textbf{258} (1964),
4194–-4196.

\bibitem{shaf}
I. R. Shafarevich, Basic algebraic geometry I, Springer-Verlag (1977).

\bibitem{shk}
I. Shkredov, \emph{On monochromatic solutions of some nonlinear equations in $Z_p$}, Mat. Zametki \textbf{88} (2010), no. 4, 625--634.

\bibitem{soli}
J. Solymosi, \emph{Incidences and the spectra of graphs}, Building bridges, 499–513, Bolyai Soc. Math.
Stud., 19, Springer, Berlin, 2008.

\bibitem{solymosi}
J. Solymosi, \emph{Expanding polynomials over the rationals}, preprint, {\tt arXiv:1212.3365}.

\bibitem{szam}
T. Szamuely, \emph{Galois groups and fundamental groups}, Cambridge Studies in Advanced Mathematics, \textbf{117}, Cambridge University Press, Cambridge, 2009.

\bibitem{szemeredi-reg}
E. Szemer\'edi, \emph{Regular partitions of graphs}, in ``Problem\'es Combinatoires et Th\'eorie des Graphes, Proc. Colloque Inter. CNRS,'' (Bermond, Fournier, Las Vergnas, Sotteau, eds.), CNRS Paris, 1978, 399--401.

\bibitem{tamme}
G. Tamme, \emph{Teilk\"orper h\"oheren Gesclechts eines algebraischen Funktionenk\"orpers}, Arch. Math. (Basel), \textbf{23} (1972), 257--259.

\bibitem{tao-hyper}
T. Tao, \emph{A variant of the hypergraph removal lemma}, J. Combin. Theory Ser. A \textbf{113} (2006), no. 7, 1257–-1280.

\bibitem{taylor}
J. Taylor, Several complex variables with connections to algebraic geometry and Lie groups. Graduate
Studies in Mathematics, 46. American Mathematical Society, Providence, RI, 2002.

\bibitem{vinh}
L. A. Vinh, \emph{Solvability of systems of bilinear equations over finite fields}, Proc. Amer. Math. Soc. \textbf{137} (2009), no. 9, 2889–-2898.

\bibitem{vu}
V. Vu, \emph{Sum-product estimates via directed expanders}, Math. Res. Lett. \textbf{15} (2008), no. 2, 375–-388.

\bibitem{wan}
D. Wan, \emph{Generators and irreducible polynomials over finite fields}, Math. Comp. \textbf{66} (1997), no. 219, 1195–-1212. 

\bibitem{weil}
A. Weil, \emph{Numbers of solutions of equations in finite fields},  Bull. Amer. Math. Soc. \textbf{55}, (1949). 497–-508. 

\bibitem{zieve}
Michael E. Zieve, Peter Mueller, \emph{On Ritt's polynomial decomposition theorems}, {\tt arXiv:0807.3578}

\bibitem{zippel}
R. Zippel, \emph{Probabilistic algorithms for sparse polynomials},  Symbolic and Algebraic Computation, Lecture Notes in Computer Science \textbf{72} (1979), 216--226.

\bibitem{zoon}
V. Zoonekynd, \emph{Th\'eor\`eme de van Kampen pour les champs alg\'ebriques}, Ann. Math. Blaise Pascal \textbf{9} (2002), no. 1, 101–-145. 

\end{thebibliography}
\end{document}